\documentclass[a4paper,11pt]{amsart}
\usepackage[all]{xy}
\usepackage{graphicx,amsmath,amssymb,amsthm,amscd,mathrsfs}

\DeclareMathOperator{\supp}{supp}

\DeclareMathOperator{\ord}{ord}

\newtheorem{thm}{Theorem}[section]
\newtheorem{df}[thm]{Definition}
\newtheorem{prop}[thm]{Proposition}
\newtheorem{lem}[thm]{Lemma}
\newtheorem{cor}[thm]{Corollary}
\newtheorem{rem}[thm]{Remark}

\newtheorem{exa}[thm]{Example}

\begin{document}
\title[Linear differential equations and representations of quivers]{Linear differential equations on the Riemann sphere 
	and representations of quivers}
\date{}
\author{Kazuki Hiroe}
\email{kazuki@josai.ac.jp}
\address{Department of Mathematics, Josai University,\\ 
1-1 Keyakidai Sakado-shi Saitama 350-0295 JAPAN.}

\maketitle
\begin{abstract}
	Our interest in this paper  is 
  a generalization of the additive Deligne-Simpson
  problem which is originally defined for Fuchsian differential
  equations on the Riemann sphere.
  We shall extend this problem to differential equations having 
  an arbitrary number of unramified 
  irregular singular points and 
  determine 
  the existence of solutions of the 
  generalized 
  additive Deligne-Simpson problems.
  Moreover we apply this result to the geometry of the moduli
  spaces of stable meromorphic connections of trivial bundles
  on the Riemann sphere. Namely, open embedding of the moduli spaces 
  into quiver varieties is given and the non-emptiness condition
  of the moduli spaces is determined. Furthermore the connectedness of 
  the moduli spaces is shown.
\end{abstract}

\section*{Introduction}
The additive Deligne-Simpson problem asks the existence of 
an irreducible Fuchsian differential equation on the Riemann sphere
with 
prescribed local data.
In this paper we shall consider analogous 
problems for differential
equations on the Riemann sphere with unramified irregular singular points.

First of all let us recall the classical case, i.e.,
the additive Deligne-Simpson problem 
for systems of linear Fuchsian differential equations (see \cite{Kos}).
A system of first order linear differential 
equations is called {\em Fuchsian} if it is 
of the form
\[
    \frac{d}{dz}Y=\sum_{i=1}^{p}\frac{A_{i}}{z-a_{i}}Y\quad
    (A_{i}\in M(n,\mathbb{C}),\,i=1,\ldots,p).
\]
Here we call each $A_{i}$ the {\em residue matrix} 
at the singular point $a_{i}$.
Also $A_{0}:=-\sum_{i=1}^{p}A_{i}$ is called 
the residue matrix at $\infty$.
We say that $\frac{d}{dz}Y
=\sum_{i=1}^{p}\frac{A_{i}}{z-a_{i}}Y$
is \textit{irreducible}
if $A_{0},\ldots,A_{p}$ have no nontrivial simultaneous invariant
vector subspace of $\mathbb{C}^{n}$, i.e., if there exists
$W\subsetneqq \mathbb{C}^{n}$ such that 
$A_{i}W\subset W$  for all $i=0,\ldots,p$, then $W=\{0\}$.

\begin{df}[additive Deligne-Simpson problem (classical case)]\normalfont
	The \textit{additive Deligne-Simpson problem} consists of 
	 points $a_{1},\ldots,a_{p}$ in $\mathbb{C}$ and 
	conjugacy classes $C_{0},C_{1},\ldots,C_{p}$
	in $M(n,\mathbb{C})$.
	A \textit{solution} of the problem 
	is an {\em irreducible} Fuchsian differential equation
	\[
		\frac{dY}{dz}=\sum_{i=1}^{p}\frac{A_{i}}{z-a_{i}}Y\quad
		(A_{i}\in M(n,\mathbb{C}),\,i=1,\ldots,p)	
	\]
	whose residue matrices $A_{i}\in C_{i}$ for 
	$i=0,\ldots,p$.
\end{df}

This problem is developed by V. Kostov as an analogy of the problem
studied by P. Deligne and C. Simpson (see \cite{S}), 
so called multiplicative Deligne-Simpson problem.
After Kostov who gave an necessary and sufficient condition
of the existence of solutions under generic conditions in \cite{Kos},
a complete necessary and sufficient condition
was given by W. Crawley-Boevey \cite{C}.

As a generalization of this problem, it seems to be natural to 
consider  similar problems for non-Fuchsian equations (see for example
\cite{Boarx},\cite{Bo}, \cite{Kos2}). 
Our generalization in this paper is an extension of the work of P. Boalch in 
\cite{Boarx} and Kostov's way of generalization given in \cite{Kos2} is somewhat 
different from ours.
Before formulating the generalized problem precisely,
let us recall some facts
of local formal theory around irregular singular points 
of differential equations.
Let $M(n,R)$ be the ring of $n\times n$ matrices with 
coefficients in a commutative ring $R$ and $\mathrm{GL}(n,R)$ the group 
which consists of all invertible elements in $M(n,R)$ of the multiplication.
Set $\mathbb{C}(\!(z)\!):=
\{\sum_{i=r}^{\infty}c_{i}z^{i}\mid c_{i}\in\mathbb{C},\,
r\in \mathbb{Z}\}$
and  
$\mathbb{C}[\![z]\!]:=
\{\sum_{i=0}^{\infty}c_{i}z^{i}\mid c_{i}\in \mathbb{C}\}$. 
Let us attach to
$C=\sum_{i=r}^{\infty}c_{i}z^{i}\in M(n,\mathbb{C}(\!(z)\!))$
the integer  $\mathrm{ord}(C):=\min\{i\mid c_{i}\neq 0\}$ called
the 
{\em order}.
For $A\in M(n,\mathbb{C}(\!(z)\!))$ and 
$X\in \mathrm{GL}(n,\mathbb{C}(\!(z)\!))$
the {\em gauge transformation} of $A$ by $X$ is 
\[
    X[A]:=XAX^{-1}+\left(\frac{d}{dz}X\right)X^{-1}.
\]

\begin{df}[Hukuhara-Turrittin-Levelt normal form]
    \normalfont
    If an element $B\in M(n,\mathbb{C}(\!(z)\!))$ is of the form
\[B=\mathrm{diag}(q_{1}(z^{-1})I_{n_1}+R_{1}z^{-1},\ldots,
q_{m}(z^{-1})I_{n_{m}}+R_{m}z^{-1})
\]
with 
        $q_{i}(s)\in s^{2}\mathbb{C}[s]$ satisfying
    $q_{i}\neq q_{j}$ if $i\neq j$ and 
    $R_{i}\in M(n_{i},\mathbb{C})$,
    then $B$ is called the {\em Hukuhara-Turrittin-Levelt normal form}
    or the {\em HTL normal form} shortly.
    Here $I_{m}$ is the identity matrix of $M(m,\mathbb{C})$.
\end{df}
Now let us introduce truncated orbits which play the same role as 
the conjugacy classes $C_{i}$ of residue 
matrices in the 
Deligne-Simpson problem of Fuchsian systems.
Let us define
$G_{k}:=\mathrm{GL}(n,\mathbb{C}[\![z]\!]/z^{k}\mathbb{C}[\![z]\!])$,
$k\ge 1$, which
can be identified with 
\[
	\left\{A_{0}+A_{1}x+\cdots+A_{k-1}z^{k-1}\,
    	\bigg|\,
	\begin{array}{l}
		A_{0}\in \mathrm{GL}(n,\mathbb{C}),\,A_{i}\in 
    	M(n,\mathbb{C}),\\
	i=1,\ldots,k-1 
	\end{array}
	\right\}.
\]
Also define 
\begin{align*}
	\mathfrak{g}_{k}&:=M(n,\mathbb{C}[\![z]\!]/z^{k}\mathbb{C}[\![z]\!])\\
        &=
    	\left\{A_{0}+A_{1}z+\cdots+A_{k-1}z^{k-1}\,\big|\,
    	A_{i}\in 
    	M(n,\mathbb{C}),\,i=0,\ldots,k-1 \right\}.
\end{align*}
The dual vector space $\mathfrak{g}_{k}^{*}:=\mathrm{Hom}_{\mathbb{C}}(
\mathfrak{g}_{k},\mathbb{C})$ is identified with 
$$M(n,z^{-k}\mathbb{C}[\![z]\!]/\mathbb{C}[\![z]\!])=
\left\{\frac{A_{k}}{z^{k}}+\cdots+\frac{A_{1}}{z}
\,\Big|\, A_{i}\in M(n,\mathbb{C})\right\}$$ 
by the nondegenerate bilinear form 
$\mathfrak{g}_{k}\times \mathfrak{g}_{k}^{*}\ni (A,B)\mapsto 
\mathrm{Res}(\mathrm{tr}(AB))\in \mathbb{C}.$
Here we set $\mathrm{Res}(\sum_{i=r}^{\infty}A_{i}z^{i}):=A_{-1}$.

Then an HTL normal form $B$
with $\mathrm{ord}(B)\ge -k$ can be seen as an element in $\mathfrak{g}^{*}_{k}$. 
Thus we can consider the $G_{k}$-orbit
$\mathcal{O}_{B}
:=
\left\{
gBg^{-1}\in \mathfrak{g}^{*}_{k}\mid g\in G_{k}
\right\}$
of $B$ in $\mathfrak{g}_{k}^{*}$ called the \textit{truncated orbit}
of $B$. 
Now let us formulate a generalization of Deligne-Simpson problem.
\begin{df}[generalized additive Deligne-Simpson problem]
\normalfont
The {\em generalized additive Deligne-Simpson problem}
	consists of a collection of points $a_{1},\ldots,a_{p}$ in $\mathbb{C}$,
	of nonzero positive integers $k_{0},\ldots,k_{p}$
	and of 	HTL normal forms 
        $B_{i}\in\mathfrak{g}^{*}_{k_{i}}\subset 
        M(n,\mathbb{C}(\!(z)\!))$  for $i=0,\ldots,p$.  
	A \textit{solution} of the generalized additive Deligne-Simpson 
	problem is an {\em irreducible} differential equation
        \[
            \frac{d}{dz}Y=\left(\sum_{i=1}^{p}\sum_{j=1}^{k_{i}}
            \frac{A_{i,j}}{(z-a_{i})^{j}}
	    +\sum_{2\le j\le k_{0}}A_{0,j}z^{j-2}\right)Y
        \]
        satisfying that $A^{(i)}(z)\in \mathcal{O}_{B_{i}}$
        for $i=0,\ldots,p$.
        Here $A^{(i)}(z):=\sum_{j=1}^{k_{i}}A_{i,j}z^{-j}$ for 
    	$i=0,\ldots,p$ and $A_{0,1}:=-\sum_{i=1}^{p}A_{i,1}$.
    	We say $$\frac{d}{dz}Y=\left(\sum_{i=1}^{p}\sum_{j=1}^{k_{i}}
        \frac{A_{i,j}}{(z-a_{i})^{j}}
	+\sum_{2\le j\le k_{0}}A_{0,j}x^{j-2}\right)Y$$
        is {\em irreducible}
        if the collection 
        $(A_{i,j})_{\substack{0\le i\le p\\1\le j\le k_{i}}}$
        of coefficient matrices is irreducible.
\end{df}
This can be seen as a natural generalization of additive
Deligne-Simpson problems which contains 
the original problems for Fuchsian 
equations as the special case $k_{0}=\cdots=k_{p}=1$. 
In the case $k_{1}=\cdots=k_{p}=1$ and $k_{0}\le 3$, P. Boalch
obtains a necessary and sufficient condition for the 
existence of a solution of the generalized additive Deligne-Simpson 
problem in \cite{Boarx} and for an arbitrary $k_{0}$, see \cite{HY}.

In order to determine the existence condition of  a solution of 
the additive Deligne-Simpson
problem for Fuchsian systems, Crawley-Boevey \cite{C} shows that 
Fuchsian 
systems can be realized as representations of  quivers,
and applies the existence theorem of irreducible representations
of deformed preprojective algebras associated with
the quivers to the existence of solutions
of additive Deligne-Simpson problems.
One can find a review of his work in Section \ref{reviewoffuchs} and 
\ref{reviewoffuchsian}.

As a generalization of his work, we can associate our generalized 
additive Deligne-Simposn problem with 
a quiver defined as follows. The detail of the construction
shall be explained in Sections \ref{Section2}, 
\ref{Section3} and \ref{equationsandquiver}.
Let us suppose that HTL normal forms $B_{0},\ldots,B_{p}$ are written by
\[
    B_{i}=
    \mathrm{diag}\left(
	    q_{[i,1]}(z^{-1})I_{n_{[i,1]}}+R_{[i,1]}z^{-1},\ldots,
	    q_{[i,m_{i}]}(z^{-1})I_{n_{[i,m_{i}]}}+R_{[i,m_{i}]}z^{-1}
    \right)
\]
and choose complex numbers 
$\xi^{[i,j]}_{1},\ldots,\xi^{[i,j]}_{e_{[i,j]}}$
so that 
\[
	\prod_{k=1}^{e_{[i,j]}}(R_{[i,j]}-\xi^{[i,j]}_{k})=0
\]
for $i=0,\ldots,p$ and $j=1,\ldots,m_{i}$.
Set $I_{\text{irr}}:=\{i\in\{0,\ldots,p\}\mid m_{i}>1\}\cup \{0\}$ and 
$I_{\text{reg}}:=\{0,\ldots,p\}\backslash I_{\text{irr}}$.

Then let $\mathsf{Q}$  be the quiver with the set of vertices
\begin{align*}
	\mathsf{Q}_{0}:=\left\{[i,j]\,\middle|\,
		\begin{array}{l}
		i\in I_{\text{irr}},\\
		j=1,\ldots,m_{i}
	\end{array}\right\}
		\cup
		\left\{
			[i,j,k]\,\middle|\,
			\begin{array}{l}
			i=0,\ldots,p,\\
			j=1,\ldots,m_{i},\\
			k=1,\ldots,e_{[i,j]}-1
			\end{array}
		\right\}
\end{align*}
and the set of arrows 
\begin{align*}
	\mathsf{Q}_{1}=&\left\{
		\rho^{[0,j]}_{[i,j']}\colon
		[0,j]\rightarrow [i,j']\,\middle|\,
		\begin{array}{l}
			j=1,\ldots,m_{0},\\
			i\in I_{\text{irr}}\backslash\{0\},\\
			j'=1,\ldots,m_{i}
		\end{array}
	\right\}\\
	&\cup
	\left\{
		\rho^{[k]}_{[i,j],[i,j']}\colon
		[i,j]\rightarrow [i,j']\,\middle|\,
		\begin{array}{l}
			i\in I_{\text{irr}},\,
			1\le j<j'\le m_{i},\\
			1\le k\le d_{i}(j,j')
		\end{array}
	\right\}\\
	&\cup
	\left\{
		\rho_{[i,j,1]}\colon [i,j,1]\rightarrow [i,j]\mid
		i\in I_{\text{irr}},\,j=1,\ldots,m_{i}
	\right\}\\
	&\cup
	\left\{
		\rho^{[i,1,1]}_{[0,j]}\colon
		[i,1,1]\rightarrow [0,j]\mid 
		i\in I_{\text{reg}},\,j=1,\ldots,m_{0}
	\right\}\\
	&\cup
	\left\{
		\rho_{[i,j,k]}\colon [i,j,k]\rightarrow
		[i,j,k-1]\,\middle|\,
		\begin{array}{l}
			i=0,\ldots,p,\\
			j=1,\ldots,m_{i},\\
			k=2,\ldots,e_{[i,j]}-1
		\end{array}
	\right\}.
\end{align*}
Here $d_{i}(j,j'):=\mathrm{deg\,}_{\mathbb{C}[z]}(q_{[i,j]}(z)
-q_{[i,j']}(z))-2$.
To each vector 
      $\beta\in \mathbb{Z}^{\mathsf{Q}_{0}}$,
      we associate integers 
      \begin{align*}
	      q(\beta)&:=\sum_{a\in \mathsf{Q}_{0}}\beta_{a}^{2}
	      -\sum_{\rho\in \mathsf{Q}_{1}}
      \beta_{s(\rho)}\beta_{t(\rho)},&
      p(\beta)&:=1-q(\beta).
  \end{align*}
Here $s(\rho)$ and $t(\rho)$ are the source and target of the arrow $\rho$
respectively.
Let $\alpha=(\alpha_{a})_{a\in \mathsf{Q}_{0}}
\in \mathbb{Z}^{\mathsf{Q}_{0}}$ be the vector defined by setting  
$\alpha_{[i,j]}:=n_{[i,j]}$ and 
$\alpha_{[i,j,k]}:=\mathrm{rank\,}
\prod_{l=1}^{k}(R^{(i)}_{j}-\xi_{l}^{[i,j]})$.
Also define $\lambda=(\lambda_{a})_{a\in \mathsf{Q}_{0}}\in 
\mathbb{C}^{\mathsf{Q}_{0}}$ by 
$\lambda_{[i,j]}:=-\xi^{[i,j]}_1$ for $i\in I_{\text{irr}}\backslash\{0\}$,
$j=1,\ldots,m_{i}$, $\lambda_{[0,j]}:=-\xi^{[0,j]}
-\sum_{i\in I_{\text{reg}}}\xi^{[i,1]}_{1}$ for $j=1,\ldots,m_{0}$,
and $\lambda_{[i,j,k]}:=\xi^{[i,j]}_{k}-\xi^{[i,j]}_{k+1}$ for 
$i=0,\ldots,p$, $j=1,\ldots,m_{i}$ and $k=1,\ldots,e_{[i,j]}-1$.
The following sublattice of $\mathbb{Z}^{\mathsf{Q}_{0}}$ plays an
essential role in this paper, 
\[
	\mathcal{L}:=\left\{\beta\in 
\mathbb{Z}^{\mathsf{Q}_{0}}\,\middle|\,
	\sum_{j=1}^{m_{0}}\beta_{[0,j]}=\sum_{j=1}^{m_{i}}\beta_{[i,j]}
\text{ for all }i\in I_{\text{irr}}\backslash\{0\}
\right\}.
\]
Set $\mathcal{L}^{+}:=\mathcal{L}\cap (\mathbb{Z}_{\ge 0})^{\mathsf{Q}_{0}}$.

Then the following is the main theorem of this paper.
\begin{thm}[see Theorem \ref{mainthm}]\label{premainthm}
	Let us consider the generalized 
	additive  Deligne-Simpson problem consisting of
	positive integers $k_{0},\ldots,k_{p}$ and HTL normal forms
	$B_{i}\in \mathfrak{g}_{k_{i}}^{*}$ for $i=0,\ldots,p$.
	Let us define $\mathsf{Q}=(\mathsf{Q}_{0},\mathsf{Q}_{1})$, 
	$\alpha\in (\mathbb{Z}_{\ge 0})^{\mathsf{Q}_{0}}$
	and $\lambda\in \mathbb{C}^{\mathsf{Q}_{0}}$ as above.
	Then the generalized additive Deligne-Simpson problem has 
	a solution if and only if the following are satisfied,
	\begin{enumerate}
		\item $\alpha$ is a positive root of $\mathsf{Q}$ and 
			$\alpha\cdot \lambda=
			\sum_{a\in \mathsf{Q}_{0}}\alpha_{a}\lambda_{a}=0$,
		\item for any decomposition $\alpha=\beta_{1}+\cdots+
			\beta_{r}$ where $\beta_{i}\in 
			\mathcal{L}^{+}$ are positive roots of 
			$\mathsf{Q}$ satisfying $\beta_{i}\cdot
			\lambda=0$, we have 
			\[
				p(\alpha)>
				p(\beta_{1})+\cdots +p(\beta_{r}).
			\]\label{test}
	\end{enumerate}
\end{thm}
Here we note that the condition \ref{test} is weaker than the corresponding 
condition of Crawley-Boevey's theorem (see Theorem 1.2 in \cite{C1} and 
see also Theorem \ref{CB}) since we deal only with positive roots in 
$\mathcal{L}$.
If, however, $I_{\text{irr}}=\{0\}$ which contains known cases by 
Crawley-Boevey \cite{C}, Boalch \cite{Boarx}, and \cite{HY},
then $\mathcal{L}=
\mathbb{Z}^{\mathsf{Q}_{0}}$ and 
the conditions in the above theorem coincide with 
Crawley-Boevey's one.
Thus the theorem covers the above preceding known results for the additive Delinge-Simpson
problems.

Let us  discuss moduli spaces of meromorphic 
connections.
In the theory of isomonodromic deformation, the Riemann-Hilbert problem of 
moduli spaces of meromorphic connections not only with
regular singularities but also 
irregular singularities plays a central role,
see \cite{JMU} for instance. 
Thus 
many researchers are interested in the geometry of the moduli spaces
of the meromorphic connections with irregular singularities, see 
\cite{Boa1} \cite{BreSag},\cite{IS} for instance.
As an application of our main theorem, we shall discuss some 
geometric properties of the moduli spaces.
Precisely to say, following the Boalch's paper \cite{Boa1}, we define 
the moduli space $\mathfrak{M}(\mathbf{B})$ 
of meromorphic connections 
on trivial bundles associated with the collection of HTL normal forms 
$\mathbf{B}=(B_{i})_{0\le i\le p}$, see Section \ref{moduli conn}. 
Then we shall show that $\mathfrak{M}(\mathbf{B})$ can 
be embedded onto an open subset of a quiver variety 
$\mathfrak{M}^{\text{reg}}_{\lambda}(\mathsf{Q},\alpha)$.
\begin{thm}[see Theorem \ref{embedding}]
Let us take $\mathbf{B}=
	 (B_{i})_{0\le i\le p}$, the collection of HTL normal forms,
	 the quiver $\mathsf{Q}$,
	 $\alpha\in (\mathbb{Z}_{\ge 0})^{\mathsf{Q}_{0}}$ and 
	 $\lambda\in \mathbb{C}^{\mathsf{Q}_{0}}$ as in 
	 Theorem \ref{premainthm}.
	 Then there exists $\lambda'\in \mathbb{C}^{\mathsf{Q}_{0}}$ and
	 an injection
	 \[
		 \Phi\colon \mathfrak{M}(\mathbf{B})
		 \hookrightarrow
		 \mathfrak{M}^{\text{reg}}_{\lambda'}(\mathsf{Q},\alpha)
	 \]
	 such that 
	 \begin{equation*}
		 \Phi(\mathfrak{M}(\mathbf{B}))=
		 \left\{x\in 
			 \mathfrak{M}^{\text{reg}}_{\lambda'}(\mathsf{Q},
			 \alpha)\,\middle|\,
			 \mathrm{det}\left(
				 x_{\rho^{[0,j]}_{[i,j']}}
			 \right)_{\substack{1\le j\le m_{0}\\
			 1\le j'\le m_{i}}}\neq 0,\
		 i\in I_{\text{irr}}\backslash\{0\}
		 \right\}.
	 \end{equation*}
	 In particular if $I_{\text{irr}}=\{0\}$, then $\lambda'=\lambda$
	 and $\Phi$ is bijective.
\end{thm}
As a corollary of this embedding theorem, we can show the connectedness of 
$\mathfrak{M}(\mathbf{B})$.
\begin{thm}[see Theorem \ref{moduliconnected}]
	If $\mathfrak{M}(\mathbf{B})\neq \emptyset$, then 
	$\mathfrak{M}(\mathbf{B})$ has a structure of 
	connected complex manifold.
\end{thm}
Let us give a remark of the theorems 
which is already obtained by many 
researchers under several restrictions.
If $\#I_{\text{irr}}=1$, it is known that 
$\mathfrak{M}(\mathbf{B})$ is isomorphic to the quiver 
variety $\mathfrak{M}^{\text{reg}}_{\lambda}(\mathsf{Q},\alpha)$
by Crawley-Boevey for 
the Fuchsian case in \cite{C} and  
by Boalch in \cite{Boarx} and the 
work of D. Yamakawa with the author in \cite{HY}
for the case $\#I_{\text{irr}}=1$.
However for the case $\#I_{\text{irr}}>1$, Boalch
gave an example of a moduli space which is not isomorphic 
to any quiver varieties in \cite{Boarx}.
To avoid the difficulty, Yamakawa \cite{Yam} defines 
a generalization of quiver varieties which can be realized as 
$\mathfrak{M}(\mathbf{B})$ for 
general $I_{\text{irr}}$ under the restriction $m_{i}\le 2$
for all $i\in I_{\text{irr}}$.
In the above theorem we impose no restriction to $\mathbf{B}$ and 
show that although the moduli space 
$\mathfrak{M}(\mathbf{B})$ may
not be isomorphic to a quiver variety itself as Boalch suggested,
there always exists
the embedding onto the open subset of the quiver variety.

Furthermore, Theorem \ref{premainthm} determines the non-emptiness of 
the moduli space $\mathfrak{M}(\mathbf{B})$ as 
a generalization of the results for the case 
$\#I_{\text{irr}}\le 1$ by Crawley-Boevey \cite{C},
Boalch \cite{Boarx} and Yamakawa and the author \cite{HY}.
\begin{thm}(see Corollary \ref{maincor})
The moduli space $\mathfrak{M}(\mathbf{B})$ is non-empty
if and only if the following are satisfied,
	\begin{enumerate}
		\item $\alpha$ is a positive root of $\mathsf{Q}$ and 
			$\alpha\cdot \lambda=
			\sum_{a\in \mathsf{Q}_{0}}\alpha_{a}\lambda_{a}=0$,
		\item for any decomposition $\alpha=\beta_{1}+\cdots+
			\beta_{r}$ where $\beta_{i}\in 
			\mathcal{L}^{+}$ are positive roots of 
			$\mathsf{Q}$ satisfying $\beta_{i}\cdot
			\lambda=0$, we have 
			\[
				p(\alpha)>
				p(\beta_{1})+\cdots +p(\beta_{r}).
			\]
	\end{enumerate}
\end{thm}

Finally let us mention the multiplicative Deligne-Simpson problem.
Simpson and also Kostov gave necessary and sufficient conditions
for the existence of a solution of the problem under some generic conditions,
see \cite{S} and \cite{Kos}. 
In \cite{CBS}, Crawley-Boevey and P. Shaw gave a correspondence 
between the space of solutions of the problem and the so-called multiplicative 
quiver variety (see also \cite{Y0} for multiplicative quiver varieties)
and obtained a sufficient condition for the existence of a solution.
Furthermore in \cite{Bo2}, Boalch considered a generalization of 
the multiplicative Deligne-Simpson problem for differential equations 
with irregular singular points and gave correspondence between 
the space of solutions of the problem and the multiplicative quiver variety
which is a further generalization of the multiplicative quiver variety 
considered by Crawley-Boevey and Shaw in \cite{CBS}.

\noindent\textbf{Acknowledgment.}

The author expresses his gratitude to Toshio Oshima and Daisuke Yamakawa.
This project would not have completed without the collaborations with them.
He also thank Philip Boalch for reading the earlier version of this paper 
and giving him helpful comments.
This work is influenced by many other mathematicians, 
in particular the participants
of ``Workshop on Accessory Parameters'' 
mainly organized by Yoshishige Haraoka.
Most work of this project was done during his stay in 
Research Institute for Mathematical Sciences in Kyoto University.
He thanks for the hospitality and the support.  

\section{Additive Delinge-Simpson problem}
In this section, we shall define a generalization of the additive 
Delinge-Simpson problem for differential equations with at most unramified 
irregular singularities on the Riemann sphere.
Also we recall moduli spaces of meromorphic connections on trivial 
vector bundles over the Riemann sphere studied by Boalch in \cite{Boa1} and 
moreover see that the additive Deligne-Simpson problem is related to the non-emptiness 
problem of the moduli spaces.
\subsection{A generalization of the additive Deligne-Simpson problem}
As we saw in Introduction, the original additive Delinge-Simpson problem
for Fuchsian differential equations consists of 
a collection of conjugacy classes of $M(n,\mathbb{C})$.
The counterparts for irregular singular cases of the conjugacy classes  
are Hukuhara-Turrittin-Levelt normal forms of $M(n,\mathbb{C}(\!(z)\!))$.
We shall recall the definition of Hukuhara-Turrittin-Levelt normal forms and 
give a definition of the additive Deline-Simpson
problem for differential equations with at most unramified irregular 
singularities.

Let us consider a differential equation
\[
	\frac{d}{dz}Y=AY,\quad(A\in M(n,\mathbb{C}(\!(z)\!))).
\]
For $X\in \mathrm{GL}(n,\mathbb{C}(\!(z)\!))$, we define 
a new differential equation $\frac{d}{dz}\tilde{Y}=B\tilde{Y}$ by 
\[
	B:=XAX^{-1}+(\frac{d}{dz}X)X^{-1}.
\]
We write $B=:X[A]$ and call this operation the {\em gauge transform}
 of $A$ by $X$.
Let $\mathbb{C}(\!(t)\!)$ be a finite field extension of $\mathbb{C}(\!(z)\!)$,
namely there exists $r \in \mathbb{Z}_{\ge 1}$ such that $t^{r}=z$.
Then the differential equation $\frac{d}{dz}Y=AY$ over 
$\mathbb{C}(\!(z)\!)$ defines the differential equation $\frac{d}{dt}Z=
\overline{A}Z$ over $\mathbb{C}(\!(t)\!)$ where $\overline{A}:=rt^{r-1}AZ$. 
\begin{df}[HTL normal form]\normalfont
By {\em Hukuhara-Turrittin-Levelt normal form} or {\em HTL normal form}
for short, we mean an element in $M(n, \mathbb{C}(\!(t)\!))$ of the form
\[
	\mathrm{diag}\left(q_{1}(t^{-1})I_{n_{1}}+R_{1}t^{-1},\ldots,
		q_{m}(t^{-1})I_{n_{m}}+R_{m}t^{-1}
	\right)
\]
where $t^{r}=z$, $q_{i}(s)\in s^2\mathbb{C}[s]$ satisfying $q_{i}\neq q_{j}$
if $i\neq j$, and $R_{i}\in M(n_{i},\mathbb{C})$.
In particular when $r=1$, the normal form is said to be {\em unramified}. 
\end{df}
The following is a fundamental fact of the local formal theory of 
differential equations with irregular singularity.
\begin{thm}[Hukuhara-Turrittin-Levelt, see \cite{W} for instance]
	For any $A\in M(n,\mathbb{C}(\!(z)\!))$, there exists a field 
	extension $\mathbb{C}(\!(t)\!)\supset \mathbb{C}(\!(z)\!)$
	with $t^{r}=z$, $r\in \mathbb{Z}_{\ge 1}$ and $X\in 
	\mathrm{GL}(n,\mathbb{C}(\!(t)\!))$ such that 
	$\overline{X[A]}$ is an HTL normal form in $M(n,\mathbb{C}(\!(t)\!))$.
\end{thm}
We call this $\overline{X[A]}$ the {\em normal form} of $A$ with 
	the {\em ramification index} $r$.

Let us consider an unramified HTL normal form
\[
	B=\mathrm{diag}\left(q_{1}(z^{-1})I_{n_{1}}+R_{1}z^{-1},\ldots,
		q_{m}(z^{-1})I_{n_{m}}+R_{m}z^{-1}
	\right)
\]
and set 
\[
	k:=\mathrm{max}_{i=1,\ldots,m}\{\mathrm{deg}_{\mathbb{C}[z^{-1}]}q_{i}(z^{-1})\}.
\]
We shall consider an orbit of $B$ under the following group action.
Let us define
$G_{k}:=\mathrm{GL}(n,\mathbb{C}[\![z]\!]/z^{k}\mathbb{C}[\![z]\!])$
which can be identified with 
\[
	\left\{
		A_{0}+A_{1}z+\cdots +
		A_{k-1}z^{k-1}\in\sum_{i=0}^{k-1} M(n,\mathbb{C})z^{i}
		\,\middle|\,
		A_{0}\in \mathrm{GL}(n,\mathbb{C})
	\right\}.
\]
Also define
\begin{align*}
	\mathfrak{g}_{k}&:=M(n,\mathbb{C}[\![z]\!]/z^{k}\mathbb{C}[\![z]\!])\\
		       &\cong
	\left\{	A_{0}+A_{1}z+\cdots +
		A_{k-1}z^{k-1}\,\middle|\,
		A_{i}\in M(n,\mathbb{C}),\,
		i=0,1,\ldots,k-1
	\right\}.
\end{align*}
The group $G_{k}$ acts on $\mathfrak{g}_{k}$ by the {\em adjoint action} 
$\mathrm{Ad}(g)X:=gXg^{-1}$  for $g\in G_{k},X\in\mathfrak{g}_{k}$.
The dual vector space $\mathfrak{g}_{k}^{*}$ is identified with
\[
	M(n,z^{-k}\mathbb{C}[\![z]\!]/\mathbb{C}[\![z]\!])
	\cong 
	\left\{
		\frac{A_{k}}{z^{k}}+\cdots 
		+\frac{A_{1}}{z}\,\middle|\,
		A_{i}\in M(n,\mathbb{C}),\,
		i=1,\ldots,k
	\right\}
\]
by the bilinear form 
\[\mathfrak{g}_{k}\times 
\mathfrak{g}_{k}^{*}\ni (A,B)\mapsto 
\mathrm{Res}(\mathrm{tr}(AB))\in \mathbb{C}\]
where 
$\mathrm{Res}(\sum_{i=r}^{\infty}a_{i}z^{i}):=a_{-1}$ for 
$\sum_{i=r}^{\infty}a_{i}z^{i}\in \mathbb{C}(\!(z)\!)$.
Let us note that the {\em coadjoint action} of $G_{k}$ on 
$\mathfrak{g}_{k}^{*}$ is defined by
$(\mathrm{Ad}^{*}(g)f)(X):=f(\mathrm{Ad}(g^{-1})X)$
for $g\in G_{k},f\in \mathfrak{g}_{k}^{*}, X\in \mathfrak{g}_{k}$.
The coadjoint action induces the action of   $G_{k}$ 
on $M(n,z^{-k}\mathbb{C}[\![z]\!]/\mathbb{C}[\![z]\!])$ defined by 
$\mathrm{Ad}^{*}(g)Z:=g^{-1}Zg$ for $g\in G_{k}, Z\in
M(n,z^{-k}\mathbb{C}[\![z]\!]/\mathbb{C}[\![z]\!])\cong \mathfrak{g}_{k}^{*}$.

Since we can regard $B\in \mathfrak{g}_{k}^{*}$, the orbit of $B$
under the coadjoint action of $G_{k}$.
\begin{df}[truncated orbit]\normalfont
	Let us regard $B\in \mathfrak{g}_{k}^{*}$. Then
\[
	\mathcal{O}_{B}:=\{\mathrm{Ad}^{*}(g)B\mid g\in G_{k}\}
\]
is called the {\em truncated orbit} of $B$.
\end{df}

Now we are ready to define a generalization of the additive Delinge-Simpson 
problem for differential equations with unramified irregular singularities. 
We say that a collection of matrices 
$(A_{1},\ldots,A_{s})\in M(n,\mathbb{C})^{s}$ is {\em irreducible} if 
$(A_{1},\ldots,A_{s})$ has no nontrivial invariant subspace of $\mathbb{C}^{n}$,
i.e., if a subspace $W\subset \mathbb{C}^{n}$ satisfies that 
$A_{i}W\subset W$ for all $i=1,\ldots,s$, then $W=\{0\}\text{ or }
\mathbb{C}^{n}$.
For a differential equation
\[
\frac{d}{dz}Y=\left(\sum_{i=1}^{p}\sum_{j=1}^{k_{i}}
\frac{A_{i,j}}{(z-a_{i})^{j}}
+\sum_{2\le j\le k_{0}}A_{0,j}z^{j-2}\right)Y,
\]
the {\em principal term} at the singular point $a_{i}$ is 
\[
	A_{i}(z_{i}):=\sum_{j=1}^{k_{i}}A_{i,j}z_{i}^{-j}
\]
for each $i=0,\ldots,p$. Here we set $A_{0,1}:=-\sum_{i=1}^{p}A_{i,1}$ 
and $z_{i}:=z-a_{i}$, $i=1,\ldots,p$, $z_{0}:=\frac{1}{z}$.
This differential equation is said to be {\em irreducible} if 
the collection of the matrices $(A_{i,j})_{\substack{0\le i\le p,\\
1\le j\le k_{i}}}$ is irreducible. 

\begin{df}[additive Delinge-Simpson problem]\label{gen add DS}\normalfont
Let us take 
$k_{i}\in \mathbb{Z}_{\ge 1}$ and  
unramified HTL normal forms $
B_{i}\in \mathfrak{g}_{k_{i}}^{*}$
for $i=0,1,\ldots,p$. Then a {\em solution} of the additive Delinge-Simpson
problem for the collection of the unramified HTL normal forms
$(B_{0},B_{1},\ldots,B_{p})$ is an 
irreducible differential equation 
 \[
\frac{d}{dz}Y=\left(\sum_{i=1}^{p}\sum_{j=1}^{k_{i}}
\frac{A_{i,j}}{(z-a_{i})^{j}}
+\sum_{2\le j\le k_{0}}A_{0,j}z^{j-2}\right)Y
\]
such that  the principal term at each singular point $a_{i}$, 
$i=0,1,\ldots,p$ satisfies
\[
	A_{i}(z)\in \mathcal{O}_{B_{i}}.
\]
\end{df}
\begin{rem}\normalfont
Let us note that if $k_{0}=k_{1}=\cdots =k_{p}=1$, then 
$G_{k_{i}}=\mathrm{GL}(n,\mathbb{C})$ and $\mathfrak{g}_{k_{i}}^{*}
=M(n,\mathbb{C})$. Thus 
the truncated orbits $\mathcal{O}_{B_{i}}$ are just conjugacy 
classes of $M(n,\mathbb{C})$. Therefore the additive Delinge-Simpson
problem in Definition \ref{gen add DS} contains the original additive 
Delinge-Simpson problem for Fuchsian differential equations.
\end{rem}
\subsection{Moduli spaces of meromorphic connections 
and  additive Delinge-Simpson problem}\label{moduli conn}
In this section we quickly 
recall the definition of moduli spaces of 
meromorphic connections on trivial vector bundles 
over the Riemann sphere following \cite{Boa1}. For the detailed treatment can be 
found in the original paper by Boalch \cite{Boa1} and we also refer \cite{HY} and 
their references.
The solvability of the 
additive Delinge-Simpson problems can be seen as the problem determining 
the necessary and sufficient condition of the non-emptiness of the moduli spaces

Let us recall the notion of meromorphic connections and 
see their relationship with differential equations.
For $f=\sum_{i>-\infty}^{\infty}a_{i}z^{i}
\in \mathbb{C}(\!(z)\!)$, the {\em order} is  
\[
	\mathrm{ord}(f):=\mathrm{min}\{i\mid a_{i}\neq 0\}.
\]
If $f=0$, we formally put $\mathrm{ord}(f)=\infty$.
For a meromorphic function $f$ locally defined near $a\in \mathbb{P}^{1}$, 
we denote the germ of $f$ at 
$a$ by $f_{a}$. 
We may see $f_{a}\in  
\mathbb{C}(\!(z_{a})\!)$ by setting $z_{a}=z-a$
if $a\in \mathbb{C}$ and $z_{a}=1/z$ if $a=\infty$ where we take 
$z$ as the standard coordinate of $\mathbb{C}$.
Then define 
\[
	\mathrm{ord}_{a}(f):=\mathrm{ord}(f_{a}).
\]
For a meromorphic 1-form $\omega$ defined on $\mathbb{P}^{1}$, the order 
$\mathrm{ord}_{a}(\omega)$ can be defined as follows. 
Set $U_{1}=\mathbb{P}^{1}\backslash 
\{\infty\}$ and $U_{2}=\mathbb{P}^{1}\backslash\{0\}$. 
Let $z_{i}$ be coordinates of $U_{i}$, $i=1,2$, such that 
$z_{1}(0)=z_{2}(\infty)=0$ and $z_{2}=1/z_{1}$ in $U_{1}\cap U_{2}$.
Then there exist meromorphic functions $f_{i}$ on $U_{i}$ such that 
\[
	\omega =f_{i}\,dz_{i}
\]
on $U_{i}$ for $i=1,2$.  
Then define
\[
	\mathrm{ord}_{a}(\omega):=\ord_{a}(f_{i})
\]
for $a\in U_{i}$, $i=1,2$.
 
Let us fix a collection of  points ${a_{0},\ldots,a_{p}}\in \mathbb{P}^1$
 and set
$S:=k_{0}a_{0}+\cdots +k_{p} a_{p}$ as an effective
divisor with $k_0,\ldots,k_{p}>0$.
For $a\in \mathbb{P}^{1}$ 
let $S(a)$ be the coefficient of $a$ in $S$, i.e.,
\[
	S(a):=
	\begin{cases}
	k_{i}&\text{if }a=a_{i}\text{ for }i=0,\ldots,p,\\
	0&\text{otherwise.}
	\end{cases}
\]
For an open set $U\subset \mathbb{P}^{1}$ we define 
$\Omega_{S}(U)$ to be the set of all meromorphic 1-forms $\omega$
on $U$ satisfying $\mathrm{ord}_a(\omega)\ge -S(a)$ for any $a\in U$.
This correspondence defines the sheaf $\Omega_{S}$ by the natural
restriction mappings.

Let $\mathcal{E}$ be a locally free sheaf of rank $n$ on $\mathbb{P}^{1}$,
namely a sheaf of modules over the sheaf $\mathcal{O}$ of 
holomorphic functions on $\mathbb{P}^{1}$ satisfying 
that for any $a\in \mathbb{P}^{1}$
there exists an open neighbourhood $V\subset \mathbb{P}^{1}$ such that 
$\mathcal{E}|_{V}\cong \mathcal{O}^{n}|_V$.
We may sometimes regard $\mathcal{E}$ as a holomorphic vector bundle over 
$\mathbb{P}^{1}$.
\begin{df}[meromorphic connection]\normalfont
A {\em meromorphic connection} is a pair $(\mathcal{E},\nabla)$ 
of a locally free sheaf $\mathcal{E}$ and a morphism 
$\nabla\colon \mathcal{E}\rightarrow \mathcal{E}\otimes \Omega_{S}$ 
of sheaves of $\mathbb{C}$-vector spaces satisfying
\[
	\nabla(fs)=df\otimes s+f\otimes \nabla(s)
\]
for all $f\in \mathcal{O}(U)$, $s\in \mathcal{E}(U)$
and open subsets $U\subset \mathbb{P}^{1}$.
\end{df}
Let $U\subset \mathbb{P}^{1}$ be an open subset which gives a local 
trivialization of $\mathcal{E}$ and $z$ a local coordinate of $U$.
Then if we fix an identification
$\mathcal{E}|_{U}\cong \mathcal{O}^{n}|_{U}$, 
we can write $\nabla=d-A\,dz$ by $A\in M(n,\mathcal{M}(U))$ on $U$. 
Note that if we write $\nabla=d-A'\,dz$ by another identification
$\mathcal{E}|_{U}\cong \mathcal{O}^{n}|_{U}$, then $A'$ can be obtained by 
a {\em holomorphic gauge transformation} of $A$, namely there exists 
$X\in \mathrm{GL}(n,\mathcal{O}(U))$ such that 
\[
	A'=X[A].
\]
Thus we may say that $(\mathcal{E},\nabla)$ defines a holomorphic gauge 
equivalent class of a local differential equation 
\[
	\frac{d}{dz}Y=AY
\]
on $U\subset \mathbb{P}^{1}$.

In particular, suppose that  $\mathcal{E}$ is {\em trivial}, i.e., 
$\mathcal{E}\cong \mathcal{O}^{n}$ and set $U_{1}=\mathbb{P}^{1}
\backslash\{\infty\}$ and $U_{2}=\mathbb{P}^{1}\backslash\{0\}$
as before.
Then if we fix a trivialization $\mathcal{E}\cong\mathcal{O}^{n}$, 
we have
$\nabla=d-A(z_{1})dz_{1}$ on $U_{1}$ with 
$A(z_{1})=(\alpha_{i,j}(z_{1}))_{i,j=1,\ldots,n}\in M(n,\mathbb{C}(z))$ 
satisfying
$\ord_{a}(\alpha_{i,j})\ge -S(a)$ for all $a\in U_{1}$.
Similarly on $U_{2}$ we have $\nabla=d-B(z_{2})dz_{2}$. 
Since $\mathcal{E}$ is trivial, 
\[
	A(z_{1})dz_{1}=B(z_{2})dz_{2}\text{ on }U_{1}\cap U_{2}.
\]
Namely, 
\[
	B(z_{2})=-\frac{A(1/z_{2})}{z_{2}^{2}}.
\]
This is nothing but the coordinate exchange $\zeta=\frac{1}{z}$ 
for a differential
equation
\[
	\frac{d}{dz}Y=A(z)Y\longmapsto
	-\zeta^{2}\frac{d}{d\zeta}Y=A(1/\zeta)Y.
\]

Thus a meromorphic connection 
$(\mathcal{E},\nabla)$ with a trivial bundle $\mathcal{E}$
on $\mathbb{P}^{1}$ corresponds to 
a meromorphic differential equation 
$\frac{d}{dz}Y=AY$ with 
$A=(\alpha_{i,j})_{i,j=1,\ldots,n}\in M(n,\mathbb{C}(z))$ satisfying
$\ord_{a}(\alpha_{i,j}\,dz)\ge -S(a)$ for all $a\in \mathbb{P}^{1}$,
and vice versa. 
This correspondence is 
unique up to the choice of $\mathcal{E}\cong \mathcal{O}^{n}$, i.e.,
$\mathrm{GL}(n,\mathbb{C})$-action.

Let $S=k_{0}a_{0}+\ldots+k_{p}a_{p}$ be an effective  divisor on $\mathbb{P}^{1}$
as before.
Define a set of meromorphic connections on $\mathbb{P}^{1}$ 
\[
	\mathrm{Triv}^{(n)}_{S}:=\left\{
		(\mathcal{E},\nabla) \middle| 
		\begin{array}{c}
			\mathcal{E}\colon \text{trivial of rank $n$},\\
			\nabla\colon \mathcal{E}\rightarrow 
			\mathcal{E}\otimes \Omega_{S}
			\end{array}
	\right\}.
\]
We say $(\mathcal{E},\nabla)\in \mathrm{Triv}^{(n)}_{S}$ 
is {\em stable} if there exists no nontrivial proper
subspace $W\subset \mathbb{C}^{n}$ 
such that the subbundle $\mathcal{W}:=W\otimes \mathcal{O}\subset \mathbb{C}^{n}\otimes 
\mathcal{O}=\mathcal{E}$ is closed under $\nabla$, i.e.,
\[
	\nabla(\mathcal{W})\subset \mathcal{W}\otimes \Omega_{S}.
\]

Let $\mathbf{B}=(B_{0},\ldots,B_{p})\in M(n,\mathbb{C}(\!(z)\!))^{p+1}$ 
be a collection of 
HTL normal forms satisfying $\mathrm{ord}(B_{i})=-k_{i}$ for all $i=0,\ldots,p$.
We write $\nabla|_{a_{i}}\in \mathcal{O}_{B_{i}}$ for a connection $(
\mathcal{E},\nabla)$ if 
there exists $A_{a_{i}}\in M(n,\mathbb{C}(\!(z_{a_{i}})\!))$ such 
that $\nabla=d-A_{a_{i}}\,dz_{a_{i}}$ and 
$\iota(A_{a_{i}})\in \mathcal{O}_{B_{i}}$ where $z_{a_{i}}$ is a local coordinate 
of $\mathbb{P}^{1}$ vanishing at $a_{i}$ and 
$\iota \colon M(n,\mathbb{C}(\!(z_{a_{i}})\!))\rightarrow 
M(n,\mathbb{C}(\!(z_{a_{i}})\!)/\mathbb{C}[\![z_{a_{i}}]\!])$ is the natural
projection.

Then the moduli space of stable meromorphic connections on trivial bundles
is 
\[
	\mathfrak{M}(\mathbf{B}):=
	\left\{
		(\mathcal{E},\nabla)\in 
		\mathrm{Triv}^{(n)}_{S}\,\middle|\,
		\begin{array}{c}
		(\mathcal{E},\nabla)\colon\text{ stable},\\
		\nabla|_{a_{i}}\in \mathcal{O}_{B_{i}}
		\text{ for all }i=0,\ldots,p
		\end{array}
	\right\}\big/\mathrm{GL}(n,\mathbb{C}).
\]
Here $\mathrm{GL}(n,\mathbb{C})=\mathrm{GL}(n,\mathcal{O}(\mathbb{P}^{1}))$
acts on $\text{Triv}^{(n)}_{S}$ as the holomorphic gauge transformation.

M\"obius transformation may allow us to suppose $a_{0}=\infty\in \mathbb{P}^{1}$.
Then by a trivialization  $\mathcal{E}\cong \mathcal{O}^{n}$  we can 
identify $(\mathcal{E},\nabla)\in \mathrm{Triv}^{n}_{S}$ with 
a meromorphic differential equation defined on $\mathbb{P}^{1}$,
\[
	\frac{d}{dz}Y=\left(\sum_{i=1}^{p}\sum_{\nu=1}^{k_{i}}
		\frac{A^{(i)}_{\nu}}{(z-a_{i})^{\nu}}
	+\sum_{2\le\nu\le k_{0}}A^{(0)}_{\nu}z^{\nu-2}
\right)Y
\]
up to $\mathrm{GL}(n,\mathbb{C})$-action, i.e., 
\[
	\frac{d}{dz}Y=A(z)Y\longmapsto \frac{d}{dz}Y'=gA(z)g^{-1}Y'
	\quad (g\in \mathrm{GL}(n,\mathbb{C})).
\]
The stability of $(\mathcal{E},\nabla)$ corresponds to the irreducibility
of the differential equation.
Thus we can regard $\mathfrak{M}(\mathbf{B})$ as a moduli space of 
meromorphic differential equations on $\mathbb{P}^{1}$,
\begin{multline*}
	\mathfrak{M}(\mathbf{B})=
	\left\{\frac{d}{dz}Y=\left(\sum_{i=1}^{p}\sum_{\nu=1}^{k_{i}}
		\frac{A^{(i)}_{\nu}}{(z-a_{i})^{\nu}}
	+\sum_{2\le\nu\le k_{0}}A^{(0)}_{\nu}z^{\nu-2}
\right)Y\right.\\
\left.\,\middle|\,\begin{array}{c}\text{irreducible},\\
 \sum_{\nu=1}^{k_{i}}\frac{A^{(i)}_{\nu}}{z^{\nu}}\in 
\mathcal{O}_{B_{i}}, i=0,\ldots,p
\end{array}
	\right\}\bigg/ \mathrm{GL}(n,\mathbb{C}).
\end{multline*}
Thus  the solvability of the additive Deligne-Simpson
problem is rephrased as the non-emptiness of the moduli space.
\begin{prop}
	There is a solution of the  additive Delinge-Simpson problem for 
	$\mathbf{B}=(B_{0},\ldots,B_{p})$ if and only if 
	$\mathfrak{M}(\mathbf{B})\neq \emptyset$.
\end{prop}

Furthermore, forgetting the location of the singular points, we may 
regard $\mathfrak{M}(\mathbf{B})$ as a subspace of the orbit space 
$\prod_{i=0}^{p}\mathcal{O}_{B_{i}}$,
\begin{multline*}
	\mathfrak{M}(\mathbf{B})=\\
	\left\{
		\mathbf{A}=(A_{i}(z))_{0\le i\le p}
		\in \prod_{i=0}^{p}\mathcal{O}_{B_{i}}
		\middle|\ 
\begin{array}{c}\mathbf{A}\text{ is irreducible }, \\
	\sum_{i=0}^{p}\mathrm{Res\,}(A_{i}(z))=0
\end{array}
\right\}\bigg/ \mathrm{GL}(n,\mathbb{C})
\end{multline*}
which is free from locations of $a_{i}$ in $\mathbb{P}^{1}$.
Here 
\[\mathrm{Res\,}(\sum_{j=1}^{k}A_{j}z^{-j}):=A_{1}
\]
and we say that 
$\mathbf{A}=(\sum_{j=1}^{k_{i}}A_{i,j}z^{-1})_{0\le i\le p}$ is 
{\em irreducible} if $(A_{i,j})_{\substack{0\le i\le p\\1\le j\le k_{i}}}$
is irreducible.
\section{A review of representations of quivers}\label{reviewoffuchs}
This section is a review of known results of the 
representation theory of quivers and theory of quiver varieties.
The review will be minimized only for our requirement for the latter sections 
and we refer original papers for the general theory 
by Nakajima \cite{Nak}, Crawley-Boevey and Holland
\cite{CBH}, Crawley-Boevey \cite{C1} and their references.
\subsection{Representations of quivers and quiver varieties} 
Here we recall the definition of representations of quivers and introduce
quiver varieties.
\begin{df}[quivers]
\normalfont
	A \textit{quiver} $\mathsf{Q}
	=(\mathsf{Q}_{0},\mathsf{Q}_{1},s,t)$ is the
	quadruple consisting 
	of $\mathsf{Q}_{0}$, the set of \textit{vertices}, 
	and $\mathsf{Q}_{1}$,
        the set of \textit{arrows} connecting 
	vertices in $\mathsf{Q}_{0}$,
	and two maps $s,t\,\colon \mathsf{Q}_{1}\rightarrow \mathsf{Q}_{0}$,
	which associate to each arrow $\rho\in \mathsf{Q}_{1}$ its 
	{\em source} $s(\rho)\in \mathsf{Q}_{0}$ and 
	its {\em target} $t(\rho)\in \mathsf{Q}_{0}$ respectively.
\end{df}
\begin{df}[representations of quivers]
      \normalfont
      Let $\mathsf{Q}$ be a finite quiver, i.e., 
      $\mathsf{Q}_{0}$ and $\mathsf{Q}_{1}$ are finite sets. 
      A {\em representation} $M$ of $\mathsf{Q}$ is defined by the 
      following data:
      \begin{enumerate}
	      \item To each vertex $a$ in $\mathsf{Q}_{0}$,  a finite 
              dimensional $\mathbb{C}$-
              vector space $M_{a}$ is attached.
      \item To each arrow $\rho\colon a\rightarrow b$ in $\mathsf{Q}_{1}$,
              a $\mathbb{C}$-linear map 
              $\psi_{\rho}\colon M_{a}\rightarrow M_{b}$
              is attached.
      \end{enumerate}   
\end{df}      
     We denote the representation by 
      $M=(M_{a},\psi_{\alpha})_{a\in \mathsf{Q}_{0},\alpha\in \mathsf{Q}_{1}}$.
      The collection of integers defined by $\mathbf{dim\,}M
      =(\mathrm{dim}_{\mathbb{C}}M_{a})_{a\in \mathsf{Q}_{0}}$ 
      is called the {\em dimension vector} of 
      $M$.

For a fixed vector $\alpha\in (\mathbb{Z}_{\ge 0})^{\mathsf{Q}_{0}}$,
the representation space is 
\[
	\mathrm{Rep}(\mathsf{Q},V,\alpha)=\bigoplus_{\rho\in \mathsf{Q}_{1}}
	\mathrm{Hom}_{\mathbb{C}}(V_{s(\rho)},V_{t(\rho)}),
\]
where $V=(V_{a})_{a\in \mathsf{Q}_{0}}$ is a collection of finite dimensional 
$\mathbb{C}$-vector spaces with 
$\mathrm{dim}_{\mathbb{C}}V_{a}=\alpha_{a}$.
If $V_{a}=\mathbb{C}^{\alpha_{a}}$ for all $a\in \mathsf{Q}_{0}$, we simply
write 
$$\mathrm{Rep\,}(\mathsf{Q},\alpha)=
\bigoplus_{\rho\in \mathsf{Q}_{1}}
\mathrm{Hom}_{\mathbb{C}}(\mathbb{C}^{\alpha_{s(\rho)}},
\mathbb{C}^{\alpha_{t(\rho)}}).$$

The representation space $\mathrm{Rep}(\mathsf{Q},V,\alpha)$ has 
an action of $\prod_{a\in \mathsf{Q}_{0}}\mathrm{GL}(V_{a})$.
For $(\psi_{\rho})_{\rho\in \mathsf{Q}_{1}}\in \mathrm{Rep\,}(\mathsf{Q},V,\alpha)$ and 
$g=(g_{a})\in \prod_{a\in \mathsf{Q}_{0}}\mathrm{GL}(V_{a})$,
then 
$g\cdot (\psi_{\rho})_{\rho\in \mathsf{Q}_{1}}\in \mathrm{Rep\,}(\mathsf{Q},V,\alpha)$ 
consists of 
$\psi'_{\rho}=g_{t(\rho)}\psi_{\rho}g_{s(\rho)}^{-1}
\in \mathrm{Hom}_{\mathbb{C}}(V_{s(\rho)},V_{t(\rho)})$.
   
Let  
$M=(M_{a},\psi^{M}_{\rho})_{a\in \mathsf{Q}_{0},\rho\in \mathsf{Q}_{1}}$
and 
$N=(N_{a},\psi^{N}_{\rho})_{a\in \mathsf{Q}_{0},\rho\in \mathsf{Q}_{1}}$ 
be representations of a quiver $\mathsf{Q}$.
Then $N$ is called the {\em subrepresentation} of M
if we have the following:
\begin{enumerate}
	\item For each $a\in \mathsf{Q}_{0}$, $N_{a}\subset M_{a}$.
          	  \item For each $\rho\colon a\rightarrow b\in \mathsf{Q}_{1}$,
               $\psi^{M}_{\rho}|_{N_{a}}
              =\psi^{N}_{\rho}$.
\end{enumerate}
In this case we denote $N\subset M$.
Moreover if  
\begin{itemize}
	\item[(3)] there exists a direct sum decomposition
		  $M_{a}=N_{a}\oplus N'_{a}$ for each $a\in \mathsf{Q}_{0}$,

	      \item[(4)] 
		      for each $\rho\colon a\rightarrow b\in \mathsf{Q}_{1}$,
              we have $\psi^{M}_{\rho}|_{N'_{a}}\subset N_{b}^{'}$,
\end{itemize}
then we say $M$ has a {\em direct sum decomposition} 
$M=N\oplus N'$ where 
$N'
=(N'_{a},\psi^{M}_{\rho}|_{N_{a}^{'}})_{a\in \mathsf{Q}_{0},\rho\in \mathsf{Q}_{1}}$.

The representation $M$ is said to be {\em irreducible} if 
$M$ has no subrepresentations other than $M$ and $\{0\}$. 
Here $\{0\}$ is the representation of $\mathsf{Q}$ 
which consists of zero vector spaces and zero
linear maps.
On the other hand if any direct sum 
decomposition  $M=N\oplus N'$
satisfies either  $N=\{0\}$ or $N'=\{0\}$, then
$M$ is said to be {\em indecomposable}.

Let us recall the double of 
a quiver $\mathsf{Q}$.
\begin{df}[double of a quiver]
\normalfont
	Let $\mathsf{Q}=(\mathsf{Q}_{0},\mathsf{Q}_{1})$ be a finite quiver. 
	Then the {\em double quiver} $\overline{\mathsf{Q}}$ 
	of $\mathsf{Q}$ is the 
        quiver obtained by 
        adjoining the reverse arrow $\rho^{*}\colon b\rightarrow a$ 
        to each arrow $\rho\colon a\rightarrow b$. Namely
	$\overline{\mathsf{Q}}=(\overline{\mathsf{Q}}_{0}=
	\mathsf{Q}_{0},\overline{\mathsf{Q}}_{1}=
	\mathsf{Q}_{1}\cup \mathsf{Q}_{1}^{*})$ where 
	$\mathsf{Q}_{1}^{*}=\{\rho^{*}\colon t(\rho)\rightarrow s(\rho)\mid 
	\rho \in \mathsf{Q}_{1}\}$. 
\end{df}

Here we note that the representation 
space $\mathrm{Rep}(\overline{
\mathsf{Q}},\alpha)$ of the double quiver $\overline{
\mathsf{Q}}$ can be regarded as 
the cotangent bundle of 
$\mathrm{Rep}(\mathsf{Q},\alpha)$,
namely 
\[
	\mathrm{Rep}(\overline{
\mathsf{Q}},\alpha)\cong 
\mathrm{Rep}(\mathsf{Q},\alpha)\oplus \mathrm{Rep}(\mathsf{Q},\alpha)^{*}
	\cong T^{*}
	\mathrm{Rep}(\mathsf{Q},\alpha),
\]
since we have the identification
\[
	\mathrm{Hom}_{\mathbb{C}}(\mathbb{C}^{\alpha_{s(\rho)}},
	\mathbb{C}^{\alpha_{t(\rho)}})^{*}\cong
	\mathrm{Hom}_{\mathbb{C}}(\mathbb{C}^{\alpha_{s(\rho^{*})}},
	\mathbb{C}^{\alpha_{t(\rho^{*})}})
\]
for each $\rho\in \mathsf{Q}_{1}$.
Then we can regard $\mathrm{Rep}(\overline{
	\mathsf{Q}},\alpha)
	\cong T^{*}
	\mathrm{Rep}(\mathsf{Q},\alpha)$
as a symplectic manifold with the canonical symplectic form 
\[
	\omega(x,y):=
	\sum_{\rho\in \mathsf{Q}_{1}}(\mathrm{tr}(x_{\rho}y_{\rho^{*}})-
	\mathrm{tr}(x_{\rho^{*}}y_{\rho}))
\]
for $x,y\in T^{*}
	\mathrm{Rep}(\mathsf{Q},\alpha)$, which is 
	invariant under the action of 
\[
	\mathbf{G}(\alpha):=
	\prod_{a\in \mathsf{Q}_{0}}\mathrm{GL}(\alpha_{a},\mathbb{C}).
\]
Then we define a {\em moment map}  of the symplectic manifold 
$\mathrm{Rep}(\overline{
	\mathsf{Q}},\alpha)$
	with the $\mathbf{G}(\alpha)$-action as follows.
The map 
\[
	\mu_{\alpha}\colon \mathrm{Rep}(\overline{\mathsf{Q}},\alpha)
	\rightarrow 
\mathrm{Lie\,}\mathbf{G}(\alpha):=
\prod_{a\in \mathsf{Q}_{0}}M(\alpha_{a},\mathbb{C})
\]
is defined by 
\[
	\mu_{\alpha}(x)_{a}=\sum_{
	\substack{\rho\in \mathsf{Q}_{1}\\ t(\rho)=a}}x_{\rho}
        x_{\rho^{*}}-
        \sum_{
	\substack{\rho\in \mathsf{Q}_{1}\\s(\rho)=a}}x_{\rho^{*}}
        x_{\rho},
\]      
for $x=(x_{\rho})_{
\rho\in \overline{\mathsf{Q}}_{1}}\in 
\mathrm{Rep\,}(\overline{\mathsf{Q}},\alpha)$.
      
Then the quiver variety is  defined as 
the symplectic reduction of $\mathrm{Rep}(
\overline{\mathsf{Q}},\alpha)$ by the moment map $\mu$.
\begin{df}[quiver variety]\normalfont
	Let us take a collection of complex numbers $\lambda=(\lambda_{a})_{a
	\in\mathsf{Q}_{0}}\in\mathbb{C}^{\mathsf{Q}_{0}}$ and 
	regard $\lambda=(\lambda_{a}I_{\alpha_{a}})_{a\in
	\mathsf{Q}_{0}}\in \prod_{a\in \mathsf{Q}_{0}}M(\alpha_{a},\mathbb{C})$.
	Then the {\em quiver variety} is the symplectic reduction 
	\[
		\mathfrak{M}_{\lambda}(\mathsf{Q},\alpha):=
		\mu^{-1}_{\alpha}(\lambda)/\mathbf{G}(\alpha).
	\]
\end{df}

Note that the symplectic reduction   
$\mathfrak{M}_{\lambda}(\mathsf{Q},\alpha)$ is homeomorphic to 
the affine quotient scheme $\mu^{-1}_{\alpha}(\lambda)/\!/\mathbf{G}(\alpha)$
by 
the theory of Kempf-Ness \cite{KN} and Kirwan \cite{Kir}.

This variety might have singularities.  
Thus let us consider the (possibly empty) subspace
\[
	\mu_{\alpha}^{-1}(\lambda)^{\text{irr}}:=\{
	x\in \mu^{-1}_{\alpha}(\lambda)\mid x\text{ is irreducible}\}.
\]
Then the action of $\mathbf{G}(\alpha)/\mathbb{C}^{\times}$ on this space 
is proper and moreover free (see King \cite{Kin}). Thus the homogeneous space
\[
	\mathfrak{M}^{\text{reg}}_{\lambda}(\mathsf{Q},\alpha):=
	\mu^{-1}_{\alpha}(\lambda)^{\text{irr}}/\mathbf{G}(\alpha)
\]
can be seen as a complex manifold with the symlpectic structure, i.e.,
a complex symplectic manifold.
We call this manifold a quiver variety too.
\begin{rem}\normalfont
	The above quiver varieties are special ones of Nakajima quiver 
	varieties which enjoy rich geometric properties and applications 
	for representation theory and theoretical physics and so on (see 
	\cite{Nak} for instance).
\end{rem}

\subsection{Crawley-Boevey's theorems for the geometry of quiver varieties}
The regular part $\mathfrak{M}^{\text{reg}}_{\lambda}(\mathsf{Q},\alpha)$
may be empty as we noted above. Thus we recall a  
necessary and sufficient condition for the non-emptiness of 
$\mathfrak{M}^{\text{reg}}_{\lambda}(\mathsf{Q},\alpha)$ given by 
Crawley-Boevey in \cite{C1}.

First let us introduce  the root system of a quiver $\mathsf{Q}$
(cf. \cite{Kac}).
Let $\mathsf{Q}$ be a finite quiver.  
From the  {\em Euler form}  $$\langle \alpha,\beta\rangle 
:=\sum_{a\in \mathsf{Q}_{0}}\alpha_{a}\beta_{a}-\sum_{\rho\in \mathsf{Q}_{1}}
\alpha_{s(\rho)}\beta_{t(\rho)},$$  
a symmetric bilinear form and quadratic form are  defined by 
\begin{align*}(\alpha,\beta)&:=\langle \alpha,\beta\rangle
+\langle \beta,\alpha\rangle,\\
	      q(\alpha)
	      &:=\frac{1}{2}(\alpha,\alpha)
\end{align*}      
	      and  set
      $p(\alpha):=1-q(\alpha)$. 
      Here $\alpha,\beta\in \mathbb{Z}^{\mathsf{Q}_{0}}$.

For each vertex $a\in \mathsf{Q}_{0}$, define 
 $\epsilon_{a}\in \mathbb{Z}^{\mathsf{Q}_{0}}$ ($a\in \mathsf{Q}_{0}$) 
  so that $(\epsilon_{a})_a=1$, $(\epsilon_{a})_{b}=0$, 
$(b\in \mathsf{Q}_{0}\backslash\{a\})$.
We call $\epsilon_{a}$ a {\em fundamental root} if 
the vertex $a$ has no edge-loop, i.e., there is no arrow $\rho$ such that 
$s(\rho)=t(\rho)=a$.
Denote by $\Pi$ the set of fundamental roots.
For a fundamental root $\epsilon_{a}$, define the {\em fundamental 
reflection} $s_{a}$ by 
\[
	s_{a}(\alpha):=\alpha-(\alpha,\epsilon_{a})\epsilon_{a}
	\text{ for }\alpha\in \mathbb{Z}^{\mathsf{Q}_{0}}.
\]
The group $W\subset \mathrm{Aut\,}\mathbb{Z}^{\mathsf{Q}_{0}}$ generated
by all fundamental reflections is called {\em Weyl group} of 
the quiver $\mathsf{Q}$. Note that the bilinear form $(\,,\,)$ is 
$W$-invariant. 
Similarly we can define the reflection $r_{a}\colon 
\mathbb{C}^{\mathsf{Q}_{0}}\rightarrow
\mathbb{C}^{\mathsf{Q}_{0}}$ by 
\[r_{a}(\lambda)_{b}:=
\lambda_{b}-(\epsilon_{a},\epsilon_{b})\lambda_{a}
\]
for $\lambda\in \mathbb{C}^{\mathsf{Q}_{0}}$ and  $a,b\in \mathsf{Q}_{0}$.
Define the set of {\em real roots} by 
\[
	\Delta^{\text{re}}:=\bigcup_{w\in W}w(\Pi).
\]

For an element $\alpha=(\alpha_{a})_{a\in \mathsf{Q}_{0}}
\in \mathbb{Z}^{\mathsf{Q}_{0}}$ the {\em support} of $\alpha$
is the set of $\epsilon_{a}$ such that $\alpha_{a}\neq 0$, and 
denoted by $\mathrm{supp\,}(\alpha)$.
We say the support of $\alpha$ is {\em connected} if 
the subquiver consisting of 
the set of vertices $a$ satisfying 
$\epsilon_{a}\in \mathrm{supp\,}(\alpha)$ and 
all arrows joining these vertices, is connected.
Define the {\em fundamental set} $F\subset \mathbb{Z}^{\mathsf{Q}_{0}}$ by
\begin{multline*}
	F:=\\
	\left\{\alpha\in (\mathbb{Z}_{\ge 0})^{\mathsf{Q}_{0}}\backslash\{0\}
        \mid (\alpha,\epsilon)\le 0\text{ for all }\epsilon\in\Pi,\,
        \text{support of }\alpha\text{ is connected}
        \right\}.
\end{multline*}
Then define the set of {\em imaginary roots} by
\[
	\Delta^{\text{im}}:=\bigcup_{w\in W}w(F\cup -F).
\]
Then the {\em root system} is
\[
	\Delta=\Delta^{\text{re}}\cup\Delta^{\text{im}}.
\]
Elements in $\Delta^{+}:=
\Delta\cap(\mathbb{Z}_{\ge 0})^{\mathsf{Q}_{0}}$
are called {\em positive roots}.
    
Now we are ready to see Crawley-Boevey's theorem.
For a fixed $\lambda=(\lambda_{a})\in \mathbb{C}^{\mathsf{Q}_{0}}$,
the set $\Sigma_{\lambda}$ consists of 
the positive roots satisfying 
\begin{enumerate}
              \item $\lambda\cdot \alpha=
		      \sum_{a\in \mathsf{Q}_{0}}\lambda_{a}\alpha_{a}=0$,
              \item if there exists a decomposition 
		      $\alpha=\beta_{1}+\beta_{2}+\cdots+\beta_{r}\ (r\ge 2),$
                  with $\beta_{i}\in \Delta^{+}$ and 
                  $\lambda\cdot \beta_{i}=0$, 
                  then 
		  $p(\alpha)>p(\beta_{1})+p(\beta_{2})+\cdots +p(\beta_{r}).$
\end{enumerate}

\begin{thm}[Crawley-Boevey. Theorem 1.2 in \cite{C1}]\label{CB}
	Let $\mathsf{Q}$ be a finite quiver and $\overline{\mathsf{Q}}$ the double of $\mathsf{Q}$.
        Let us fix a dimension vector 
	$\alpha\in (\mathbb{Z}_{\ge 0})^{\mathsf{Q}_{0}}$ and 
	$\lambda\in \mathbb{C}^{\mathsf{Q}_{0}}.
        $
	Then $\mu_{\alpha}^{-1}(\lambda)^{\text{irr}}\neq \emptyset$ 
        if and only if 
	$\alpha\in \Sigma_{\lambda}$. Furthermore, in this case 
	$\mu_{\alpha}^{-1}(\lambda)$ is an irreducible algebraic variety and 
	$\mu_{\alpha}^{-1}(\lambda)^{\text{irr}}$ is dense in $\mu_{\alpha}^{-1}(\lambda)$.
\end{thm}
This provides the following geometric properties of 
quiver varieties.
\begin{thm}[Crawley-Boevey Corollary 1.4 in \cite{C1}] 
	If $\alpha\in \sigma_{\lambda}$ then 
	the quiver variety $\mathfrak{M}_{\lambda}(\mathsf{Q},\alpha)$ is 
	reduced and irreducible algebraic variety of dimension $2p(\alpha)$.
\end{thm}
Combining these results, we have the following non-emptiness condition of 
regular parts of quiver varieties.
\begin{cor}[Crawley-Boevey \cite{C1}]\label{nonemptyquiv}
	The regular part of quiver variety
	$\mathfrak{M}^{\text{reg}}_{\lambda}(\mathsf{Q},\alpha)$
	is non-empty if and only if $\alpha\in\Sigma_{\lambda}$.
	Furthermore in this case, it is a connected symplectic 
	complex manifold of dimension $2p(\alpha)$.
\end{cor}
\section{A review of Fuchsian cases}\label{reviewoffuchsian}
The necessary and sufficient condition for the existence of a solution  
of the additive Deligne-Simpson 
problem for Fuchsian differential equations is  determined by Crawley-Boevey
in 
\cite{C}.
The strategy is as follows.
For the additive Deligne-Simpson problem for $\mathbf{C}=
(C_{0},C_{1},\ldots,C_{p})$, a collection of conjugacy classes 
in $M(n,\mathbb{C})$, it is shown that 
there exists a quiver $\mathsf{Q}$, 
dimension vector $\alpha$, and complex parameter $\lambda$ such that 
the quiver variety $\mathfrak{M}_{\lambda}^{\text{reg}}(\mathsf{Q},\alpha)$
is isomorphic to the moduli space $\mathfrak{M}(\mathbf{C})$.
Thus Theorem \ref{CB} determines the non-emptiness condition of 
$\mathfrak{M}(\mathbf{C})$ which is equivalent to the solvability of the
additive Deligne-Simpson problem.
We shall recall this correspondence between $\mathfrak{M}(\mathbf{C})$ and 
$\mathfrak{M}^{\text{reg}}_{\lambda}(\mathsf{Q},\alpha)$.

First we construct a representation of a quiver from 
a conjugacy class $C$ of $M(n,\mathbb{C})$.
Let us choose complex numbers $\xi_{1},\ldots,\xi_{d}$
so that 
\begin{equation}\label{ranks}
	\prod_{i=1}^{d}(A-\xi_{i}I_{n})=0
\end{equation}
for all $A\in C$. 
The minimal polynomial of $C$ is an example of this equation.
Set $m_{k}:=\mathrm{rank}\prod_{i=1}^{k}(A-\xi_{i}I_{n})$
for $k=1,\ldots,d$.
Then let us note that these $\xi_{1},\ldots,\xi_{d}$ and $m_{1},\ldots,m_{d}$
characterize $C$.
Namely  $B\in M(n,\mathbb{C})$ is contained in $C$ if and only if 
 $B$ satisfies 
$$\mathrm{rank\,}\prod_{i=1}^{k}(B-\xi_{i}I_{n})=m_{k}$$
for all $k=1,\ldots,d$. 
This observation leads us to the following correspondence
between 
the elements in $C$ and some representations of a quiver.
\begin{prop}[see Crawley-Boevey \cite{C} and also Lemma A.5 in \cite{HY}]
	\label{conjclass}
        Let us fix a conjugacy class $C$ of $M(n,\mathbb{C})$
        and 
        choose $\xi_{1},\ldots,\xi_{d}\in \mathbb{C}$ so that 
	the equation $(\ref{ranks})$ holds for all $A\in C$.
        Set $m_{k}:=\mathrm{rank\,}\prod_{i=1}^{k}(A-\xi_{i}I_{n})$ for 
        $k=1,\ldots,d-1$ and 
	$A\in C$, also set $m_{0}:=n$ and $\mathbf{m}:=(m_{i})_{
	i=0,\ldots,d-1}$.
	Define a quiver $\mathsf{Q}$ as below.
        \[\begin{xy}
              (0,0) *++!U{0} *\cir<4pt>{}="A",
              (10,0) *++!U{1} *\cir<4pt>{}="B",
              (50,0) *++!U{d-1} *\cir<4pt>{}="C",
              \ar@{<-} "A";"B"^{\rho_{1}}
              \ar@{<-} "B";(20,0)^{\rho_{2}}
              \ar@{.} (25,0);(35,0)
              \ar@{<-} (40,0);"C"^{\rho_{d-1}}
        \end{xy}\]
	Also define a subspace of 
$\mathrm{Rep}(\overline{\mathsf{Q}},\mathbf{m})$ by
	\begin{multline*}
		Z:=
	\left\{
	x=(x_{\rho}) \in \mathrm{Rep\,}
	(\overline{\mathsf{Q}},\mathbf{m})\,\Big|\, 
	\right.\\
	\left.
	\begin{array}{c}
		\mu_{\mathbf{m}}(x)_{i}=(\xi_{i}-\xi_{i+1})I_{m_{i}}
        \text{ for all }i=1,\ldots,d-1,\\
        x_{\rho}:\text{ injective},\,
        x_{\rho^{*}}:\text{ surjective}
	\text{ for all }\rho\in \mathsf{Q}_{1},
	\rho^{*}\in \mathsf{Q}_{1}^{*}
	\end{array}
        \right\}.
	\end{multline*}
        Then 
        \begin{align*}
	\Phi_{\xi}\colon \{A&\in C\}\longrightarrow
	Z/
	\prod_{i=1}^{d-1}\mathrm{GL}(m_{i},\mathbb{C})
        \end{align*}
        defined below is  bijective.
	For $A\in C$, we define $(M_{a},\psi_{\rho})_{
	a\in \mathsf{Q}_{0},\rho\in 
\overline{\mathsf{Q}}_{1}}$, a representation of 
	$\overline{\mathsf{Q}}$
	as follows:
        \begin{align*}
        &M_{0}:=\mathbb{C}^{n}, 
        \quad\quad 
        M_{k}:=\mathrm{Im}\prod_{i=1}^{k}(A-\xi_{i}I_{n})
        \text{ for all }k=1,\ldots,d-1,\\
        &\psi_{\rho_{i}}:M_{i}\hookrightarrow M_{i-1}\text{ :
        inclusion},
        \quad\quad\psi_{\rho_{i}^{*}}=(A-\xi_{i})|_{M_{i-1}}.
      	\end{align*}
	Then $\Phi_{\xi}(A)$ is the projection of  
	$(M_{a},\psi_{\rho})$.
     	The inverse map is given by 
      	$$(x_{\rho})_{\rho\in \overline{\mathsf{Q}}_{1}}\mapsto
      	x_{\rho_{1}}x_{\rho_{1}^{*}}+\xi_{1}.$$

     	 Furthermore for any 
	 $x=(x_{\rho})_{\rho\in 
      	 \overline{\mathsf{Q}}_{1}}
      	 \in Z$ and 
      	 any subspace $S\subset \mathbb{C}^{n}$ invariant under
      	 $x_{\rho_{1}}x_{\rho_{1}^{*}}+\xi_{1}$,
      	 there exists a subrepresentation 
      	 $y$ of $x$ such that 
      	 $x=0$ (resp. $N=M$) if and only if $S=0$ (resp. $S=\mathbb{C}^{n}$).
\end{prop}

Let $C_{0},C_{1},\ldots,C_{p}$ be a collection of conjugacy classes 
in $M(n,\mathbb{C})$ and write 
$\mathbf{C}:=(C_{0},C_{1},\ldots,C_{p})$.
As we noted before, a conjugacy class $C$ can be seen as a truncated orbit 
of an HTL normal form of the case $k=1$.
Thus 
\[
	\mathfrak{M}(\mathbf{C}):=\left\{
		(A_{i})_{i=0,1,\ldots,p}\in 
		\prod_{i=0}^{p}C_{i}\,\middle|\, 
		\begin{array}{c}
			(A_{i})_{i=0,\ldots,p}\text{ is irreducible},\\
			\sum_{i=0}^{p}A_{i}=0
		\end{array}
	\right\}\bigg/
	\mathrm{GL}(n,\mathbb{C})
\]
is a moduli space of Fuchsian differential equations or equivalently that of 
meromorphic connections defined in Section \ref{moduli conn}.
Crawely-Boevey obtained a realization of $\mathfrak{M}(\mathbf{C})$ as a 
quiver variety.
\begin{thm}[Crawley-Boevey \cite{C}]\label{one to one Fuchs}
	Let $C_{0},\ldots,C_{p}$ be conjugacy classes of 
        $M(n,\mathbb{C})$. 
        For $i=0,\ldots,p$, choose 
        $\xi_{[i,1]},\ldots,\xi_{[i,d_{i}]}\in \mathbb{C}$ so that 
        \[
            \prod_{j=1}^{d_{i}}(A_{i}-\xi_{[i,j]}I_{n})=0
        \]
        for all $A_{i}\in C_{i}$.
	Let $\xi=(\{\xi_{[i,1]},\ldots,\xi_{[i,d_{i}]}\})_{0\le i\le p}$
	be the collection of ordered sets 
	$\{\xi_{[i,1]},\ldots,\xi_{[i,d_{i}]}\}$.
        Set $m_{0}:=n$ and $m_{[i,j]}
        :=\mathrm{rank\,}\prod_{k=1}^{j}(A_{i}-\xi_{[i,k]}I_{n})$
        for $j=1,\ldots,d_{i}-1$.
	Consider the following quiver $\mathsf{Q}$.
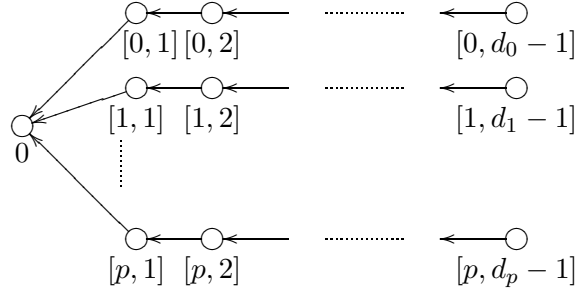
        \[
            \begin{xy}
            (-10,0) *++!U{0} *\cir<4pt>{}="A",
            (5,15) *++!U!L(0.3){[0,1]} *\cir<4pt>{}="B",
            (15,15) *++!U{[0,2]} *\cir<4pt>{}="C",
            (55,15) *++!U{[0,d_{0}-1]} *\cir<4pt>{}="D",
            (5,5) *++!U{[1,1]} *\cir<4pt>{}="E",
            (15,5) *++!U{[1,2]} *\cir<4pt>{}="F",
            (55,5) *++!U{[1,d_{1}-1]} *\cir<4pt>{}="G",
            (5,-15) *++!U{[p,1]} *\cir<4pt>{}="H",
            (15,-15) *++!U{[p,2]} *\cir<4pt>{}="I",
            (55,-15) *++!U{[p,d_{p}-1]} *\cir<4pt>{}="J",
            \ar@{<-} "A";"B"
            \ar@{<-} "A";"E"
            \ar@{<-} "A";"H"
            \ar@{<-} "B";"C"
            \ar@{<-} "C";(25,15)
            \ar@{.} (30,15);(40,15)
            \ar@{<-} (45,15);"D"
            \ar@{<-} "E";"F"
            \ar@{<-} "F";(25,5)
            \ar@{.} (30,5);(40,5)
            \ar@{<-} (45,5);"G"
            \ar@{<-} "H";"I"
            \ar@{<-} "I";(25,-15)
            \ar@{.} (30,-15);(40,-15)
            \ar@{<-} (45,-15);"J"
            \ar@{.} (3,-2);(3,-8)
        \end{xy}
        \]
	Define  $\alpha=(\alpha_{a})_{a\in \mathsf{Q}_{0}}\in
	(\mathbb{Z}_{\ge 0})^{\mathsf{Q}_{0}}$ 
        by $\alpha_{0}:=m_{0}$ and $\alpha_{[i,j]}:=m_{[i,j]}$
        for $i=0,\ldots,p$, $j=1,\ldots,d_{i}-1$.
	Define $\lambda=(\lambda_{a})_{a\in \mathsf{Q}_{0}}\in 
	\mathbb{C}^{\mathsf{Q}_{0}}$ by
	$\lambda_{0}:=-\sum_{i=0}^{p}\xi_{[i,1]}$ and
        $\lambda_{[i,j]}:=\xi_{[i,j]}-\xi_{[i,j+1]}$
        for $i=0,\ldots,p$, $j=1,\ldots,d_{i}-1$.

        Then there exists a bijection 
	\[		\Phi_{\xi}\colon
		\mathfrak{M}(\mathbf{C})\longrightarrow 
		\mathfrak{M}^{\text{reg}}_{\lambda}(\mathsf{Q},\alpha).
	\]
\end{thm}

Thus Theorem \ref{CB} solves additive Deligne-Simpson problem.
\begin{thm}[Crawley-Boevey \cite{C}]
	Let $C_{0},\ldots,C_{p}$ be conjugacy classes of 
        $M(n,\mathbb{C})$. 
	Let us choose the quiver $\mathsf{Q}$,
	$\alpha\in (\mathbb{Z}_{\ge 0})^{\mathsf{Q}_{0}}$
	and $\lambda\in \mathbb{C}^{\mathsf{Q}_{0}}$ as Theorem \ref{one to one 
        Fuchs}.
        Then the additive Deligne-Simpson problem for 
        $C_{0},\ldots,C_{p}$ has a solution if and only if 
	$\alpha\in \Sigma_{\lambda}$.
\end{thm}
\section{Moduli spaces of meromorphic connections and quiver varieties}
In the previous section, we saw that moduli spaces of Fuchisan differential 
equations are isomorphic to quiver varieties and moreover 
the solvability of the additive Delinge-Simpson problem for Fuchsian 
differential equations is determined through these isomorphisms. 
In this section, we shall give a generalization of this correspondence.
Namely we shall consider  
a collection of HTL normal forms $\mathbf{B}=(B_{0},B_{1},\ldots,B_{p})$,
and give a correspondence between the moduli space $\mathfrak{M}(\mathbf{B})$
and a quiver variety.
This is first done by Boalch in \cite{Boarx} when the orders of HTL normal
 forms $k_{i}=-
 \mathrm{ord}(B_{i})$ for $i=0,\ldots,p$ satisfy
 \[
	 k_{0}\le 3\text{
	 and }k_{1}=\cdots =k_{p}=1.
\]
This result is generalized for arbitrary $k_{0}$ by Yamakawa and the author 
in \cite{HY}.
Thus when the number of unramified irregular singular points is at most one,
it is already known that there exist isomorphisms between moduli spaces of meromorphic connections and 
quiver varieties.

However in Introduction of \cite{Boarx} Boalch suggested
that  moduli spaces of meromorphic connections with more than two unramified 
irregular singular points might not be isomorphic to quiver varieties and 
gave an example.

Therefore it may not be expected to obtain isomorphisms between arbitrary
$\mathfrak{M}(\mathbf{B})$ and quiver varieties.
Based on these previous results, for an arbitrary $\mathfrak{M}(\mathbf{B})$, 
we shall construct an injective map from 
$\mathfrak{M}(\mathbf{B})$ into a quiver variety which becomes an isomorphism
if and only if the number of unramified irregular singular points are 
less than or equal to one. 

\subsection{A preliminary example: 
	Differential equations with poles of order 2 and representations
of quivers}\label{Section2}
As we noted above, the correspondence between the moduli spaces of 
meromorphic connections and quiver varieties is already known if 
$k_{1}=\cdots =k_{p}=1$ and $k_{0}$ is an arbitrary positive integer.
Before going to general cases, we shall see the first nontrivial case
$k_{0}=k_{1}=\cdots =k_{p}=2$.

\subsubsection{Splitting lemma and truncated orbits}
Let $B\in \mathfrak{g}^{*}_{2}$ be an HTL normal form,
\[
    B=\mathrm{diag}\left(c_{1}I_{n_{1}}z^{-2}+R_{1}z^{-1},
   \ldots,
   c_{m}I_{n_{m}}z^{-2}+R_{m}z^{-1}
    \right).
\]
Here $R_{i}\in M(n_{i},\mathbb{C})$ and 
$c_{i}\in \mathbb{C}$, $i=0,\ldots,p$ satisfying
$c_{i}\neq c_{j}$ if 
$i\neq j$.

Let us put $B_{\text{irr}}:=\mathrm{diag}\left(
c_{1}I_{n_{1}},\ldots,c_{m}I_{n_{m}}\right)$ 
and denote by  $V(c_{i})\subset \mathbb{C}^{n}$ 
the eigenspace of $B_{\text{irr}}$ for
each eigenvalue $c_{i}$, $i=1,\ldots,m$.
For each $X\in M(n,\mathbb{C})$, $X_{i,j}$ denotes
the $\mathrm{Hom}_{\mathbb{C}}(V(c_{j}),V(c_{i}))$-component
of $X$.

Then for the $G_{2}=\mathrm{GL}(n,\mathbb{C}[\![z]\!]/z^{2}
\mathbb{C}[\![z]\!])$-orbit of $B$, denoted by $\mathcal{O}_{B}$,
we have the following lemma which  is a direct consequence of 
the splitting lemma (see the section 3.2 in 
\cite{Ba} or section 2.3 in \cite{HY} for example) .
\begin{lem}\label{trunc}
	Let $B\in \mathfrak{g}^{*}_{2}$ be the HTL normal form 
	as above.
        Then $\mathcal{O}_{B}$ consists of $A(x)=\sum_{i=1}^{2}A_{i}x^{-i}
    	\in \mathfrak{g}^{*}_{2}$ satisfying that 
    	there exists $G\in \mathrm{GL}(n,\mathbb{C})$ such that 
	\[G^{-1}A_{2}G=B_{\text{irr}}\quad{ and }\quad 
	(G^{-1}A_{1}G)_{i,i}\in C_{R_i}
	\]
	where $C_{R_i}$ are 
	conjugacy classes of $R_{i}$ for $i=1,\ldots,m$.
	Moreover if $G_{1},G_{2}\in \mathrm{GL}(n,\mathbb{C})$ satisfy
	$G_{i}^{-1}A_{2}G_{i}=B_{\text{irr}}$, $i=1,2$, then
	$G_{2}^{-1}G_{1}=\mathrm{diag}(h_{1},\ldots,h_{m})$ where 
	$h_{i}\in \mathrm{GL}(n_{i},\mathbb{C})$ for $i=1,\ldots,m$.
\end{lem}

From this lemma we have the following one to one correspondence.
\begin{multline}\label{rank2orbit}
	\mathcal{O}_{B}\longrightarrow\\
	\left\{
		(G,A_{1})\in \mathrm{GL}(n,\mathbb{C})\times 
		M(n,\mathbb{C})\,\middle|\,\begin{array}{l}
		(G^{-1}A_{1}G)_{i,i}\in C_{R_i}\\\text{ for all }
		i=1,\ldots,m
	\end{array}
	\right\}
	/\prod_{i=1}^{m}\mathrm{GL}(n_{i},\mathbb{C}).
\end{multline}
Here $\prod_{i=1}^{m}\mathrm{GL}(n_{i},\mathbb{C})$ acts 
on $\mathrm{GL}(n,\mathbb{C})\times M(n,\mathbb{C})$ by 
\[
	\begin{array}{ccc}
	\prod_{i=1}^{m}\mathrm{GL}(n_{i},\mathbb{C})\times
	\left(\mathrm{GL}(n,\mathbb{C})\times M(n,\mathbb{C})\right)&
	\longrightarrow&\mathrm{GL}(n,\mathbb{C})\times M(n,\mathbb{C})\\
	\left((h_{1},\ldots,h_{m}),(G,A) \right)&
	\longmapsto& (G\cdot\mathrm{diag}(h_{1},\ldots,h_{m}),A).
\end{array}
\]
The inverse map is induced by sending $(G,A_{1})$ 
 to
$GB_{\text{irr}}G^{-1}x^{-2}+A_{1}x^{-1}\in 
\mathcal{O}_{B}$.
\begin{rem}\normalfont
	Let us recall that $T^{*}\mathrm{GL}(n,\mathbb{C})
	\cong \mathrm{GL}(n,\mathbb{C})\times M(n,\mathbb{C})$. Then
	the above correspondence can be seen as a special case of 
	the identification of  $G_{k}$-orbits of HTL normal forms 
	and a symplectic 
	reductions of the extended orbits given by Boalch (see Lemma 2.3 in 
	\cite{Boa1}).
\end{rem}

\subsubsection{Moduli spaces without irreducibility and 
quivers}
Under the above observation, let us consider the relation 
between $\mathfrak{M}(\mathbf{B})$ and a quiver variety.
Let $B_{0},\ldots,B_p\in \mathfrak{g}^{*}_{2}$ be 
HTL normal forms written by
\[
	B_i=\mathrm{diag}\left(c_{[i,1]}I_{n_{[i,1]}}z^{-2}
		+R_{[i,1]}z^{-1},
   \ldots,
   c_{[i,m_i]}I_{n_{[i,m_i]}}z^{-2}+R_{[i,m_i]}z^{-1}
    \right).
\]

Let $V(c_{[i,j]})\subset\mathbb{C}^{n}$ be the eigenspace
of $(B_i)_{\text{irr}}$ for each
eigenvalue $c_{[i,j]}$,
$i=0,\ldots,p$, $j=1,\ldots,m_{i}$.
Let  $X_{[i,j],[i'j']}$ be the
$\mathrm{Hom}_{\mathbb{C}}(V(c_{[i',j']}),V(c_{[i,j]}))$-component
of $X\in M(n,\mathbb{C})$.
We may write $X=\left(X_{[i,j],[i'j']}\right)
_{\substack{1\le j\le m_{i}\\1
\le j'\le m_{i'}}}$.

First let us consider the moduli space $\overline{\mathfrak{M}(\mathbf{B})}$ 
without the irreducibility, namely,
\[
	\overline{\mathfrak{M}(\mathbf{B})}:=
	\left\{
		(A_{i}(z))_{i=0,\ldots,p}\in\prod_{i=0}^{p}
		\mathcal{O}_{B_{i}}\,\middle|\,
		\sum_{i=0}^{p}\mathrm{Res}A_{i}(z)=0
	\right\}\big/\mathrm{GL}(n,\mathbb{C})
\]
which is isomorphic to 
\begin{multline*}
	\left\{
		(G_{i},A_{i})_{i=0,\ldots,p}\in 
		\prod_{i=0}^{p}\mathrm{GL}(n,\mathbb{C})\times
		M(n,\mathbb{C})\,
		\middle|\,
		\begin{array}{l}
			\mathrm{(i)}\
			(G_{i}^{-1}A_{i}G_{i})_{j,j}\in C_{R_{[i,j]}}\\
			\text{ for all }i=0,\ldots,p\text{ and }\\
			j=1,\ldots,m_{i}\\
			\mathrm{(ii)}\ 
			\sum_{i=0}^pA_{i}=0\\
			\mathrm{(iii)}\ G_{0}=I_{n}\\
		\end{array}
	\right\}\\
	\big/ \prod_{i=0}^{p}\prod_{j=1}^{m_{i}}
	\mathrm{GL}(n_{[i,j]},\mathbb{C})
\end{multline*}
from the above identification (\ref{rank2orbit}).
Here the condition $\mathrm{(ii)}$ comes from the condition 
$\sum_{i=0}^{p}\mathrm{Res}A_{i}(z)=0$ 
and $\mathrm{(iii)}$ comes from taking the quotient 
under the $\mathrm{GL}(n,\mathbb{C})$-action.
We shall give a realization of $\overline{\mathfrak{M}(\mathbf{B})}$ as 
a representation space of a quiver as follows.

\noindent\underline{\textbf{Step 1.}}

Let us consider the quiver $\mathsf{Q}^{(1)}$ defined as follows.
The set of vertices  is
\[\mathsf{Q}_{0}^{(1)}:
=\{0,\ldots,p\}.\]
The set of arrows is
\[
	\mathsf{Q}^{(1)}_{1}:=\left\{\rho^{[0]}_{[i]}\colon 0\rightarrow
    i\,\middle|\, i=1,\ldots,p
\right\}.
\]
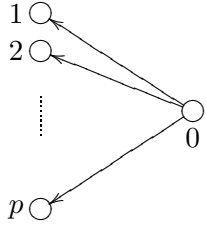
\[
	\begin{xy}
		*++!U{0}*\cir<4pt>{}="A",
		(-20,13)*++!R{1}*\cir<4pt>{}="B",
		(-20,8)*++!R{2}*\cir<4pt>{}="C",
		(-20,2)="D",
		(-20,-3)="F",
		(-20,-13)*++!R{p}*\cir<4pt>{}="E",
		\ar@{..} "F";"D",
		\ar@{->} "A";"B",
		\ar@{->} "A";"C",
		\ar@{->} "A";"E"
	\end{xy}
\]

Fix a dimension vector $\alpha^{(1)}:=(\alpha_{i})_{i=0,\ldots,p}$ so that 
$\alpha_{i}:=n$ for all $i=0,\ldots,p$.
Then we have a bijection,
\begin{multline*}
	\left\{
	(G_{i},A_{i})_{i=0,\ldots,p}\in 
		\prod_{i=0}^{p}\mathrm{GL}(n,\mathbb{C})\times
		M(n,\mathbb{C})\,\middle|\,
		G_{0}=I_{n},\sum_{i=0}^{p}A_{i}=0
	\right\}\longrightarrow\\
	\left\{
		x=(x_{\rho})_{\rho \in \overline{\mathsf{Q^{(1)}}}_{1}}
		\in \mathrm{Rep}(\overline{\mathsf{Q^{(1)}}},\alpha^{(1)})\,
		\middle|\,
		\begin{array}{c}
		x_{\rho^{[0]}_{[i]}}\in \mathrm{GL}(n,\mathbb{C})\\
		\text{ for all }
		i=1,\ldots,p
	\end{array}
	\right\}
\end{multline*}
by setting $x_{\rho^{[0]}_{[i]}}:=G_{i}^{-1}$ and 
$x_{(\rho^{[0]}_{[i]})^{*}}:=A_{i}G_{i}$ for all $i=1,\ldots,p$.
Let us note that from $x=(x_{\rho})$ in the target space,
setting \begin{align*}
	&G_{i}:=x_{\rho^{[0]}_{[i]}}^{-1},&
	&A_{i}:=x_{(\rho^{[0]}_{[i]})^{*}}x_{\rho^{[0]}_{[i]}}
\end{align*}
for $i=1,\ldots,p$ and 
\[
	A_{0}:=-\sum_{i=1^{p}}x_{(\rho^{[0]}_{[i]})^{*}}x_{\rho^{[0]}_{[i]}}
	=\mu_{\alpha^{(1)}}(x)_{0},
\]
we obtain the inverse map.

\noindent\underline{\textbf{Step 2.}}

In Step 1, we could associate  representations of a quiver to 
$(G_{i},A_{i})_{i=0,\ldots,p}\in \prod_{i=0}^{p}\mathrm{GL}(n,\mathbb{C})
\times M(n,\mathbb{C})$ satisfying the 
conditions $\mathrm{(ii)}\ \sum_{i=0}^{p}A_{i}=0$ and $\mathrm{(iii)}\ 
G_{0}=I$.
However to obtain the one to one correspondence with $\overline{\mathfrak{M}}
(\mathbf{B})$, we need one more condition 
$\mathrm{(i)}\ (G_{i}^{-1}A_{i}G_{i})_{j,j}\in 
C_{R_{[i,j]}}$ for $i=0,\ldots,p$, $j=1,\ldots,m_{i}$. 
Let us recall that  
\begin{align*}
	G^{-1}_{i}A_{i}G_{i}=x_{\rho^{[0]}_{[i]}}x_{(\rho^{[0]}_{[i]})^{*}}=
	\mu_{\alpha^{(1)}}(x)_{i}
\end{align*}
for $i=1,\ldots,p$ and 
\begin{align*}
	G_{0}^{-1}A_{0}G_{0}=A_{0}=-\sum_{i=0}^{p}A_{i}
	=-\sum_{i=1^{p}}x_{(\rho^{[0]}_{[i]})^{*}}x_{\rho^{[0]}_{[i]}}
=\mu_{\alpha^{(1)}}(x)_{0}
\end{align*}
for $(G_{i},A_{i})_{i=0,\ldots,p}$ in the domain of the isomorphism in 
Step 1 and its image 
$x\in \mathrm{Rep}(\overline{\mathsf{Q^{(1)}}},\alpha^{(1)})$.
To obtain block diagonal components $(G_{i}^{-1}A_{i}G_{i})_{j,j}$ as 
images of the moment map, we shall break up the vertex 
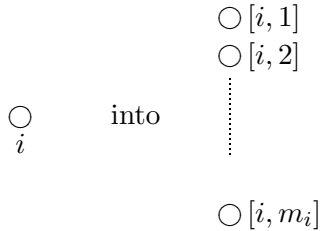
\[
\begin{xy}*++!U{i}*\cir<4pt>{}\end{xy}
	\quad\quad\text{ into} \quad\quad
	\begin{xy}
		(0,13)	*++!L{[i,1]}*\cir<4pt>{},
		(0,8) *++!L{[i,2]}*\cir<4pt>{},
		(0,-13) *++!L{[i,m_{i}]}*\cir<4pt>{},
		\ar@{..} (0,5);(0,-5)
\end{xy}
\]for each $i=0,\ldots,p$ and define the following 
quiver $\mathsf{Q^{(2)}}$.

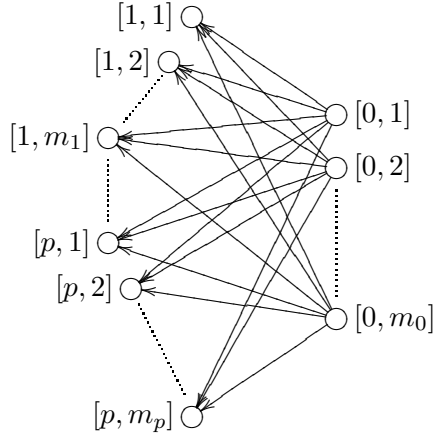
\[	\begin{xy}
		(0,-7)*++!L{[0,2]}*\cir<4pt>{}="A";
		(0,-27) *++!L{[0,m_{0}]}*{}*\cir<4pt>{}="B",
		(0,0)*++!L{[0,1]}*\cir<4pt>{}="C",
		(-19,13)*++!R{[1,1]}*\cir<4pt>{}="D",	
		(-22,7)*++!R{[1,2]}*\cir<4pt>{}="E",
		(-30,-3)*++!R{[1,m_{1}]}*\cir<4pt>{}="H",
		(-27,-23)*++!R{[p,2]}*\cir<4pt>{}="F",	
		(-30,-17)*++!R{[p,1]}*\cir<4pt>{}="G",
		(-19,-40)*++!R{[p,m_{p}]}*\cir<4pt>{}="I",
		\ar@{..}(-30, -6);(-30, -14),
		\ar@{..} (0,-10);(0,-24),
		\ar@{..} (-23,5);(-28,-1),
		\ar@{..} (-26,-25);(-20,-37),
		\ar@{->} "A";"D",
		\ar@{->} "A";"E",
		\ar@{->} "A";"F",
		\ar@{->} "A";"G",
		\ar@{->} "A";"H",
		\ar@{->} "A";"I",
		\ar@{->} "B";"D",
		\ar@{->} "B";"E",
		\ar@{->} "B";"F",
		\ar@{->} "B";"G",
		\ar@{->} "B";"H",
		\ar@{->} "B";"I",
		\ar@{->} "C";"D",
		\ar@{->} "C";"E",
		\ar@{->} "C";"F",
		\ar@{->} "C";"G",
		\ar@{->} "C";"H",
		\ar@{->} "C";"I",
    	\end{xy}
\]
Namely, the set of vertices is 
\[\mathsf{Q^{(2)}}_{0}:=
\{[i,j]\mid i=0,\ldots,p,\,j=1,\ldots,m_{i}\}.
\]
The set of arrows is 
\[
	\mathsf{Q^{(2)}}_{1}:=
\left\{
	\rho^{[0,j]}_{[i,j']}\colon
	[0,j]\rightarrow [i,j']\,\middle|\,
	\begin{array}{l}
	j=1,\ldots,m_{0},\\
	i=1,\ldots,p,\\
	j'=1,\ldots,m_{i}
\end{array}
\right\}.
\]

Define $\alpha^{(2)}
=(\alpha^{(2)}_{a})_{a\in \mathsf{Q}_{0}}\in 
\mathbb{Z}^{\mathsf{Q}_{0}}$ by $\alpha^{(2)}_{[i,j]}
:=\mathrm{dim}_{\mathbb{C}}V(c_{[i,j]})$, 
$i=0,\ldots,p$, $j=1,\ldots,m_{i}$.
Then we have a bijection from $\overline{\mathfrak{M}(\mathbf{B})}$ 
to an open subset of 
$\mathrm{Rep}(\overline{\mathsf{Q}^{(2)}},\alpha^{(2)})
/\mathbf{G}(\alpha^{(2)})$.
\begin{prop}\label{trunc order2}
	We use the same notation as above. Then there exists a bijection 
	\begin{multline*}
		\Phi\colon\overline{\mathfrak{M}(\mathbf{B})}
        \longrightarrow\\
	\left\{\begin{aligned}
			x&=(x_{\rho})_{\rho\in 
	\overline{\mathsf{Q^{(2)}}}_{1}}\\
	&\in \mathrm{Rep\,}(\overline{\mathsf{Q^{(2)}}},\alpha^{(2)})
	\end{aligned}
        \,\middle|\, 
	\begin{array}{l}
	\mathrm{det}\left(x_{\rho^{[0,j]}_{[i,j']}}
        \right)_{
        \substack{1\le j\le m_{0}\\1\le j'\le m_{i}}}
        \neq 0\\
	\text{for all }i=1,\ldots,p,
        \\
	\mu_{\alpha^{(2)}}(x)_{[i,j]}\in C_{R_{[i,j]}},\,
	[i,j]\in \mathsf{Q^{(2)}}_{0}
	\end{array}
\right\}/\mathbf{G}(\alpha^{(2)}).
	\end{multline*}
\end{prop}
\begin{proof}
	It suffices to show that there is a bijection from 
\begin{multline*}
	\left\{
		(G_{i},A_{i})_{i=0,\ldots,p}\in 
		\prod_{i=0}^{p}\mathrm{GL}(n,\mathbb{C})\times
		M(n,\mathbb{C})\,
		\middle|\,
		\begin{array}{l}
			\mathrm{(i)}\
			(G_{i}^{-1}A_{i}G_{i})_{j,j}\in C_{R_{[i,j]}}\\
			\text{ for all }i=0,\ldots,p\text{ and }\\
			j=1,\ldots,m_{i}\\
			\mathrm{(ii)}\ 
			\sum_{i=0}^pA_{i}=0\\
			\mathrm{(iii)}\ G_{0}=I_{n}\\
		\end{array}
	\right\}\\
	\big/ \prod_{i=0}^{p}\prod_{j=1}^{m_{i}}
	\mathrm{GL}(n_{[i,j]},\mathbb{C})
\end{multline*}
to the target space of the above map.
Let $(G_{i},A_{i})_{i=0,\ldots,p}$ be a representative of an 
element in this space. Then we define $x\in \mathrm{Rep}(\overline{
\mathsf{Q^{(2)}}},\alpha^{(2)})$ as follows.
	\item \begin{align*}
			x_{\rho^{[0,j]}_{[i,j']}}
			&=(G_{i}^{-1})_{[i,j'],[0,j]},&
        x_{\left(\rho^{[0,j]}_{[i,j']}\right)^{*}}
	&=\left(A_{i}G_{i}\right)_{[0,j],[i,j']},
	\end{align*}
        for $j=1,\ldots,m_{0}$, $i=1,\ldots,p$, $j'=1,\ldots,m_{i}$.
	Then 
	\[
		\mu_{\alpha^{(2)}}(x)_{[i,j']}=\sum_{j=1}^{m_{0}}
		x_{\rho^{[0,j]}_{[i,j']}}
x_{\left(\rho^{[0,j]}_{[i,j']}\right)^{*}}=
(G_{i}^{-1}A_{i}G_{i})_{j',j'}\in C_{R_{[i,j']}}
	\]
	for $i=1,\ldots,p$ and $j'=1,\ldots,m_{i}$. Also 
\[
	\mu_{\alpha^{(2)}}(x)_{[0,j]}=-
	\sum_{i=1}^{p}\sum_{j'=1}^{m_{i}}
	x_{\left(\rho^{[0,j]}_{[i,j']}\right)^{*}}
x_{\rho^{[0,j]}_{[i,j']}}
=-\sum_{i=1}^{p}(A_{i})_{j,j}=(A_{0})_{j,j}\in C_{R_{[0,j]}}
\]
for $j=1,\ldots,m_{0}$.
	Since this correspondence is $\prod_{i=0}^{p}\prod_{j=1}^{m_{i}}
	\mathrm{GL}(n_{[i,j]},\mathbb{C})\cong 
	\mathbf{G}(\alpha^{(2)})$-equivariant, we have the well-defined map.
	The inverse maps can be defined as we saw in Step 1. 
	Thus it is bijective.
\end{proof}
\subsubsection{Moduli space and quiver varieties}
Now we are ready to consider $\mathfrak{M}(\mathbf{B})$. 
As we will see later, the irreducibility of differential equations does not 
coincide with the irreducibility of representations of quiver under the 
bijection in Proposition \ref{trunc order2}. Thus we shall introduce 
a weaker condition which is called $\mathcal{L}$-irreducibility in this paper.

Let us define a sublattice $\tilde{\mathcal{L}}$ of 
$\mathbb{Z}^{\mathsf{Q^{(2)}}_{0}}$ by 
	\[
		\tilde{\mathcal{L}}:=
	\left\{
		\beta=(\beta_{a})_{a\in \mathsf{Q^{(2)}}_{0}}
		\in \mathbb{Z}^{\mathsf{Q^{(2)}}_{0}}\,
		\middle|\,
		\sum_{j=1}^{m_{0}}\beta_{[0,j]}=
		\sum_{j=1}^{m_{i}}\beta_{[i,j]}
		\text{ for all }i=1,\ldots,p
	\right\}
\]
\begin{df}[$\tilde{\mathcal{L}}$-irreducible]\label{ellirred}\normalfont
	An element in 
	\[\left\{\begin{aligned}
			x&=(x_{\rho})_{\rho\in 
	\overline{\mathsf{Q^{(2)}}}_{1}}\\
	&\in \mathrm{Rep\,}(\overline{\mathsf{Q^{(2)}}},\alpha^{(2)})
	\end{aligned}
        \,\middle|\, 
	\begin{array}{l}
	\mathrm{det}\left(x_{\rho^{[0,j]}_{[i,j']}}
        \right)_{
        \substack{1\le j\le m_{0}\\1\le j'\le m_{i}}}
        \neq 0\\
	\text{for all }i=1,\ldots,p,
        \\
	\mu_{\alpha^{(2)}}(x)_{[i,j]}\in C_{R_{[i,j]}},\,
	[i,j]\in \mathsf{Q^{(2)}}_{0}
	\end{array}
\right\}/\mathbf{G}(\alpha)\]
is said to be {\em $\tilde{\mathcal{L}}$-irreducible}, 
if it has no proper subrepresentation $y$ with the dimension 
vector $\mathbf{dim}(y)\in \tilde{\mathcal{L}}$ other than $\{0\}.$
\end{df}

\begin{prop}\label{preisom}
	Let $\mathbf{A}=(\sum_{j=1}^{2}A^{(i)}_{j}z^{-j})
	\in \prod_{i=0}^{p}\mathcal{O}_{B_{i}}$ 
	with $\sum_{i=0}^{p}A^{(i)}_{1}=0$ and 
	$x\in \mathrm{Rep}(\overline{\mathsf{Q^{(2)}}},\alpha^{(2)})$ 
	be the corresponding elements
under the map 	$\Phi$ in Proposition \ref{trunc order2}.
	If $\mathbf{A}$ 
	is irreducible, 
	then $x$ is $\tilde{\mathcal{L}}$-irreducible and vice versa.
\end{prop}
\begin{proof}
	Suppose that $\mathbf{A}$ has a nontrivial 
	invariant subspace $W\subsetneqq
	\mathbb{C}^{n}$, i.e., $W$ is invariant under all $A^{(i)}_{j}$.
	Set $W^{(i)}:=(x_{\rho^{[0,j]}_{[i,j']}})_{
	\substack{1\le j\le m_{0}\\1\le j'\le m_{i}}}W
	\cong W$ for $i=1,\ldots,p$ and $W^{(0)}:=W$.
	Also set
	\begin{align*}
		\tilde{V}_{[i,j]}&:=W^{(i)}\cap V(c_{[i,j]}), &
		\tilde{x}_{\rho^{[0,j]}_{[i,j']}}=&
		x_{\rho^{[0,j]}_{[i,j']}}|_{W^{(0)}},\\
		\tilde{x}_{\left(\rho^{[0,j]}_{[i,j']}
		\right)^{*}}&=
		x_{\left(\rho^{[0,j]}_{[i,j']}
		\right)^{*}}|_{W^{(i)}}.
	\end{align*}
	Then $\tilde{x}=(\tilde{V}_{a},
	\tilde{x}_{\rho})_{a\in \mathsf{Q^{(2)}}_{0},
\rho\in \overline{\mathsf{Q^{(2)}}}_{1}}$ defines
	a subrepresentation of $x$.
	Since $W$ is $A^{(i)}_{2}$-invariant,
	$W^{(i)}$ is $G_{i}^{-1}A^{(i)}_{2}G_{i}=(B_{i})_{\text{irr}}$-
	invariant. 
	Thus we have $W^{(i)}=\oplus_{j=1}^{m_{i}}\tilde{V}_{[i,j]}$,
	which shows that $\sum_{j=1}^{m_{0}}\mathrm{dim}_{\mathbb{C}}
	\tilde{V}_{[0,j]}=\cdots=
	\sum_{j=1}^{m_{p}}\mathrm{dim}_{\mathbb{C}}\tilde{V}_{[p,j]}$.
	Finally we need to check that 
	$\tilde{x}_{(\rho^{[0,j]}_{[i,j']})^{*}}
	(\tilde{V}_{[i,j']})\subset \tilde{V}_{[0,j]}$.
	To show this, it suffices to see that 
	$(\tilde{x}_{(\rho^{[0,j]}_{[i,j']})^{*}})_{
	\substack{1\le j\le m_{0}\\1\le j'\le m_{i}}}
	W^{(i)}\subset W$,
	which follows from the fact
	that 
	\begin{align*}
		(\tilde{x}_{(\rho^{[0,j]}_{[i,j']})^{*}})_{
	\substack{1\le j\le m_{0}\\1\le j'\le m_{i}}}
	W^{(i)}&=(x_{(\rho^{[0,j]}_{[i,j']})^{*}})_{
	\substack{1\le j\le m_{0}\\1\le j'\le m_{i}}}
	W^{(i)}\\
	&=(x_{(\rho^{[0,j]}_{[i,j']})^{*}})_{
	\substack{1\le j\le m_{0}\\1\le j'\le m_{i}}}
(x_{\rho^{[0,j]}_{[i,j']}})_{
	\substack{1\le j\le m_{0}\\1\le j'\le m_{i}}}W\\
	&=
	A^{(i)}_{1}W\subset W.
\end{align*}

	Conversely suppose that $x$ has a nontrivial proper subrepresentation 
	$\tilde{x}=(\tilde{V}_{a},\tilde{x}_{\rho})$
	satisfying 
	$\sum_{j=1}^{m_{0}}\mathrm{dim}_{\mathbb{C}}\tilde{V}_{[0,j]}=\cdots
	=\sum_{j=1}^{m_{p}}\mathrm{dim}_{\mathbb{C}}\tilde{V}_{[p,j]}$.
	Then $W=\bigoplus_{j=1}^{m^{(0)}}\tilde{V}_{[0,j]}$ is an 
	$\mathbf{A}$-invariant subspace.
	Indeed $W$ is $(A_{1}^{(0)},A_{2}^{(0)})$-invariant. Also for 
	$i=1,\ldots,p$,
	set $W^{(i)}:=(x_{\rho^{[0,j]}_{[i,j']}})_{
	\substack{1\le j\le m_{0}\\1\le j'\le m_{i}}}W
	\subset \oplus_{j=1}^{m_{i}}\tilde{V}_{[i,j]}$.
	Then we have 
	\[
		\sum_{j=1}^{m_{i}}\mathrm{dim}_{\mathbb{C}}
	\tilde{V}_{[i,j]}
	=\sum_{j=1}^{m_{0}}\mathrm{dim}_{\mathbb{C}}\tilde{V}_{[0,j]}
	=\mathrm{dim_{\mathbb{C}}}W=\mathrm{dim}_{\mathbb{C}}W^{(i)},
\]
	which implies that 
	$W^{(i)}=\oplus_{j=1}^{m_{i}}\tilde{V}_{[i,j]}.$
	Thus since 
	\[(x_{\rho^{[0,j]}_{[i,j']}})_{
	\substack{1\le j\le m_{0}\\1\le j'\le m_{i}}}
	(x_{(\rho^{[0,j]}_{[i,j']})^{*}})_{
	\substack{1\le j\le m_{0}\\1\le j'\le m_{i}}}
	=G_{i}^{-1}A^{(i)}_{1}G_{i},
\] $W^{(i)}$ is $G_{i}^{-1}A^{(i)}_{1}G_{i}
$-invariant,
	which shows that 
	$W=G_{i}W^{(i)}$ is $(A^{(i)}_{1},A^{(i)}_{2})$-invariant
	for each $i=1,\ldots,p$. 
\end{proof}

Finally let us give a realization of $\mathfrak{M}(\mathbf{B})$ as a 
subset of a quiver variety $\mathfrak{M}_{\lambda}(\mathsf{Q},\alpha)$
defined as below.

For each $i=0,\ldots,p$ and $j=1,\ldots,m_{i}$, we choose 
$\xi_{[i,j,k]}\in \mathbb{C}$, $k=1,\ldots,e_{[i,j]}$ so that
\[
	\prod_{k=1}^{e_{[i,j]}}(R_{[i,j]}-\xi_{[i,j,k]}I_{n_{[i,j]}})=0.
\]
Then the quiver $\mathsf{Q}$ is defined by the set of vertices 
\[
	\mathsf{Q}_{0}:=\left\{[i,j,k]\,\middle|\,
		\begin{array}{l}
			i=0,\ldots,p,\\
			j=1,\ldots,m_{i},\\
			k=1,\ldots,e_{[i,j]}-1
		\end{array}
	\right\}
\]
and the set of arrows
\begin{align*}
	\mathsf{Q}_{1}:=&
	\left\{
		\rho^{[0,j]}_{[i,j']}\colon
		[0,j]\rightarrow [i,j']\,\middle|\,
		\begin{array}{l}
			j=1,\ldots,m_{0},\\
			i=1,\ldots,p,\\
			j'=1,\ldots,m_{i}
		\end{array}
	\right\}\\
	&\bigsqcup
	\left\{
		\rho_{[i,j,k]}\colon
		[i,j,k]\rightarrow [i,j,k-1]\,
		\middle|\,
		\begin{array}{l}
		i=0,\ldots,p,\\
		j=1,\ldots,m_{i},\\
		k=1,\ldots,e_{[i,j]}-1
		\end{array}
	\right\}.
\end{align*}
Here we set $[i,j,0]:=[i,j]$.
Define a dimension vector 
$\alpha=(\alpha_{a})_{a\in \mathsf{Q}_{0}}$
by 
\begin{align*}
	\alpha_{[i,j]}&:=n_{[i,j]}&
	\alpha_{[i,j,k]}:=\mathrm{rank}\prod_{l=1}^{k}(R_{[i,j]}-
	\xi_{[i,j,l]}I_{n_{[i,j]}}).
\end{align*}
Also define $\lambda=(\lambda_{a})_{a\in \mathsf{Q}_{0}}\in 
\mathbb{C}^{\mathsf{Q}_{0}}$ by 
\begin{align*}
	\lambda_{[i,j]}&:=-\xi_{[i,j,1]}&
	\lambda_{[i,j,k]}&:=\xi_{[i,j,k]}-\xi_{[i,j,k+1]}
\end{align*}
where $\xi_{[i,j,e_{[i,j]}]}:=0$.
Let us define a sublattice 
\[
	\mathcal{L}:=
	\left\{
		\beta=(\beta_{a})_{a\in \mathsf{Q}_{0}}
		\in \mathbb{Z}^{\mathsf{Q}_{0}}\,
		\middle|\,
		\sum_{j=1}^{m_{0}}\beta_{[0,j]}=
		\sum_{j=1}^{m_{i}}\beta_{[i,j]}
		\text{ for all }i=1,\ldots,p
	\right\}\subset 
	\mathbb{Z}^{\mathsf{Q}_{0}}
\]
Then we define a subset of the quiver variety $\mathfrak{M}_{\lambda}
(\mathsf{Q},\alpha)$ as follows,
\[
	\mathfrak{M}_{\lambda}(\mathsf{Q},\alpha)^{\text{dif}}:=
	\mu_{\alpha}^{-1}(\lambda)^{\text{dif}}/\mathbf{G}(\alpha)
\]
where 
\[
	\mu_{\alpha}^{-1}(\lambda)^{\text{dif}}:=
	\left\{
		x\in\mu_{\alpha}^{-1}(\lambda)\,\middle|\,
		\begin{array}{l}
		x\text{ is $\mathcal{L}$-irreducible},\\
		\mathrm{det\,}(x_{\rho^{[0,j]}_{[i,j']}})_{
		\substack{1\le j\le m_{0}\\
	1\le j'\le m_{i}}}\neq 0,\,i=1,\ldots,p
	\end{array}
	\right\}.
\]
Here $\mathcal{L}$-irreducibility is defined as in Definition \ref{ellirred}.
Then from Proposition \ref{conjclass}, \ref{trunc order2}, and \ref{preisom},
we obtain the following identification.
\begin{thm}\label{preproof}
	We have a bijection
	\[
		\mathfrak{M}(\mathbf{B})\longrightarrow
		\mathfrak{M}_{\lambda}(\mathsf{Q},\alpha)^{\text{dif}}.
	\]
\end{thm}
\begin{proof}
	By Proposition \ref{conjclass}, \ref{trunc order2}, and 
	\ref{preisom}, it suffices to see that 
	$x\in \mu^{-1}(\lambda)^{\text{dif}}$
	implies
	that $x_{\rho_{[i,j,k]}}$ are injective and 
	$x_{\left(\rho_{[i,j,k]}\right)^{*}}$ are surjective.
	This can be checked similarly to the proof of Theorem 1 in \cite{C}.
	Indeed, if there exists $x_{\rho_{[i,j,k]}}$ which is not 
	injective, then there exists a nonzero element $v \in 
	\mathrm{Ker}(x_{\rho_{[i,j,k]}})$.
	Set $v_{k}:=v$ and $v_{l+1}:=
	\psi_{\left(\rho_{[i,j,l+1]}\right)^*}(v_{l})$ for $l\le k$.
	Then the  relation
	\[
		x_{\rho_{[i,j,l+1]}}
		x_{\left(\rho_{[i,j,l+1]}\right)^*}-
		x_{\left(\rho_{[i,j,l]}\right)^*}
		x_{\rho_{[i,j,l]}}=
		\lambda_{[i,j,l]}
	\]
	shows that $x_{\rho_{[i,j,l+1]}}(v_{l+1})$ is a
	multiple of $v_{l}$ for $l\ge k$.
	Thus  $v_{l}$, $l\le k$, span a subrepresentaion of $x$,
	which contradicts to the $\mathcal{L}$-irreducibility of $x$.
	A dual argument shows that 
	$x_{\left(\rho_{[i,j,k]}\right)^{*}}$ are surjective.
	\end{proof}

\begin{rem}\normalfont
	In the above theorem, we obtain an isomorphism between 
	the moduli space of meromorphic connections $\mathfrak{M}(\mathbf{B})$
	and the subset $\mathfrak{M}_{\lambda}(\mathsf{Q},\alpha)^{\text{dif}}$
	of the quiver variety.
	However we should notice that $\mathfrak{M}_{\lambda}(\mathsf{Q},\alpha)^{\text{dif}}$ does not coincide with $\mathfrak{M}^{\text{reg}}_{\lambda}(
	\mathsf{Q},\alpha)$ since 
	we imposed the $\mathrm{det\,}(x_{\rho^{[0,j]}_{[i,j']}})_{
	\substack{1\le j\le m_{0},\\1\le j'\le m_{i}}}\neq 0$ and 
	the $\mathcal{L}$-irreducibility does not coincide with the 
	irreducibility in general.
	Thus Crawley-Boevey's theorem (see Theorem \ref{CB})
	is not 
	applicable directly to our case.
\end{rem}
\subsection{Truncated orbits and representations of quivers}\label{Section3}
Let us recall the description of truncated orbits 
$\mathcal{O}_{B}$ of arbitrary orders $k$ as quiver varieties.
The description which will be given in this section was 
obtained by Boalch in \cite{Boarx} when $k\le 3$ and conjectured 
for arbitrary $k$.
In this section we shall give a quiver picture of truncated orbits 
of arbitrary orders following the paper by Yamakawa and the author \cite{HY}
in which the above Boalch's conjecture was finally settled.

Fix $k>1$
and $B=\sum_{i=1}^{k}B_{i}z^{-i}\in \mathfrak{g}^{*}_{k}$ 
, an HTL normal form written by 
\[
    B=\mathrm{diag}\left(
    q_{1}(z^{-1})I_{n_{1}}+R_{1}z^{-1},\ldots,
    q_{m}(z^{-1})I_{n_{m}}+R_{m}z^{-1}
    \right)
\]
where $R_{i}\in M(n_{i},\mathbb{C})$,
$q_{i}(z^{-1})\in z^{-2}\mathbb{C}[z^{-1}]$, $i=1,\ldots,m$ and 
$q_{i}\neq q_{j}$ if $i\neq j$.
To a pair $(j,j')$, $1\le j\neq j'\le  m$, we attach an integer
\begin{equation}\label{difference}
\begin{aligned}
d(j,j')
&:=\mathrm{deg\,}_{\mathbb{C}[x]}(q_{j}(x)-q_{j'}(x))-2.
\end{aligned}
\end{equation}
Moreover we set $d(j,j):=-1$ for the latter use.

Let $\bigoplus_{j=1}^{m(s)}V_{\langle s,j\rangle}$
be the decomposition of $\mathbb{C}^{n}$ as simultaneous invariant spaces of 
$\{B_{s+1},B_{s+2},\ldots,B_{k}\}$ for $s=1,\ldots,k-1$.
Especially we write $V_{j}:=V_{\langle 1,j\rangle}$ for 
$j=1,\ldots,m=m(1)$.

Let $X_{j,j'}$ be 
the $\mathrm{Hom}_{\mathbb{C}}(V_{j'},
V_{j})$-component of  $X\in M(n,\mathbb{C})$.
For 
$g(z)=\sum_{i=r}^{\infty}g_{i}z^{i}\in M(n,\mathbb{C}(\!(z)\!))$, 
write $(g(z))_{j,j'}:=\sum_{i=r}^{\infty}(g_{i})_{j,j'}z^{i}$,
$1\le j,j'\le m$.
We denote the
$\mathrm{Hom}_{\mathbb{C}}(V_{j},\mathbb{C}^{n})$-component
of $X\in M(n,\mathbb{C})$
by $X_{*,j}$ for each $i=1,\ldots,m$.
Similarly $X_{j,*}$ denote the 
$\mathrm{Hom}_{\mathbb{C}}(\mathbb{C}^{n},V_{j})$-
component.
We sometimes use the notation 
$X=(X_{j,j'})_{1\le j,j'\le m}=(X_{*,j'})_{1\le j'\le m}=
(X_{j,*})_{1\le j\le m}$.

Let $\pi_{s}:J_{s}
=\{1,\ldots,m(s)\}\rightarrow J_{s+1}=\{1,\ldots,m(s+1)\}$
be the natural surjection such that 
$V_{\langle s,j\rangle }\subset V_{\langle s+1,\pi_{s}(j)\rangle}$.
Define the total ordering $\{1<2<\cdots<m\}$ on  $J_{1}$  and 
also define total orderings on $J_{s}$, $s=2,\ldots,k-1$,
so that 
\[
    \text{if }  j_{1}<j_{2}, \text{ then }
    \pi_{s}(j_{1})\le \pi_{s}(j_{2}),\quad j_{1},j_{2}\in J_{s}.
\]

Let us define the  subgroup of $G_{k}$ by
\[
    G^{o}_{k}:=\left\{\sum_{i=0}^{k-1}A_{i}z^{i}\in G_{k}
    \,\bigg|\, A_{0}=I_{n}
    \right\}.
\]

Similarly define the subspace $\mathfrak{g}_{k}^{o}:=z\mathfrak{g}_{k}=
M\left(n,z\mathbb{C}[\![z]\!]/z^{k}\mathbb{C}[\![z]\!]\right)$
of $\mathfrak{g}_{k}$, which can be identified with 
\[
	\left\{
		\sum_{i=0}^{k-1}A_{i}z^{i}\in \mathfrak{g}_{k}
		\,\bigg|\, A_{0}=0
	\right\}.
\]
Then the dual space $(\mathfrak{g}_{k}^{o})^{*}$ can be identified with 
\[
	\left\{
		\sum_{i=0}^{k-1}A_{i}z^{-i-1}\in\mathfrak{g}^{*}_{k}
		\,\bigg|\, A_{0}=0
	\right\}.
\]
For $A=\sum_{i=1}^{k}A_{i}z^{-i}\in \mathfrak{g}^{*}_{k}$,
set $A_{\text{irr}}:=\sum_{i=2}^{k}A_{i}z^{-i}$ and 
$A_{\text{res}}:=A_{1}$.
Then we define the following two orbits 
\begin{align*}
	\mathcal{O}^{o}_{B}&:=
	\{gBg^{-1}\mid g\in G^{o}_{k}\}\subset \mathfrak{g}_{k}^{*},\\
	\mathcal{O}_{B_{\text{irr}}}&:
	=\{gB_{\text{irr}}g^{-1}\mid g\in G^{o}_{k}\}
\subset (\mathfrak{g}_{k}^{o})^{*}.
\end{align*}
Let us define the subgroup $H\subset \mathrm{GL}(n,\mathbb{C})$ by
$$H=\left\{
	h=\mathrm{diag}(h_{1},\ldots,h_{m})\mid h_{i}\in 
	\mathrm{GL}(n_{i},\mathbb{C}),\,
	i=1,\ldots,m
\right\}.$$
The following proposition links $\mathcal{O}^{o}_B$ with $\mathcal{O}_{B}$.
\begin{prop}[cf. Lemmas 2.2 and 2.4 in \cite{Boa1}]\label{modified truncated orbit}
	Set $$\mathrm{Ad}_{H}(\mathcal{O}^{o}_{B}):=
	\left\{
		hAh^{-1}\in \mathfrak{g}_{k}^{*}\mid
		h\in H,\,A\in \mathcal{O}^{o}_{B}
	\right\}.$$
	Then we have a bijection
	\[
		\mathrm{GL}(n,\mathbb{C})
		\times_{H}\mathrm{Ad}_{H}(\mathcal{O}^{o}_{B})
		\stackrel{\sim}{\rightarrow} \mathcal{O}_{B}.
	\]
	Here $\mathrm{GL}(n,\mathbb{C})
	\times_{H}\mathrm{Ad}_{H}(\mathcal{O}^{o}_{B})=
	\left(
	\mathrm{GL}(n,\mathbb{C})\times\mathrm{Ad}_{H}(\mathcal{O}^{o}_{B})
	\right)/\sim$,
	the equivalence relation $\sim$ is defined by
	$(g,A)\sim (gh^{-1},hAh^{-1})$ for $h\in H$.
\end{prop}
\begin{proof}
	Let us send 
	$(g,A)\in 
	\mathrm{GL}(n,\mathbb{C})\times_{H}\mathrm{Ad}_{H}(
	\mathcal{O}^{o}_{B})$ to
	$gAg^{-1}\in \mathcal{O}_{B}$ and show that this map is 
	well-defined and bijective.
	If $(g,A)=(g',A')\in \mathrm{GL}(n,\mathbb{C})
	\times_{H}\mathrm{Ad}_{H}(\mathcal{O}^{o}_{B})$,
	then there exists $h\in H$ such that 
	$g'=gh^{-1}$ and $A'=hAh^{-1}$. Thus $g'A'(g')^{-1}
	=gh^{-1}hAh^{-1}hg^{-1}
	=gAg^{-1}$.
	Let us see the surjectivity. For any $A'\in \mathcal{O}_{B}$,
	there exists $g=g_{0}+g_{1}x+\cdots\in G_{k}$ such that 
	$g^{-1}A'g=B$. Then $(g_{0}^{-1},g_{0}A'g_{0}^{-1})\in 
	\mathrm{GL}(n,\mathbb{C})\times \mathrm{Ad}_{H}(\mathcal{O}^{o}_{B})$
	and $g_{0}^{-1}(g_{0}A'g_{0}^{-1})g_{0}=A'$, 
	which shows the surjectivity.
	Next  we shall see the injectivity.
	Suppose that $(g,A),(g',A')\in 
	\mathrm{GL}(n,\mathbb{C})\times_{H}\mathrm{Ad}_{H}(
	\mathcal{O}^{o}_{B})$ are sent to the same element, i.e.,
	$gAg^{-1}=g'A'(g')^{-1}$.
	Then putting $h=g^{-1}g'$, we have $A'=h^{-1}Ah$.
	On the other hand, we may assume $A,A'\in \mathcal{O}^{o}_{B}$.
	Thus there exist $b,b'\in G_{k}^{o}$ such that 
	$b^{-1}Ab=(b')^{-1}A'b'=B$ which implies 
	$B=b^{-1}Ab=(hb')^{-1}A(hb')$.
	Lemma 2.9 in \cite{HY} shows that 
	the constant term of $b'hb^{-1}\in G_{k}$ is contained in 
	$\mathrm{Stab}_{H}(B_{1})$, the stabilizer of $B_{1}$ in $H$,
	i.e., $h\in \mathrm{Stab}_{H}(B_{1})\subset H$.
\end{proof}

According to the  ordering on each $J_{s}$, $s=1,\ldots,k-1$, 
let us define parabolic subalgebras of $M(n,\mathbb{C})$ as below,
\begin{align*}
	\mathfrak{p}(s)^{+}&:=\bigoplus_{\substack{j_{1},j_{2}\in J_{i},\\
    j_{1}\ge j_{2}}}
    \mathrm{Hom}_{\mathbb{C}}(V_{\langle s,j_{1}\rangle},
    V_{\langle s,j_{2}\rangle}),\\
    \mathfrak{p}(s)^{-}&:=\bigoplus_{\substack{j_{1},j_{2}\in J_{i},\\
    j_{1}\le j_{2}}}
    \mathrm{Hom}_{\mathbb{C}}(V_{\langle s,j_{1}\rangle},
    V_{\langle s,j_{2}\rangle}),
\end{align*}
and  similarly nilpotent subalgebras
\begin{align*}
	\mathfrak{u}(s)^{+}&:=\bigoplus_{\substack{j_{1},j_{2}\in J_{i},\\
    j_{1}> j_{2}}}
    \mathrm{Hom}_{\mathbb{C}}(V_{\langle s,j_{1}\rangle},
    V_{\langle s,j_{2}\rangle}),\\
    \mathfrak{u}(s)^{-}&:=\bigoplus_{\substack{j_{1},j_{2}\in J_{i},\\
    j_{1}< j_{2}}}
    \mathrm{Hom}_{\mathbb{C}}(V_{\langle s,j_{1}\rangle},
    V_{\langle s,j_{2}\rangle}),
\end{align*}
for $s=1,\ldots,k-1$.
Note that $\mathfrak{p}(s)^{\pm}=\mathfrak{h}(s)\oplus 
\mathfrak{u}(s)^{\pm}$ where 
\[
	\mathfrak{h}(s)
:=\bigoplus_{j\in J_{i}}
\mathrm{End\,}_{\mathbb{C}}(V_{\langle s,j\rangle}).
\]

\begin{lem}(cf. Lemma 3.4 in \cite{HY})\label{centralizer}
	An element $g\in G^{o}_{k}$ preserves 
	$B+(\mathfrak{u}^{+}_{1}\oplus \mathfrak{u}^{-}_{1})x^{-1}$,
	i.e., $g^{-1}bg\in B+(\mathfrak{u}^{+}_{1}\oplus
	\mathfrak{u}^{-}_{1})x^{-1}$ for all $b\in B
	+(\mathfrak{u}^{+}_{1}\oplus
	\mathfrak{u}^{-}_{1})x^{-1}$
	if and only if 
	\[
		g\in \mathcal{H}_{k}=
		\left\{
			h:=\sum_{i=1}^{k-1}h_{i}z^{i}\in G^{o}_{k}\,\bigg|\,
			h_{i}\in \mathfrak{h}_{i+1} \text{ for }
			i=1,\ldots,k-1
		\right\}.
	\]
	Here we put $\mathfrak{h}_{k}:=M(n,\mathbb{C})$.
\end{lem}
\begin{proof}
	We know that $\mathcal{H}_{k}$ is the stabilizer of $B_{\text{irr}}$
	in $(\mathfrak{g}^{o}_{k})^{*}$, 
	see Lemma 3.4 in \cite{HY} for example.
	Thus if $g\in G^{o}_{k}$ preserves 
	$B+(\mathfrak{u}^{+}_{1}\oplus \mathfrak{u}^{-}_{1})z^{-1}$,
	then $g\in \mathcal{H}_{k}$.
	Conversely take $g\in \mathcal{H}_{k}$ and 
	put $u_{1}z^{-1}=gBg^{-1}-B\in M(n,\mathbb{C})z^{-1}$.
	Then $u_{1}=B_{1}+
	\sum_{t=2}^{k}\sum_{s=1}^{t-1}g_{s}B_{t}g'_{t-s-1}-B_{1}$.
	Here we put $g=I_{n}+g_{1}z+g_{2}z^{2}+\cdots$ and 
	$g^{-1}=I_{n}+g'_{1}z+g'_{2}z^{2}+\cdots$.
	Then we  have 
	\begin{align*}
		u_{1}&=
	\sum_{t=2}^{k}B_{t}\sum_{s=1}^{t-1}g_{s}g'_{t-s-1}+
	\sum_{t=2}^{k}(g_{t-1}B_{t}-B_{t}g_{t-1})\\
	&=\sum_{t=2}^{k}B_{t}\cdot 0
	+\sum_{t=2}^{k}(g_{t-1}B_{t}-B_{t}g_{t-1})\\
	&=\sum_{t=2}^{k}(g_{t-1}B_{t}-B_{t}g_{t-1})
	\in \mathfrak{u}_{1}^{+}\oplus \mathfrak{u}_{1}^{-}.
	\end{align*}
	Here we put $g_{0}=g'_{0}=I_{n}$.
	Thus  
	we are done.
\end{proof}

Let us define subsets of $G^{o}_{k}$,
\begin{align*}
	&\mathcal{P}^{\pm}_{k}:	=\left\{
        \sum_{i=0}^{k-1}P_{i}z^{i}\in G^{o}_{k}\,\bigg|\, P_{i}\in 
        \mathfrak{p}^{\pm}_{i+1}, i=1,\ldots,k-1
    \right\},\\
    &\mathcal{U}^{\pm}_{k}:=\left\{
        \sum_{i=0}^{k-1}U_{i}z^{i}\in G^{o}_{k}\,\bigg|\, U_{i}\in 
        \mathfrak{u}^{\pm}_{i+1}, i=1,\ldots,k-1
    \right\},
\end{align*}
and subspaces of $\mathfrak{g}^{o}_{k}$ and $(\mathfrak{g}^{o}_{k})^{*}$,
\begin{align*}
	\mathfrak{U}^{\pm}_{k}&:=
	\left\{
		\sum_{i=1}^{k-1}U_{i}z^{i}\,\bigg|\,
		U_{i}\in \mathfrak{u}^{\pm}_{i+1},\,
		i=1,\ldots,k-1
	\right\},\\
    (\mathfrak{U}^{\mp}_{k})^{*}&:=
    \left\{
        \sum_{i=1}^{k-1}U_{i}z^{-i-1}\,\bigg|\,
        U_{i}\in \mathfrak{u}^{\pm}_{i+1},\,
        i=1,\ldots,k-1
    \right\}.
\end{align*}
Here we put $\mathfrak{p}^{\pm}_{k}:=M(n,\mathbb{C})$ and 
$\mathfrak{u}^{\pm}_{k}:=\{0\}$.
Let us note that $\mathcal{P}^{\pm}_{k}$ are subgroups of 
$G^{o}_{k}$ but $\mathcal{U}^{\pm}_{k}$
are not closed under the multiplication.
\begin{lem}[cf. Lemma 3.8 in \cite{HY}]\label{P-orbit}
	The $\mathcal{P}^{+}_{k}$-orbit through $B$ in $\mathfrak{g}^{*}_{k}$
	is $B+(\mathfrak{U}^{-}_{k})^{*}\oplus
	(\mathfrak{u}_{1}^{+}\oplus \mathfrak{u}_{1}^{-})z^{-1}$.
\end{lem}
\begin{proof}
	First we see that $\mathcal{P}^{+}_{k}$-orbit through $B$
	is included in $B+(\mathfrak{U}^{-}_{k})^{*}\oplus
	(\mathfrak{u}_{1}^{+}\oplus \mathfrak{u}_{1}^{-})x^{-1}$.
	From Lemma 3.8 in \cite{HY}, the $\mathcal{P}^{+}_{k}$-orbit
	through $B_{\text{irr}}$ in $(\mathfrak{g}^{o}_{k})^{*}$
	is $B_{\text{irr}}+(\mathfrak{U}^{-}_{k})^{*}$.
	Thus it suffices to show that $(pBp^{-1})_{\text{res}}
	\in 
	\left(B_{1}+(\mathfrak{u}^{+}_{1}\oplus\mathfrak{u}^{-}_{1})
	\right)z^{-1}$
	for all $p\in \mathcal{P}^{+}_{k}$.
	Let us set 
	$u:=(pBp^{-1})_{\text{res}}-B_{1}$.
	Then  
	$u=\sum_{t=2}^{k}\sum_{s=1}^{t-1}p_{s}B_{t}p'_{t-s-1}$
	where $p=p_{0}+p_{1}x+p_{2}x_{2}+\cdots$, $p^{-1}=p'_{0}+p'_{1}x+
	p'_{2}x^{2}+\cdots$.
	Let $h_{s}$ (resp. $h_{s}'$) be the $\mathfrak{h}_{s}$-component of 
	$p_{s}$ (resp. $p'_{s}$) $\in \mathfrak{p}^{+}_{s}
	=\mathfrak{h}_{s}\oplus 
	\mathfrak{u}^{+}_{s}$. 
	Then
	\begin{align*}
		u_{i,i}&=\left(
		\sum_{t=2}^{k}\sum_{s=1}^{t-1}p_{s}B_{t}p'_{t-s-1}
		\right)_{i,i}=\left(
		\sum_{t=2}^{k}\sum_{s=1}^{t-1}h_{s}B_{t}h'_{t-s-1}
		\right)_{i,i}\\
		&=\left(\sum_{t=2}^{k}B_{t}\sum_{s=1}^{t-1}h_{s}h'_{t-s-1}
		\right)_{i,i}
		+\sum_{t=2}^{k}(h_{t-1}B_{t}-B_{t}h_{t-1})_{i,i}
		=0
	\end{align*}
	for $i=1,\ldots,m$. 
	Here $X_{i,i}$ denote the
	$\mathrm{End}_{\mathbb{C}}(V_{i})$-component
	of $X\in M(n,\mathbb{C})$.
	Thus $u\in \mathfrak{u}^{+}_{1}\oplus \mathfrak{u}^{-}_{1}$ as 
	required.

	Conversely take an arbitrary element $D\in 
	B+(\mathfrak{U}^{-}_{k})^{*}\oplus
	(\mathfrak{u}_{1}^{+}\oplus \mathfrak{u}_{1}^{-})z^{-1}$.
	Lemma 3.8 in \cite{HY} shows that there exists 
	$p\in \mathcal{P}^{+}_{k}$ such that 
	$(pDp^{-1})_{\text{irr}}=B_{\text{irr}}$.
	Namely $pDp^{-1}=B_{\text{irr}}+(h+u)x^{-1}$ where 
	$h\in \mathfrak{h}_{1}, u\in \mathfrak{u}_{1}^{+}\oplus
	\mathfrak{u}_{1}^{-}$. 
	Thus there exists $p'\in \mathcal{P}^{+}_{k}$ such that 
	$p'D(p')^{-1}=D'=B_{\text{irr}}+hx^{-1}$, an HTL normal form.
	Moreover recall that if $D'\in \mathcal{O}^{o}_{B}$ is 
	of HTL normal form, then $D'=B$, cf. Lemma 2.9 in \cite{HY}.
	Thus $D$ is in $\mathcal{P}^{+}_{k}$-orbit through $B$.
\end{proof}
\begin{lem}[Lemma 3.5 in \cite{HY}]
	For any $g\in G^{o}_{k}$, there uniquely exist $u_{-}\in 
	\mathcal{U}^{-}_{k}$ 
	and $p_{+}\in \mathcal{P}^{+}_{k}$ such that $g=u_{-} p_{+}$.
\end{lem}

For $A\in \mathcal{O}_{B_{\text{irr}}}$, take $g\in G_{k}^{o}$ so that  
$g^{-1}Ag=B_{\text{irr}}$ and 
decompose $g=u_{-}p_{+}$ as above.
Note that $u_{-}$ does not depend on the choice of $g$ because
the stabilizer of $B_{\text{irr}}$
is contained in $\mathcal{P}^{+}_{k}$. 
Thus $u_{-}$ is uniquely determined by $A\in \mathcal{O}_{B_{\text{irr}}}$.
Then let us put $Q=u_{-}-I_{n}$,
$A'=u_{-}^{-1}A$
and $P=A'|_{(\mathfrak{U}^{-}_{k})^{*}}$.

\begin{prop}[Theorem 3.6 in \cite{HY}]
    The map 
    \[
	    \begin{array}{cccc}
		    \Phi\colon&\mathcal{O}_{B_{\text{irr}}}&\longrightarrow&
		    \mathfrak{U}_{k}^{-}\times (\mathfrak{U}^{-}_{k})^{*}\\
		    &A&\longmapsto&(Q,P)
	    \end{array}
    \]
    is bijective.
\end{prop}

\begin{prop}\label{residue}
	Let us take $A\in \mathcal{O}^{o}_{B}$ and set 
	$(Q,P)=\Phi(A_{\text{irr}})$.
	Then the following equations in $\mathfrak{g}_{k}^{*}$  hold for 
	$l=1,\ldots,m$.
	\begin{align*}
		B_{l,l}-A_{l,l}=&
		-Q_{l,1}P_{1,l}-Q_{l,2}P_{2,l}-\cdots- Q_{l,l-1}P_{l-1,l}\\
		&+P_{l,l+1}Q_{l+1,l}+P_{l,l+2}Q_{l+2,l}+\cdots
		+P_{l,m}Q_{m,l}
	\end{align*}
\end{prop}
\begin{proof}
	First we note that putting $u=I_{n}+Q$, we have
	$u^{-1}Au \in B+(\mathfrak{U}^{-}_{k})^{*}\oplus
	(\mathfrak{u}_{1}^{+}\oplus \mathfrak{u}_{1}^{-})z^{-1}$
	by Lemma \ref{P-orbit}.
	Denote the $(\mathfrak{U}^{-}_{k})^{*}$ and 
	$(\mathfrak{u}^{+}_{1}\oplus\mathfrak{u}^{-}_{1})z^{-1}$ components of
	$u^{-1}Au$ by $R$ and $U$ respectively.
	Since $Au=u(B+R+U)$, we have 
	\begin{equation}\label{eq1}
		B_{l,l}+\sum_{s=1}^{l-1}Q_{l,s}(R+U)_{s,l}
	=
	A_{l,l}+\sum_{s=l+1}^{m}A_{l,s}Q_{s,l}.
	\end{equation}
	Recalling that $Q\in M(n,z\mathbb{C}[z])$,
	we have $Q_{l,s}(R+U)_{s,l}=Q_{l,s}R_{s,l}$  in the above equation.

	Let us set $\bar{A}=A_{\text{irr}}$.
	Note that 
	\[
		\bar{A}_{s,t}=\sum_{j=1}^{s-1}Q_{s,j}P_{j,t}+P_{s,t}
	\] for $1\le s<t\le m$.
	Then right hand side of the proposition is written as follows,
	\begin{align*}
		&\bar{A}_{l,l+1}Q_{l+1,l}+\cdots+\bar{A}_{l,m}Q_{m,l}\\
		&-Q_{l,1}(P_{1,l}+P_{1,l+1}Q_{l+1,l}+\cdots+P_{1,m}Q_{m,l})\\
		&-\cdots\\
		&-Q_{l,l-1}(P_{l-1,l}+P_{l-1,l+1}Q_{l+1,l}+\cdots
		+P_{l-1,m}Q_{m,l})\\
		=&A_{l,l+1}Q_{l+1,l}+\cdots+A_{l,m}Q_{m,l}
		-Q_{l,1}R_{1,l}-\cdots
		-Q_{l,l-1}R_{l-1,l}.
	\end{align*}
	Here we use fact that each component of  $Q_{l,l'}$ 
	$(l>l')$ is in $z\mathbb{C}[z]$, which deduces
	$A_{l',l}Q_{l,l'}=\bar{A}_{l',l}Q_{l,l'}$.
	Then the above equation and the equation $(\ref{eq1})$  
	induce the required equations.
\end{proof}
\begin{prop}\label{orbit to quiver}
	The map 
	\[
		\begin{array}{cccc}
			\Psi_{B}\colon &\mathcal{O}^{o}_{B}&\longrightarrow&
			\left(\mathfrak{U}^{-}_{k}\times
			(\mathfrak{U}^{-}_{k})^*\right)\times
			(\mathfrak{u}_{1}^{+}\oplus \mathfrak{u}_{1}^{-})\\
			&A&\longmapsto&
			(\Phi(A_{\text{irr}}),
			A_{\text{res}}|_{\mathfrak{u}_{1}^{+}\oplus
			\mathfrak{u}_{1}^{-}})
		\end{array}
	\]
	is bijective.
\end{prop}
\begin{proof}
	For $A\in \mathcal{O}^{o}_{B}$, let us put 
	$(Q,P)=\Phi(A_\text{irr})$.
	Then Proposition \ref{residue} shows that 
	$A_{\text{res}}|_{\mathfrak{h}_{1}}$
	is uniquely determined by $(Q,P)$ and $B$,
	which implies that $\Psi_{B}$ is injective.
	
	Let us take an arbitrary $\left( (Q,P ),U\right)\in 
	\left(\mathfrak{U}^{-}_{k}\times
	(\mathfrak{U}^{-}_{k})^*\right)\times
	(\mathfrak{u}_{1}^{+}\oplus \mathfrak{u}_{1}^{-})$.
	Define
	\[
		h=\mathrm{diag}(h_{1},\ldots,h_{m})\in \mathfrak{h}_{1}
	\]
	by
	\[
		h_{i}=R_{i}+\left(
		\sum_{j=1}^{i-1}Q_{i,j}P_{j,i}
		-\sum_{j=i+1}^{m}P_{i,j}Q_{j,i}
		\right)_{\text{res}}.
	\]
	Here recall that
	$R_{i}=(B_{\text{res}})_{i,i}$ for $i=1,\ldots,m$.
	Then there exists an HTL normal form $B'\in \mathfrak{g}^{*}_{k}$
	such that $B'_{\text{irr}}=B_{\text{irr}}$ and 
	$\Phi^{-1}( (Q,P))+(h+U)x^{-1}$ is contained in 
	some $\mathcal{O}^{o}_{B'}$.
	However $B=B'$ by Proposition \ref{residue} and the construction
	of $h$.
	Thus $\Phi^{-1}( (Q,P))+(h+U)x^{-1}$ is 
	the inverse image 
	of $( (Q,P),U)$ by $\Psi_{B}$.
	Hence $\Psi_{B}$ is surjective.
\end{proof}

Let $C_{{R}_{i}}$ be the conjugacy class of each
$R_{i}$, block diagonal component of the 
residue of $B$. 
Recalling that $h^{-1}Bh=B_{\text{irr}}+h^{-1}B_{\text{res}}h$ for $h\in H$,
we have one to one correspondence,
\[
	\begin{array}{ccc}
		\mathrm{Ad}_{H}(B):=\{hBh^{-1}\mid h\in H\}&
		\longrightarrow &\prod_{i=1}^{m}C_{R_{i}}\\
		B'&\longmapsto&( (B'_{\text{res}})_{i,i})_{i=1,\ldots,m}
	\end{array}.
\]

Also recall that  
$\mathrm{Ad}_{H}(\mathcal{O}^{o}_{B})$ is  
the disjoint union of $\mathcal{O}^{o}_{B'}$ for $B'\in \mathrm{Ad}_{H}(B)$,
i.e., $\mathrm{Ad}_{H}(\mathcal{O}^{o}_{B})
=\bigsqcup_{B'\in \mathrm{Ad}_{H}(B)}\mathcal{O}^{o}_{B'}$,
see Lemma 2.9 in \cite{HY}.

For any $B'\in \mathrm{Ad}_{H}(B)$ we can define $\Psi_{B'}$ 
as in the above proposition.
Thus we can define the map
\begin{equation*}
	\Psi\colon
	\mathrm{Ad}_{H}(\mathcal{O}^{o}_{B})=
	\bigsqcup_{B'\in \mathrm{Ad}_{H}(B)}\mathcal{O}^{o}_{B'}
	\longrightarrow
	\left(\prod_{i=1}^{m}C_{R_{i}}\right)\times
	\left(
	\left(\mathfrak{U}^{-}_{k}\times
	(\mathfrak{U}^{-}_{k})^*\right)\times
	(\mathfrak{u}_{1}^{+}\oplus \mathfrak{u}_{1}^{-})
	\right)
\end{equation*}
by $\Psi(A)=(((B'_{\text{res}})_{i,i})_{1\le i\le m},\Psi_{B'}(A))$ 
for  $A\in \mathcal{O}^{o}_{B'}$ with $B'=
\sum_{i=1}^{k-1}B'_{i}x^{-i}\in S_{B}$. 
Then it
is bijective.

Under these preparations, now we can define a quiver $\mathsf{Q}$ as follows. 
The set of vertices is
\[\mathsf{Q}_{0}:=\{0\}\cup\{1,\ldots,m\}.\]
The set of arrows is 
\begin{align*}
    \mathsf{Q}_{1}:=&\left\{
        \rho^{[j]}_{i,i'}\colon i\rightarrow i'\,\bigg|\,
	\begin{array}{l}
	1\le i<i' \le m,\\
	1\le j\le d(i,i')
	\end{array}
    \right\}
    \cup
    \left\{
        \rho_{i}\colon 0\rightarrow i\mid 
        i=1,\ldots,m
    \right\}.
\end{align*}
Fix the dimension vector $\alpha=(\alpha_{a})_{a\in \mathsf{Q}_{0}}$ defined by
$\alpha_{0}=:n$ and $\alpha_{i}:=\mathrm{dim}_{\mathbb{C}}V_{i}$,
$i=1,\ldots,m$.

Let us construct a map from $\mathrm{Ad}_{H}(\mathcal{O}_{B}^{o})$ 
to the representation space
of $\overline{\mathsf{Q}}$.
For $A\in \mathcal{O}^{o}_{B'}$, $B'\in \mathrm{Ad}_{H}(B)$, we set 
$(Q,P)=\Phi(A_{\text{irr}})$ and 
define the representation 
$
x_{A}\in \mathrm{Rep}(\overline{\mathsf{Q}},\alpha)$ 
as follows:
\begin{align*}
	(x_{A})_{\rho^{[j]}_{i,i'}}&:=P^{[j]}_{i,i'},
	&(x_{A})_{(\rho^{[j]}_{i,i'})^{*}}&:=
    Q^{[j]}_{i',i},\\
    (x_{A})_{\rho_i}&:=(I_{n})_{i,*}, 
    &(x_{A})_{\rho_{i}^{*}}&:=\left(A_{\text{res}}
    \right)_{*,i},
\end{align*}
for $i,i'=1,\ldots,m$.
Here we set $P=\sum_{i=1}^{k-2}P^{[i]}z^{-i-1}$ and
$Q=\sum_{i=1}^{k-2}Q^{[i]}z^{i}$.
Then Proposition \ref{residue} tells us 
that $\mu_{\alpha}(x_{A})_{i}=(B'_{\text{res}})_{i,i}\in C_{R_{i}}$
for $i=1,\ldots,m$.
\begin{prop}\label{orbit to quiver1}
The following map is bijective,
\begin{multline*}
	\tilde{\Psi}\colon \mathrm{Ad}_{H}(\mathcal{O}^{o}_{B})\longrightarrow
	\left\{
            x\in 
	    \mathrm{Rep\,}(\overline{\mathsf{Q}},\alpha)
           \,\middle|\,
	   \begin{array}{l}
           (\psi_{\rho_{i}})_{1\le i\le m}=I_n,\,
	   \mu_{\alpha}(x)_{i}\in C_{R_i}\\
	   \text{ for }
           i=1,\ldots,m
   	\end{array}
        \right\},
\end{multline*}
which is defined by $\tilde{\Psi}(A):=x_{A}$ for $A\in 
\mathrm{Ad}_{H}(\mathcal{O}^{o}_{B})$ as above.
Moreover $\tilde{\Psi}$ preserves $H$-actions, i.e.,
$\tilde{\Psi}(hAh^{-1})=h\cdot x_{A}$ for all $h\in H$.
\end{prop}
\begin{proof}
	Proposition \ref{orbit to quiver} 
	shows that $\tilde{\Psi}$ is bijective.
	The last assertion can be directly checked.
\end{proof}
Finally we can obtain a correspondence between 
$\mathcal{O}_{B}$ and  representations of $\overline{\mathsf{Q}}$.
\begin{prop}\label{symplectic}
    There exists a bijection
    \begin{multline*}
	    \mathcal{O}_{B}\cong \mathrm{GL}(n,\mathbb{C})
	    \times_{H}\mathrm{Ad}_{H}(\mathcal{O}^{o}_{B})
	   \longrightarrow\\
        \left\{
           x\in \mathrm{Rep\,}(\overline{\mathsf{Q}},\alpha)
           \,\bigg|\,
	   \begin{array}{l}
           \mathrm{det\,}(x_{\rho_{i}})_{1\le i\le m}\neq 0,
	   \mu_{\alpha}(x)_{i}\in C_{R_{i}}\\\text{ for }
           i=1,\ldots,m
   \end{array}
        \right\}/
	\prod_{i=1}^{m}\mathrm{GL}(\alpha_{i},\mathbb{C}).
\end{multline*}
\end{prop}
\begin{proof}
	Let us define a map $\overline{\Psi}$ from 
	$\mathrm{GL}(n,\mathbb{C})\times \mathrm{Ad}_{H}(\mathcal{O}^{o}_{B})$
	to  
	\begin{equation*}
        \left\{
	    x\in \mathrm{Rep\,}(\overline{\mathsf{Q}},\alpha)
           \,\bigg|\,
	   \begin{array}{l}
           \mathrm{det\,}(x_{\rho_{i}})_{1\le i\le m}\neq 0,\,
	   \mu_{\alpha}(x)_{i}\in C_{R_i}\\
	   \text{ for }
           i=1,\ldots,m
   	\end{array}
        \right\}.
	\end{equation*}
	For $(g,A)\in \mathrm{GL}(n,\mathbb{C})
	\times \mathrm{Ad}_{H}(\mathcal{O}^{o}_{B})$,
	$x=\overline{\Psi}( (g,A))=
	$ 
	is defined as below:
	\begin{align*}
    	x_{\rho^{[j]}_{i,i'}}&=P^{[j]}_{i,i'},
    	&x_{(\rho^{[j]}_{i,i'})^{*}}=
    	Q^{[j]}_{i',i},\\
	x_{\rho_i}&=(g^{-1})_{i,*}, 
	&x_{\rho_{i}^{*}}=\left((gA)_{\text{res}}
    	\right)_{*,i},
	\end{align*}
	where $(Q,P)=\Phi(A_{\text{irr}})$ and 
	write $P=\sum_{i=1}^{k-2}P^{[i]}z^{-i-1}$,
	$Q=\sum_{i=1}^{k-2}Q^{[i]}z^{i}$.
	Proposition \ref{orbit to quiver1} shows that this map is bijective.
	Moreover we can directly check that this map preserves $H$-actions.
	Thus we are done.
\end{proof}

For example, let us consider an HTL normal form $B=\sum_{i=1}^{4}B_{i}z^{-i}$
such that 
\begin{align*}
	B_{4}&=\mathrm{diag\,}(a^{(4)}_{1},a^{(4)}_{2},a^{(4)}_{2},
	a_{2}^{(4)}),&
	B_{3}&=\mathrm{diag\,}(\ast,a^{(3)}_{1},a^{(3)}_{2},
	a^{(3)}_{2}),\\
	B_{2}&=\mathrm{diag\,}(\ast,\ast,a^{(2)}_{1},a_{2}^{(2)}),&
	B_{1}&=\mathrm{diag\,}(\ast,\ast,\ast,\ast),
\end{align*}
where $a^{(i)}_{1}\neq a^{(i)}_{2}$.
Then the corresponding quiver is as follows.

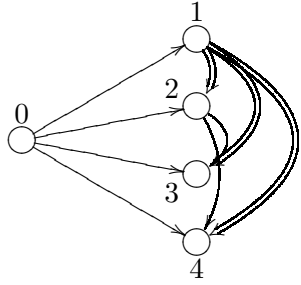
\[
\begin{xy}
	(0,27)*++!D{1}*\cir<5pt>{}="A",
	(0,18)*+!D+!R{2}*\cir<5pt>{}="D",
	(0,9)*+!U+!R{3}*\cir<5pt>{}="C",
	(0,0)*++!U{4}*\cir<5pt>{}="B",
	(-23,13.5)*++!D{0}*\cir<5pt>{}="E"
\ar@/^13mm/@{=>} "A";"B",
\ar@/^8mm/@{=>}"A";"C",
\ar@/^/@{=>} "A";"D",
\ar@/^3mm/@{->} "D";"B",
\ar@/^4mm/@{->} "D";"C",
\ar@{->} "E";"A",
\ar@{->} "E";"B",
\ar@{->} "E";"C",
\ar@{->} "E";"D",
\end{xy}
\]
\begin{prop}[cf. Lemma 3.1 in \cite{HY}]\label{sub-inv}
	Let $(g,A)\in 
	\mathrm{GL}(n,\mathbb{C})\times_{H}\mathrm{Ad}_{H}(\mathcal{O}^{o}_{B})
	\cong\mathcal{O}_{B}$ and $x\in 
	\mathrm{Rep}(\overline{\mathsf{Q}},\alpha)$ the corresponding elements
	under the isomorphism in Proposition \ref{symplectic}.
	\begin{itemize}
		\item[$(i)$]
			Let us suppose that $x$ has a 
			subrepresentation 
			$N=(N_{a},\psi_{\rho})_{a\in \mathsf{Q}_{0},\rho
			\in \overline{\mathsf{Q}}_{1}}$
			satisfying 
			$\mathrm{dim}_{\mathbb{C}}\bigoplus_{i=1}^{m}N_{i}
			=\mathrm{dim}_{\mathbb{C}}N_{0}$, then
			$W=\bigoplus_{i=1}^{m}N_{i}$ is 
			invariant under all 
			$A_{i}$ where $A=\sum_{i=1}^{k}A_{i}z^{-i}$
			and similarly 
			$gW=N_{0}$ is invariant under 
			all $gA_{i}g^{-1}$.
		\item[$(ii)$]
			Any subspace $S\subset \mathbb{C}^{n}$ invariant
			under all $A_{i}$ is homogeneous with respect 
			to the decomposition $\mathbb{C}^{n}
			=\bigoplus_{i=1}^{m}V_{i}$. Moreover
			$N_{i}=S\cap V_{i}$, i=1,\ldots,m,
			and $N_{0}=gS$ define
			the subrepresentation $N=(N_{a},\psi_{\rho})$ 
			where 
			$\psi_{\rho}$ are restrictions of $x_{\rho}$
			on $N_{i}$.
	\end{itemize}
\end{prop}
\begin{proof}
	If $x$ has a subrepresentation $N$ as above,
	then Lemma 3.10 in \cite{HY} shows that 
	$W$ is invariant under $A_{i}$ for $i\ge 2$ and obviously
	invariant under $A_{1}$ since
	$A_{1}=(x_{\rho^{*}_{i}})_{1\le i\le m}
	(x_{\rho_{i}})_{1\le i\le m}$.
	
	Conversely let $S\subset \mathbb{C}^{n}$ be a subspace 
	invariant under all
	$A_{i}$.
	We need to check that all $\psi_{\rho}$ are well-defined, namely
	$x_{\rho}(N_{s(\rho)})\subset N_{t(\rho)}$ which
	are already checked in Lemma 3.10 in \cite{HY}
	for $\rho^{[j]}_{i,i'}$ and $(\rho^{[j]}_{i,i'})^{*},
	i,i'=1,\ldots,m,\,j=1,\ldots,d(i,i')$.
	For $\rho_{i}$, we have 
	$x_{\rho_{i}}(N_{i})=gN_{i}\subset gS=N_{0}$.
	For $\rho^{*}_{i}$, we have
	$x_{\rho_{i}^{*}}(N_{0})=V_{i}
	\cap A_{1}g^{-1}N_{0}=V_{i}\cap A_{1}S\subset N_{i}$.

\end{proof}
\subsection{Quivers associated with differential equations}\label{equationsandquiver}
Now we are ready to consider a correspondence between
moduli spaces $\mathfrak{M}(\mathbf{B})$ of arbitrary $k_{i}$ and 
subsets of quiver varieties $\mathfrak{M}_{\lambda}(\mathsf{Q},\alpha)$
as we saw in Section \ref{Section2} under the restriction $k_{0}=\cdots
=k_{p}=2$.

For  $i=0,\ldots,p$, let us fix a collection of nonzero positive integers 
$k_{i}$ and HTL normal forms 
$B_{i}=\sum_{j=1}^{k_{i}}B^{(i)}_{j}z^{-j}\in 
 \mathfrak{g}^{*}_{k_{i}}$.
Then write
\[
    B_{i}=
    \mathrm{diag}\left(
    q_{[i,1]}(z^{-1})I_{n_{[i,1]}}+R_{[i,1]}z^{-1},\ldots,
    q_{[i,m_i]}(z^{-1})I_{n_{[i,m_i]}}+R_{[i,m_{i}]}z^{-1}
    \right)
\]
for $i=0,\ldots,p$ 
where $q_{[i,j]}(z^{-1})\in z^{-2}\mathbb{C}[z^{-1}]$ satisfying 
$q_{[i,j]}\neq q_{[i,j']}$ if $j\neq j'$ 
and $R_{[i,j]}\in M(n_{[i,j]},\mathbb{C})$.

For each  $i=0,\ldots,p$,
decompose $\mathbb{C}^{n}=\bigoplus_{j=1}^{m_{i}(s)}V^{(i)}_{\langle
s,j\rangle}$
as simultaneous  $(B^{(i)}_{s+1},\ldots,B^{(i)}_{k_{i}})$-invariant 
subspaces.
In particular we write 
$V_{[i,j]}:=V^{(i)}_{\langle 1,j\rangle}$ for 
$i=0,\ldots,p$ and $j=1,\ldots,m_{i}$.
Here we note $m_{i}(1)=m_{i}$.

For each pair $j,j'\in \{1,\ldots,m_{i}\}$, attach the integer
$d_{i}(j,j')$ defined by
\[
	d_{i}(j,j'):=\mathrm{deg\,}_{\mathbb{C}[z^{-1}]}(q_{[i,j]}(z^{-1})-
	q_{[i,j']}(z^{-1}))-2
\]
if $j\neq j'$ or $d_{i}(j,j'):=-1$ if $j=j'$.
Set  $I_{\text{irr}}:=\{i\in\{0,\ldots,p\}\mid m_{i}>1
\}\cup \{0\}$ and $I_{\text{reg}}:=\{0,\ldots,p\}\backslash I_{\text{irr}}$. 
\begin{rem}\normalfont
	Suppose that $m_{i}=1$ for some $i\in \{0,\ldots,p\}$. Then
	the truncated orbit of the normal form $B_{i}$ is 
	trivial, namely $\mathcal{O}_{B_{i}}\cong C_{R_{[i,1]}}$. 
	Thus $I_{\text{irr}}$ can be seen as the set of singular points 
	at which truncated orbits are nontrivial and we add the point $0$
	as a ``base point'' to $I_{\text{irr}}$.
\end{rem}
Now consider the following quiver $\mathsf{Q}^{\text{irr}}$.
The set of vertices is 
\begin{align*}
	\mathsf{Q}^{\text{irr}}_{0}:=\left\{
        [i,j]\mid i\in I_{\text{irr}},\,j=1,\ldots,m_{i}
    \right\}.
\end{align*}
Also define
\begin{align*}
	\mathsf{Q}^{0\to i_{\text{irr}}}_{1}&:=
    \left\{
        \rho^{[0,j]}_{[i,j']}\colon
        [0,j]\rightarrow [i,j']\,\middle|\,
	\begin{array}{l}j=1,\ldots,m_{0},\\
        i\in I_{\text{irr}}\backslash\{0\},\\
        j=1,\ldots,m_{i}
\end{array}
    \right\},\\
    \mathsf{Q}^{B_{i}}_{1}&:=
    \left\{
        \rho^{[k]}_{[i,j],[i,j']}\colon
        [i,j]\rightarrow [i,j']\,\middle|\,
	\begin{array}{l}
        i\in I_{\text{irr}},\\
        1\le j<j'\le m_{i},\\
	1\le k\le d_{i}(j,j')
	\end{array}
    \right\}.
\end{align*}
Then the set of vertices is 
\[
	\mathsf{Q}^{\text{irr}}_{1}:=\mathsf{Q}^{0\to i_{\text{irr}}}_{1}\sqcup
	\bigsqcup_{i\in I_{\text{irr}}} \mathsf{Q}^{B_{i}}_{1}.
\]

Let us define the dimension vector 
$\alpha^{\text{irr}}=(\alpha_{a})_{a\in \mathsf{Q}^{\text{irr}}_{0}}$
by $\alpha_{[i,j]}:=n_{[i,j]}$. 

\begin{prop}\label{orbitquiv}
    There exists a bijection
    \begin{multline*}
        \Phi\colon
        \left\{
           \left(\sum_{j=1}^{k_{i}}A^{(i)}_{j}z^{-j}
	   \right)_{0\le i\le p}
            \in
    \prod_{i=0}^{p}\mathcal{O}_{B_{i}}\,
	    \middle|\,
	    \begin{aligned}
	    \sum_{i=0}^{p}A^{(i)}_{1}=0,\\
		\text{irreducible}
	\end{aligned}
    \right\}\big/\mathrm{GL}(n,\mathbb{C})\longrightarrow\\
    \left\{\begin{array}{c}
		    (x,(L_{i})_{i\in I_{\text{reg}}})\in\\ 
	    \mathrm{Rep}(\overline{\mathsf{Q}^{\text{irr}}},\alpha^{
		    \text{irr}})\times 
	    \prod_{i\in I_{\text{reg}}}\mathcal{O}_{B_{i}}\end{array}
    	\,\middle|\,
	    \begin{aligned}
		    \text{$\mathcal{L}^{\text{irr}}$-irreducible},\\
	    \mathrm{det}\left(
            x_{\rho^{[0,j]}_{[i,j']}}
            \right)_{\substack{1\le j\le m_{0}\\
            1\le j'\le m_{i}}}&\neq 0
            ,\,i\in I_{\text{irr}}\backslash\{0\},\\
	    \mu_{\alpha^{\text{irr}}}(x)_{[i,j]}&\in C_{R_{[i,j]}},
	    i\in I_{\text{irr}}
    \end{aligned}
    \right\}\\
    \big/\prod_{a\in \mathsf{Q}^{\text{irr}}_{0}}\mathrm{GL}(\alpha_{a},\mathbb{C}).
    \end{multline*}
    Here 
    \[
	    \mathcal{L}^{\text{irr}}:=\left\{
		    \beta\in \mathbb{Z}^{\mathsf{Q}^{\text{irr}}_{0}}
		    \,\middle|\,
		    \sum_{j=1}^{m_{0}}\beta_{[0,j]}=\sum_{j=1}^{m_{i}}
		    \beta_{[i,j]}\text{ for all }i\in I_{\text{irr}}\backslash
		    \{0\}
	    \right\}.
    \]
\end{prop}
\begin{proof}
	Setting $H_{0}:=\{\mathrm{diag}(h_{1},\ldots,h_{m_{0}})
	\mid h_{i}\in \mathrm{GL}(n_{[0,i]},\mathbb{C})\}$, 
	we can identify 
	$\mathcal{O}_{B_{0}}/\mathrm{GL}(n,\mathbb{C})$ with 
	$\mathrm{Ad}_{H_{0}}(\mathcal{O}^{o}_{B_{0}})/H_{0}$,
	see Proposition \ref{modified truncated orbit}.	
	Let us set 
	\begin{equation*}
		V:=\\
		\left\{\begin{array}{c}
           \left(\sum_{j=1}^{k_{i}}A^{(i)}_{j}z^{-j}\right)
	   _{0\le i\le p}
            \in\\
	    \mathrm{Ad}_{H_{0}}(\mathcal{O}^{o}_{B_{0}})
	    \times\prod_{i\in I_{\text{irr}}\backslash\{0\}}
	    \mathcal{O}_{B_{i}}\times 
    \prod_{i\in I_{\text{reg}}}\mathcal{O}_{B_{i}}\end{array}
	    \,
	    \middle|\,
    \sum_{i=0}^{p}A^{(i)}_{1}=0
	\right\}.
	\end{equation*}
	Then we  can identify 
	\[\left\{
           \left(\sum_{j=1}^{k_{i}}A^{(i)}_{j}z^{-j}
	   \right)_{0\le i\le p}
            \in
    \prod_{i=0}^{p}\,
	    \middle|\,
		\sum_{i\in I_{\text{irr}}}A^{(i)}_{1}=0
	\right\}/\mathrm{GL}(n,\mathbb{C})\]
	with $V/H_{0}$.
	
	Comparing with Proposition \ref{trunc order2}, 
	we see that 
	Propositions \ref{orbit to quiver1} and \ref{symplectic}
	give a bijection from $V$ to 
	\begin{multline*}
        \left\{
		x\in 
		\mathrm{Rep}(\overline{\mathsf{Q}^{\text{irr}}},\alpha^{
			\text{irr}})\,\middle|\,
	    \begin{array}{l}\mathrm{det}\left(
            x_{\rho^{[0,j]}_{[i,j']}}
            \right)_{\substack{1\le j\le m_{0}\\
            1\le j'\le m_{i}}}\neq 0
            ,\,i\in I_{\text{irr}}\backslash\{0\},\\
	    \mu_{\alpha^{\text{irr}}}(x)_{[i,j]}\in C_{R_{[i,j]}},
	    i\in I_{\text{irr}}
    \end{array}
    \right\}\\
    \bigg/\prod_{a\in \mathsf{Q}^{\text{irr}}_{0}\backslash\{[0,j]\mid 
	    j=1,\ldots,m_{0}\}}\mathrm{GL}(\alpha_{a},\mathbb{C})
    \end{multline*}
    and the bijection preserves $H_{0}$-actions. 
    Thus we have the required bijection.
    The correspondence between the irreducibility and 
    $\mathcal{L}^{\text{irr}}$-irreducibility follows from
    Proposition \ref{sub-inv} and the arguments in the proof of 
    Proposition \ref{trunc order2}.
\end{proof}

Finally as we saw in Section \ref{reviewoffuchsian},
we shall associate conjugacy classes of 
residue matrices of HTL normal forms to representations of quivers. 
For each $R_{[i,j]}$, $i=0,\ldots,p$ and  $j=1,\ldots,m_{i}$,
let us choose  
$\xi^{[i,j]}_{1},\ldots,\xi^{[i,j]}_{e_{[i,j]}}\in \mathbb{C}$
so that 
\[
	\prod_{k=1}^{e_{[i,j]}}(R_{[i,j]}-\xi^{[i,j]}_{k})=0.
\]
Let  $\xi=\left(
\left\{\xi^{[i,j]}_{1},\ldots,\xi^{[i,j]}_{e_{[i,j]}}\right\}
\right)_{\substack{0\le i\le p,\\1\le j\le m_{i}}}$
be the collection of the ordered sets 
$\left\{\xi^{[i,j]}_{1},\ldots,\xi^{[i,j]}_{e_{[i,j]}}\right\}$.

Now consider the following quiver $\mathsf{Q}$.
Set 
\[
	\mathsf{Q}_{0}^{\text{leg}}:=\left\{
        [i,j,k]\,\middle|\,
	\begin{array}{l}i=0,\ldots,p,\\
		j=1,\ldots,m^{(i)},\\
        k=1,\ldots,e_{[i,j]}-1
	\end{array}
    \right\}
\]
Then the set of vertices is 
\[
	\mathsf{Q}_{0}:=\mathsf{Q}^{\text{irr}}_{0}\sqcup 
	\mathsf{Q}_{0}^{\text{leg}}.
\]

Also set 
\begin{align*}
	\mathsf{Q}_{1}^{\text{leg}_{i}}&:=
	\left\{
		\rho_{[i,j,k]}\colon [i,j,k]\rightarrow
		[i,j,k-1]\,\middle|\,
		\begin{array}{l}
			j=1,\ldots,m_{i},\\
			k=2,\ldots,e_{[i,j]}-1
		\end{array}
	\right\},\\
	\mathsf{Q}_{1}^{\text{leg}_{i}\to B_{i}}&:=\left\{
		\rho_{[i,j,1]}\colon [i,j,1]\rightarrow [i,j]\mid
		j=1,\ldots,m_{i}
	\right\},\\
	\mathsf{Q}^{\text{leg}_{i}\to 0}_{1}&:=\left\{
		\rho^{[i,1,1]}_{[0,j]}\colon
		[i,1,1]\rightarrow [0,j]\mid 
		i\in I_{\text{reg}},\,j=1,\ldots,m_{0}
	\right\}.
\end{align*}
The set of arrows is 
\begin{align*}
	\mathsf{Q}_{1}:=
	\mathsf{Q}_{1}^{0\to I_{\text{irr}}}
	&\sqcup
	\bigsqcup_{i\in I_{\text{irr}}}\left(
		\mathsf{Q}_{1}^{B_{i}}\sqcup
		\mathsf{Q}_{1}^{\text{leg}_{i}\to B_{i}}\sqcup
		\mathsf{Q}_{1}^{\text{leg}_{i}}
	\right)\\
	&\sqcup
	\bigsqcup_{i\in I_{\text{reg}}}\left(
		\mathsf{Q}_{1}^{\text{leg}_{i}\to 0}\sqcup
		\mathsf{Q}_{1}^{\text{leg}_{i}}
	\right).
\end{align*}

For example, let us consider the following 
$\mathbf{B}=(B_{0},B_{1},B_{2})$.
	\begin{align*}
		B^{(0)}&=
		\begin{pmatrix}
			a^{(0)}_{4}&&&\\
			&a^{(0)}_{4}&&\\
		        &&a^{(0)}_{4}&\\
			&&&b^{(0)}_{4}
		\end{pmatrix}z^{-4}
		+
		\begin{pmatrix}
			a^{(0)}_{3}&&&\\
			&a^{(0)}_{3}&&\\
			&&b^{(0)}_{3}&\\
			&&&c^{(0)}_{3}
		\end{pmatrix}z^{-3}
		\\
		&+
		\begin{pmatrix}
			a^{(0)}_{2}&&&\\
			&b^{(0)}_{2}&&\\
			&&c^{(0)}_{2}&\\
			&&&d^{(0)}_{2}
		\end{pmatrix}z^{-2}
		+
		\begin{pmatrix}
			\xi^{[0,1]}_{1}&&&\\
			&\xi^{[0,2]}_{1}&&\\
		 &&\xi^{[0,3]}_{1}&\\
			&&&\xi^{[0,4]}_{1}
		\end{pmatrix}z^{-1},
	\end{align*}
	\begin{align*}
		B_{1}&=
		\begin{pmatrix}
			a^{(1)}_{2}&&&\\
			&a^{(1)}_{2}&&\\
			&&a^{(1)}_{2}&\\
			&&&b^{(1)}_{2}
		\end{pmatrix}
		z^{-2}+
		\begin{pmatrix}
			\xi^{[1,1]}_{1}&&&\\
		      &\xi^{[1,1]}_{2}&&\\
		       &&\xi^{[1,1]}_{3}&\\
		       &&&\xi^{[1,2]}_{1}
		\end{pmatrix}z^{-1},
	\end{align*}
	\begin{align*}
		B_{2}&=
		\begin{pmatrix}
			\xi^{[2,1]}_{1}&&&\\
		       &\xi^{[2,1]}_{2}&&\\
		       &&\xi^{[2,1]}_{3}&\\
		       &&&\xi^{[2,1]}_{4}
		\end{pmatrix}z^{-1}.
	\end{align*}
	Here any distinct 
	two of $\{a^{(i)}_{j},b^{(i)}_{j},c^{(i)}_{j},d^{(i)}_{j}\}$
	stand for distinct complex numbers and
	$\xi^{[i,j]}_{k}\neq \xi^{[i,j]}_{k'}$ if $k\neq k'$. 

	Then we can associate the following quiver to this $\mathbf{B}$.
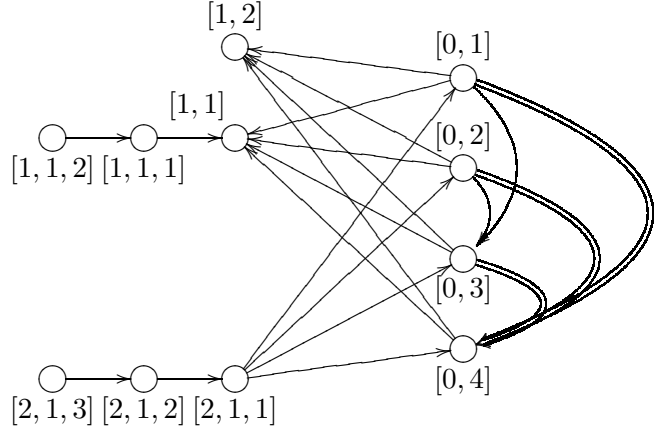
	\[
	\begin{xy}
		(0,12) *++!U{[0,3]}*\cir<5pt>{}="A",
		(0,0) *++!U{[0,4]}*\cir<5pt>{}="B",
		(0,36)*++!D{[0,1]}*\cir<5pt>{}="C",
		(0,24)*++!D{[0,2]}*\cir<5pt>{}="D",
		(-30,28)*+!R+!D{[1,1]}*\cir<5pt>{}="E",
		(-30,40)*++!D{[1,2]}*\cir<5pt>{}="F",
		(-42,28)*++!U{[1,1,1]}*\cir<5pt>{}="G",
		(-54,28)*++!U{[1,1,2]}*\cir<5pt>{}="H",
		(-30,-4)*++!U{[2,1,1]}*\cir<5pt>{}="I",
		(-42,-4)*++!U{[2,1,2]}*\cir<5pt>{}="J",
		(-54,-4)*++!U{[2,1,3]}*\cir<5pt>{}="K",
		\ar@/^50pt/@{=>}"D";"B",	
		\ar@/^70pt/@{=>}"C";"B",
		\ar@/^30pt/@{=>}"A";"B",
		\ar@/^20pt/@{->}"C";"A",
		\ar@/^10pt/@{->}"D";"A",
		\ar@{->}"H";"G",
		\ar@{->}"G";"E",
		\ar@{->}"K";"J",
		\ar@{->}"J";"I",
		\ar@{->}"I";"A",
		\ar@{->}"I";"B",
		\ar@{->}"I";"C",
		\ar@{->}"I";"D",
		\ar@{<-}"E";"A",
		\ar@{<-}"E";"B",
		\ar@{<-}"E";"C",
		\ar@{<-}"E";"D",
		\ar@{<-}"F";"A",
		\ar@{<-}"F";"B",
		\ar@{<-}"F";"C",
		\ar@{<-}"F";"D",
		\end{xy}
	\]

Define the dimension vector $\alpha=(\alpha_{a})_{a\in \mathsf{Q}_{0}}$ 
by 
\begin{align*}
	\alpha_{[i,j]}&:=n_{[i,j]},&
\alpha_{[i,j,k]}&:=
\mathrm{rank\,}\prod_{l=1}^{k}(R_{[i,j]}-\xi^{[i,j]}_{l})
.
\end{align*}
Also define 
	$\lambda=(\lambda_{a})_{a\in \mathsf{Q}_{0}}$ by 
\begin{align*}
\lambda_{[i,j]}&:=-\xi^{[i,j]}_{1},&&
\text{ for }i\in I_{\text{irr}}\backslash\{0\},\,
j=1,\ldots,m_{i},\\
\lambda_{[0,j]}&:=-\xi^{[0,j]}_{1}
-\sum_{i\in I_{\text{reg}}}\xi^{[i,1]}_{1}&&
\text{ for }j=1,\ldots,m_{0},\\
\lambda_{[i,j,k]}&:=\xi^{[i,j]}_{k}-\xi^{[i,j]}_{k+1}&&
\text{ for }\begin{array}{l}i=0,\ldots,p,\, j=1,\ldots,m_{i},\\
k=1,\ldots,e_{[i,j]}-1.
\end{array}
\end{align*}
Then Propositions \ref{conjclass} and \ref{orbitquiv} show the following.

\begin{prop}\label{irregularquiver}
    Let $B_{0},\ldots,B_{p}$ be HTL normal forms chosen as above.
    Then there exists a bijection
    \begin{align*}
	    \Phi_{\xi}\colon
        &\left\{
		\left(\sum_{j=1}^{k_{i}}A^{(i)}_{j}x^{-j}\right)
		_{0\le i\le p}
            \in
	    \prod_{i=0}^{p}\mathcal{O}_{B_{i}}\,\bigg|\,
            \sum_{i=0}^{p}A^{(i)}_{1}=0
    \right\}/\mathrm{GL}(n,\mathbb{C})
    \end{align*}
    \begin{align*}
        &\rightarrow
        \Big\{
		x\in\mu^{-1}(\lambda)
		\subset
            \mathrm{Rep}(\overline{\mathsf{Q}},\alpha)\,\big|\,\\
            &\quad\quad\mathrm{det}\left(
            x_{\rho^{[0,j]}_{[i,j']}}
            \right)_{\substack{1\le j\le m_{0}\\
            1\le j'\le m_{i}}}\neq 0
            ,\,i\in I_{\text{irr}}\backslash\{0\},\\
            &\quad\quad\left(
            x_{\rho^{[i,1,1]}_{[0,j]}}
            \right)_{1\le j\le m_{0}}\colon
	    \mathbb{C}^{\alpha_{[i,1,1]}}\rightarrow 
	    \bigoplus_{j=1}^{m_{0}}
	    \mathbb{C}^{\alpha_{[0,j]}}
            ,\text{ injective},\,
            i\in I_{\text{reg}},\\ 
            &\quad\quad\left(
            x_{\left(\rho^{[i,1,1]}_{[0,j]}\right)^{*}}
            \right)_{1\le j\le m_{0}}\colon
            \bigoplus_{j=1}^{m_{0}}
	    \mathbb{C}^{\alpha_{[0,j]}}
	    \rightarrow 
	    \mathbb{C}^{\alpha_{[i,1,1]}}
            ,\text{ surjective},\,
            i\in I_{\text{reg}},\\
	    &\quad\quad x_{\rho_{[i,j,k]}},\text{ injective},\
	   x_{(\rho_{[i,j,k]})^{*}},\text{ surjective}
    \Big\}/\prod_{a\in \mathsf{Q}_{0}}\mathrm{GL}(\alpha_{a},\mathbb{C}).
    \end{align*}
\end{prop}

Finally we shall close this section
 giving the isomorphism from 
 the moduli space of meromorphic connections $\mathfrak{M}(\mathbf{B})$ as 
 follows.
 
Define a sublattice of $\mathbb{Z}^{\mathsf{Q}_{0}}$, 
\[
	\mathcal{L}:=\left\{\beta\in 
\mathbb{Z}^{\mathsf{Q}_{0}}\,\middle|\,
	\sum_{j=1}^{m^{(0)}}\beta_{[0,j]}=\sum_{j=1}^{m^{(i)}}\beta_{[i,j]}
\text{ for all }i\in I_{\text{irr}}\backslash\{0\}
\right\}.
\]
Set $\mathcal{L}^{+}:=\mathcal{L}\cap (\mathbb{Z}_{\ge 0})^{\mathsf{Q}_{0}}$.

\begin{df}[$\mathcal{L}$-irreducible]
    \normalfont
    If $x\in \mathrm{Rep}(\overline{
	    \mathsf{Q}},\alpha)$
	    has 
    no nontrivial proper subrepresentation
    $\{0\}\neq y\subsetneqq x$  
    with 
    $\mathbf{dim\,}y\in \mathcal{L}$,
    then $x$ is said to be {\em $\mathcal{L}$-irreducible}.
\end{df}
Then we define a subset of the quiver variety $\mathfrak{M}_{\lambda}
(\mathsf{Q},\alpha)$ as follows,
\[
	\mathfrak{M}_{\lambda}(\mathsf{Q},\alpha)^{\text{dif}}:=
	\mu_{\alpha}^{-1}(\lambda)^{\text{dif}}/\mathbf{G}(\alpha)
\]
where 
\[
	\mu_{\alpha}^{-1}(\lambda)^{\text{dif}}:=
	\left\{
		x\in\mu_{\alpha}^{-1}(\lambda)\,\middle|\,
		\begin{array}{l}
		x\text{ is $\mathcal{L}$-irreducible},\\
		\mathrm{det\,}(x_{\rho^{[0,j]}_{[i,j']}})_{
		\substack{1\le j\le m_{0}\\
	1\le j'\le m_{i}}}\neq 0,\,i
	\in I_{\text{irr}}\backslash
\{0\}
	\end{array}
	\right\}.
\]
Also define 
\[
	\mu_{\alpha}^{-1}(\lambda)^{\text{det}}:=
	\left\{
		x\in\mu_{\alpha}^{-1}(\lambda)\,\middle|\,
		\mathrm{det\,}(x_{\rho^{[0,j]}_{[i,j']}})_{
		\substack{1\le j\le m_{0}\\
	1\le j'\le m_{i}}}\neq 0,\,i
	\in I_{\text{irr}}\backslash
\{0\}
	\right\}
\]
for the latter use.
\begin{thm}\label{irreducibleandquasi}
	The restriction of the map in Proposition \ref{irregularquiver},
    \begin{equation*}
	    \Phi_{\xi}\colon
	    \mathfrak{M}(\mathbf{B})
        \longrightarrow
		\mathfrak{M}_{\lambda}(\mathsf{Q},\alpha)^{\text{dif}}
	\end{equation*}
    is well-defined and bijective.
\end{thm}
\begin{proof}
	The map $\Phi_{\xi}$ in Proposition \ref{irregularquiver}
	sends irreducible $\mathbf{A}\in 
	\prod_{i=0}^{p}\mathcal{O}_{B_{i}}$ to 
	the $\mathcal{L}$-irreducible representation $\Phi_{\xi}(\mathbf{A})$
	and vice versa 
	by Propositions \ref{conjclass} and \ref{orbitquiv}.
	Thus we only need to check that 
	the $\mathcal{L}$-irreducibility of a representation
	$x,\rho\in \mu^{-1}(\lambda)^{\text{dif}}
	$
	implies
	that $\left(x_{\rho^{[i,1,1]}_{0,j}}\right)_{1\le j\le m_{0}}$
	and $x_{\rho_{[i,j,k]}}$ are injective, and 
	$\left(x_{\left(\rho^{[i,1,1]}_{0,j}\right)^{*}}
	\right)_{1\le j\le m_{0}}$
	and $x_{\left(\rho_{[i,j,k]}\right)^{*}}$ are surjective.
	This can be shown as in the proof of Theorem \ref{preproof}.
\end{proof}
\section{Middle convolutions and reflections}
\label{middleconvolutionandreflection}
In the previous section we gave an isomorphism $
\mathfrak{M}(\mathbf{B})\cong \mathfrak{M}_{\lambda}(\mathsf{Q},\alpha)^{\text{dif}}$. However $\mathfrak{M}_{\lambda}(\mathsf{Q},\alpha)^{\text{dif}}$
does not coincide with $\mathfrak{M}_{\lambda}^{\text{reg}}(\mathsf{Q},\alpha)$
in general. Namely, we can not refer to 
Theorem 1.2 in \cite{C1} by Crawley-Boevey (see Corollary \ref{nonemptyquiv} )
for the non-emptiness of $\mathfrak{M}(\mathbf{B})$ in our general setting.
Therefore in the remaining sections, we shall determine the condition for 
the non-emptiness of $\mathfrak{M}_{\lambda}(\mathsf{Q},\alpha)^{\text{dif}}$
tracing the case of $\mathfrak{M}_{\lambda}^{\text{reg}}(\mathsf{Q},\alpha)$
done by Crawley-Boevey in \cite{C1}.
In the case of $\mathfrak{M}_{\lambda}^{\text{reg}}(\mathsf{Q},\alpha)$,
reflection functors (see \cite{CBH} and \cite{Nak2}) plays very impotent role
as we can see in Kac's theorem for the existence of 
indecomposable representations of quivers 
in \cite{Kac}.
Unfortunately it, however, is easy to see 
that all reflection functors on $\mathfrak{M}(
\mathsf{Q},\alpha)$ do not preserve $\mathfrak{M}_{\lambda}(\mathsf{Q},\alpha)^{
	\text{dif}}$ because of the open condition 
	$\mathrm{det}(x_{\rho^{[0,j]}_{[i,j']}})_{\substack{1\le j\le m_{0},\\
	1\le j'\le m_{i}}}\neq 0$.
Thus in this section we shall introduce some operations for 
$\mathfrak{M}_{\lambda}(\mathsf{Q},\alpha)^{\text{dif}}$
which are induced from operations originally defined for 
$\mathfrak{M}(\mathbf{B})$, so-called 
middle convolutions,
additions, etc. 
The map in Proposition \ref{irregularquiver} enable us to 
define analogous operations on $\mathfrak{M}_{\lambda}(\mathsf{Q},\alpha)^{\text{dif}}$.

In this section we retain the notation in the previous section.
\subsection{A review of middle convolutions}
Let us give a review of middle convolutions on differential 
equations with irregular singular points.
The middle convolution is originally defined by N. Katz in 
\cite{Katz} and 
reformulated as an operation on Fuchsian systems by 
Dettweiler-Reiter \cite{DR2}, see also \cite{DR}  and V\"olklein's paper \cite{V}.
There are several studies to generalize the middle convolution
to non-Fuchsian differential equations, 
see \cite{A},\cite{Kaw},\cite{T},\cite{Y1}
for example.
Among them we shall give a review of middle convolutions following  \cite{Y1}.

From $\mathbf{A}=(\sum_{j=1}^{k_{i}}A^{(i)}_{j}z^{-j})_{0\le i\le p}\in \prod_{i=0}^{p}
\mathcal{O}_{B_{i}}$, 
let us construct a 5-tuple
$(V,W,T,Q,P)$ consisting of $\mathbb{C}$-vector spaces $V$, $W$
and 
$T\in \mathrm{End}_{\mathbb{C}}(W)$,
$Q\in \mathrm{Hom}_{\mathbb{C}}(W,V)$, 
$P\in \mathrm{Hom}_{\mathbb{C}}(V,W)$.
Set $V:=\mathbb{C}^{n}$ and $\widehat{W}_{i}:=V^{\oplus k_{i}}$ for 
 $i=0,\ldots,p$.
Then define
\begin{align*}
	\widehat{Q}_{i}&:=(A^{(i)}_{k_{i}},A^{(i)}_{k_{i}-1},\ldots,A^{(i)}_{1})
	\in \mathrm{Hom}_{\mathbb{C}}(\widehat{W}_{i},V),\\
	\widehat{P}_{i}&:=\begin{pmatrix}0\\\vdots\\0\\
		\mathrm{Id}_{V}
	\end{pmatrix}
	\in \mathrm{Hom}_{\mathbb{C}}(V,\widehat{W}_{i}),\
	\widehat{N}_{i}:=
	\begin{pmatrix}
		0&\mathrm{Id}_{V}&&0\\
		&0&\ddots&\\
		&&\ddots&\mathrm{Id}_{V}\\
		0&&&0
	\end{pmatrix}
	\in \mathrm{End}_{\mathbb{C}}(\widehat{W}_{i}).
\end{align*}
Setting  
\begin{align*}\widehat{W}&:=\bigoplus_{i=0}^{p}\widehat{W}_{i},\\
\widehat{T}&:=(\widehat{N}_{i})_{0\le i\le p}\in \bigoplus_{i=0}^{p}
\mathrm{End}_{\mathbb{C}}(\widehat{W}_{i})\subset\mathrm{End}_{\mathbb{C}}
(\widehat{W}),\\
\widehat{Q}&:=(\widehat{Q}_{i})_{0\le i\le p}\in 
\bigoplus_{i=0}^{p}\mathrm{Hom}_{\mathbb{C}}(\widehat{W}_{i},V)
=\mathrm{Hom}_{\mathbb{C}}(\widehat{W},V),\\
\widehat{P}&:=(\widehat{P}_{i})_{0\le i\le p}\in 
\bigoplus_{i=0}^{p}\mathrm{Hom}_{\mathbb{C}}(V,\widehat{W}_{i})
=\mathrm{Hom}_{\mathbb{C}}(V,\widehat{W}),
\end{align*}
we have a 5-tuple $(V,\widehat{W},\widehat{T},\widehat{Q},\widehat{P})$.
Further setting
\[
	\widehat{A}_{i}:=
\begin{pmatrix}
	A^{(i)}_{k_{i}}&A^{(i)}_{k_{i}-1}&\cdots&A^{(i)}_{1}\\
	&A^{(i)}_{k_{i}}&\ddots&\vdots\\
	&&\ddots&A^{(i)}_{k_{i}-1}\\
	0&&&A^{(i)}_{k_{i}}
\end{pmatrix}
\in \mathrm{End}_{\mathbb{C}}(\widehat{W}_{i}),
\]
we define 
$W_{i}:=\widehat{W}_{i}/\mathrm{Ker}\widehat{A}_{i}$ and 
$W:=\bigoplus_{i=0}^{p}W_{i}$.
Then $T,Q,P$ are the maps induced from $\widehat{T},\widehat{Q},
\widehat{P}$ respectively.
\begin{df}[Yamakawa \cite{Y1}]\normalfont
	The 5-tuple $(V,W,T,Q,P)$ given above is called the 
	{\em canonical datum} for $\mathbf{A}\in
	\prod_{i=0}^{p}\mathcal{O}_{B_{i}}$.
\end{df}
\begin{lem}\label{sizeofW}
	Let  $(V,W,T,Q,P)$ is the canonical datum of $\mathbf{A}
	\in\prod_{i=0}^{p}\mathcal{O}_{B_{i}}$. Then
	\[
		\mathrm{dim}_{\mathbb{C}}W=
		\sum_{i=0}^{p}\sum_{j=0}^{k_{i}-1}
		\left(n-\mathrm{dim}_{\mathbb{C}}
		\bigcap_{l=0}^{j} \mathrm{Ker}B^{(i)}_{k_{i}-l}
		\right).
	\]
\end{lem}
\begin{proof}
	For $\mathbf{A}=
	(\sum_{j=1}^{k_{i}}A^{(i)}_{j}z^{-j})
	_{0\le i\le p}\in \prod_{i=0}^{p}\mathcal{O}_{B_{i}}$,
there exists $g^{(i)}=g^{(i)}_{0}+g^{(i)}_{1}z+\cdots
+g^{(i)}_{k_{i}-1}z^{k_{i}-1}
\in G_{k_{i}}$ such that 
$(g^{(i)})^{-1}(\sum_{j=1}^{k_{i}}A^{(i)}_{j}z^{-j})g^{(i)}=B_{i}$.
Thus
\begin{align*}
	&\begin{pmatrix}
		g^{(i)}_{0}&g^{(i)}_{1}&\cdots&g^{(i)}_{k_{i}-1}\\
	&g^{(i)}_{0}&\ddots&\vdots\\
	&&\ddots&g^{(i)}_{1}\\
	0&&&g^{(i)}_{0}
\end{pmatrix}^{-1}
\begin{pmatrix}
	A^{(i)}_{k_{i}}&A^{(i)}_{k_{i}-1}&\cdots&A^{(i)}_{1}\\
	&A^{(i)}_{k_{i}}&\ddots&\vdots\\
	&&\ddots&A^{(i)}_{k_{i}-1}\\
	0&&&A^{(i)}_{k_{i}}
\end{pmatrix}
\begin{pmatrix}
	g^{(i)}_{0}&g^{(i)}_{1}&\cdots&g^{(i)}_{k_{i}-1}\\
	&g^{(i)}_{0}&\ddots&\vdots\\
	&&\ddots&g^{(i)}_{1}\\
	0&&&g^{(i)}_{0}
\end{pmatrix}\\
&=\begin{pmatrix}
	B^{(i)}_{k_{i}}&B^{(i)}_{k_{i}-1}&\cdots&B^{(i)}_{1}\\
	&B^{(i)}_{k_{i}}&\ddots&\vdots\\
	&&\ddots&B^{(i)}_{k_{i}-1}\\
	0&&&B^{(i)}_{k_{i}}
\end{pmatrix}=\widehat{B}_{i}.
\end{align*}
This implies 
that $\mathrm{dim}_{\mathbb{C}}\mathrm{Ker}\widehat{A}_{i}=
\mathrm{dim}_{\mathbb{C}}\mathrm{Ker}\widehat{B}_{i}$
as required.
\end{proof}

Fix $t\in \{0,\ldots,p\}$, take a polynomial $p_{t}(z^{-1})
=\sum_{j=1}^{k_{t}}p^{(t)}_{j}z^{-j}\in 
z^{-1}\mathbb{C}[z^{-1}]$ and define an operation, 
called {\em addition}, as follows.
For an element $\mathbf{A}=
(A_{i}(x^{-1}))_{0\le i\le p}\in 
\prod_{i=0}^{p}\mathcal{O}_{B_{i}}$,
we define $\mathrm{Add}^{(t)}_{p_{t}(z^{-1})}(\mathbf{A})
:=(A'_{i}(z^{-1}))_{0\le i\le p}$
by
\[
	A'_{i}(z^{-1}):=
	\begin{cases}
		A_{i}(z^{-1})&\text{if }i\neq t,\\
		A_{t}(z^{-1})-p_{t}(x^{-1})&
		\text{if }i=t.
	\end{cases}
\]
Then $
\mathrm{Add}^{(t)}_{p_{t}(x^{-1})}(\mathbf{A})\in \prod_{i=0}^{p}\mathcal{O}
_{B'_{i}}$
where 
\[
	B'_{i}:=
	\begin{cases}
		B_{i}&\text{if }i\neq t,\\
		B_{t}-p_{t}(z^{-1})&\text{if }i=t.
	\end{cases}
\]

Set \[
	\mathcal{J}_{i}:=\{[i,j]\mid j=1,\ldots,m_{i}\}
\]
for $i=0,\ldots,p$
and  
\[
	\mathcal{J}:=\prod_{i=0}^{p}\mathcal{J}_{i}.
\]
Then let us define
\[
	\mathrm{Add}_{\mathbf{i}}:=\prod_{i=0}^{p}
	\mathrm{Add}^{(i)}_{q_{[i,j_{i}]}(z^{-1})+\xi^{[i,j_{i}]}_{1}z^{-1}},
\]  
for $\mathbf{i}=([i,j_{i}])_{0\le i\le p}\in \mathcal{J}$.
Here we use the notation $\prod_{i\in \{a,b,\ldots,\}}f_{i}=f_{a}\circ f_{b}
\circ\cdots$ and 
 note that the operators 
 $\mathrm{Ad}^{(i)}_{q_{[i,j_{i}]}(z^{-1})+\xi^{[i,j_{i}]}_{1}z^{-1}}$
for $i\in {0,\ldots,p}$ are commutative.

Take $\mathbf{A}=(A_{i}(z^{-1}))_{0\le i\le p}
\in \prod_{i=0}^{p}\mathcal{O}_{B_{i}}$ satisfying
$\sum_{i=0}^{p}\mathrm{Res}A_{i}(z^{-1})=0$.
Suppose that we can 
choose $\mathbf{i}\in \mathcal{J}$ so that \[
	\xi_{\mathbf{i}}:=
\sum_{i=0}^{p}\xi^{[i,j_{i}]}_{1}\neq 0.
\]
Let $(V,W,T,Q,P)$  be the canonical datum of 
$\mathrm{Add}_{\mathbf{i}}(\mathbf{A})$. 
Following Example 3 in \cite{Y1}, we construct a new 5-tuple
$(V',W,T,Q',P')$ as follows.
Note that $QP=-\xi_{\mathbf{i}}\mathrm{Id}_{V}$.
Thus $Q$ and $P$ are surjective and injective respectively.
Let us set $V':=\mathrm{Coker\,}P$ and $Q'\colon W\rightarrow V'$,
the natural projection. Then we have the split exact sequence
\[
	0\longrightarrow V\stackrel{P}{\longrightarrow}W
	\stackrel{Q'}{\longrightarrow}V'\longrightarrow 0
\]
with the left splitting $(-\xi_{\mathbf{i}}^{-1}Q)P=\mathrm{Id}_{V}$.
Then from the splitting, 
we can define $P'\colon V'\rightarrow W$ be the injection such that 
$Q'(\xi_{\mathbf{i}}^{-1}P')=\mathrm{Id}_{V'}$.
Then we have a 5-tuple
$(V',W,T,Q',P')$.

Next we set $Q'_{i}$ (resp. $P'_{i}$) to be the 
$\mathrm{Hom}_{\mathbb{C}}(W_{i},V)$
(resp. $\mathrm{Hom}_{\mathbb{C}}(V,W_{i})$) component of $Q'$ (resp. $P'$).
Also set $N_{i}$ to be the $\mathrm{End}_{\mathbb{C}}(W_{i})$-component
of $T$.
Define 
\[
	(A')^{(i)}_{j}:=Q'_{i}N_{i}^{j-1}P'_{i}
\]
and $\mathbf{A}':=(A'_{i}(z^{-1}))_{0\le i\le p}$ where 
$A'_{i}(z^{-1}):=\sum_{j=1}^{k_{i}}(A')^{(i)}_{j}z^{-j}$.
We note that $\sum_{i=0}^{p}(A')^{(i)}_{1}=Q'P'
=\xi_{\mathbf{i}}\mathrm{Id}_{V'}$.

Finally let us set 
\[
	\mathbf{A}'':=\mathrm{Add}_{\mathbf{i}}^{-1}\circ
	\mathrm{Add}^{(0)}_{2\xi_{\mathbf{i}}z^{-1}}(\mathbf{A}').
\]
Then $\mathbf{A}''=(A''_{i}(x^{-1}))_{0\le i\le p}$ satisfies that 
$\sum_{i=0}^{p}\mathrm{Res}A''_{i}(x^{-1})=0$.
Let us denote $\mathbf{A}''$ by 
$\mathrm{mc}_{\mathbf{i}}(\mathbf{A})$ and 
call the operator $\mathrm{mc}_{\mathbf{i}}$ the {\em middle convolution}
at $\mathbf{i}$.

Let us recall basic properties of middle convolutions.
\begin{lem}\label{sizeW}
	Let $(V,W,T,Q,P)$ be the canonical datum of 
	the above $\mathrm{mc}_{\mathbf{i}}(\mathbf{A})$. 
	Then  
	$\mathrm{dim}_{\mathbb{C}}W=\sum_{i=0}^{p}w_{i}$ where 
	\begin{align*}
		w_{i}:=
		\sum_{j=1}^{m_{i}}&
			\left((d_{i}(j,j_{i})+1)
			\mathrm{dim}_{\mathbb{C}}V_{[i,j]}\right)\\
			&+
			(n-\mathrm{dim}_{\mathbb{C}}V_{[i,j_{i}]})+
			\mathrm{rank}(R_{[i,j_{i}]}-\xi_{[i,j_{i},1]}).
	\end{align*}
\end{lem}
\begin{proof}
	From Lemma \ref{sizeofW}, we have  
	\[
		w_{i}=\sum_{t=0}^{k_{i}-1}\left(n-\mathrm{dim}_{\mathbb{C}}
		\bigcap_{l=0}^{t}\mathrm{Ker }\tilde{B}^{(i)}_{k_{i}-l}
	\right)
	\]
	where $\tilde{B}_{i}:=B_{i}-\left(q_{[i,j_{i}](z^{-1})
	+\xi^{[i,j_{i}]}_{1}z^{-1}}\right)I_{n}
	$ for $i=0,\ldots,p$.
	Suppose that $0\le t\le k_{i}-2$. Then we have
	$
		V_{[i,j]}\subset 
		\bigcap_{l=0}^{t}\mathrm{Ker }\tilde{B}^{(i)}_{k_{i}-l}
	$
	if and only if $\mathrm{deg}_{\mathbb{C}[z^{-1}]}(
	q_{[i,j]}-q_{[i,j_{i}]}) < k_{i}-t$.
	Equivalently, we have 
	\[
		n-\mathrm{dim}_{\mathbb{C}}
		\bigcap_{l=0}^{t}\mathrm{Ker }\tilde{B}^{(i)}_{k_{i}-l}
		=\sum_{j\in\{
		j'\mid d_{i}(j',j_{i})\ge k_{i}-t-2\}}
		\mathrm{dim}_{\mathbb{C}}V_{[i,j]}
	\]
	for $0\le t\le k_{i}-2$. Thus
	\[
		\sum_{t=0}^{k_{i}-2}\left(
			n-\mathrm{dim}_{\mathbb{C}}
		\bigcap_{l=0}^{t}\mathrm{Ker }\tilde{B}^{(i)}_{k_{i}-l}
		\right)
		=\sum_{j=0}^{m_{i}}
		(d_{i}(j,j_{i})+1)\mathrm{dim}_{\mathbb{C}}V_{[i,j]}.
	\]
	Combining this computation and the fact
	\[
		\mathrm{dim}_{\mathbb{C}}\mathrm{Ker}
		\bigcap_{l=0}^{k_{i}-1}\tilde{B}^{(i)}_{k_{i}-l}
		=\mathrm{dim}_{\mathbb{C}}
		\mathrm{Ker}(R_{[i,j_{i}]}-\xi^{[i,j_{i}]}_{1}),
	\]
	we obtain the required result.
\end{proof}
\begin{prop}[Yamakawa \cite{Y1}]\label{middleconv}
	Let us take $\mathbf{A}=(A_{i}(z^{-1}))_{0\le i\le p}\in 
	\prod_{i=0}^{p}\mathcal{O}_{B_{i}}$ satisfying 
	$\sum_{i=0}^{p}\mathrm{Res}A_{i}(z^{-1})=0$.
	Suppose we can choose $\mathbf{i}\in\mathcal{J}$ so that
	$\xi_{\mathbf{i}}\neq 0$.
	\begin{enumerate}
		\item If $\mathbf{A}$ is irreducible, then 
			$\mathrm{mc}_{\mathbf{i}}(\mathbf{A})$ is irreducible.
			\item If $\mathbf{A}$ is irreducible, 
			\[
				\mathrm{mc}_{\mathbf{i}}\circ
				\mathrm{mc}_{\mathbf{i}}(\mathbf{A})
				\sim \mathbf{A},
			\]
			i.e., there exists $g\in \mathrm{GL}(n,\mathbb{C})$
			 such that $$\mathrm{mc}_{\mathbf{i}}\circ
				\mathrm{mc}_{\mathbf{i}}(\mathbf{A})
				=g\mathbf{A}g^{-1}:=(gA_{i}(z^{-1})g^{-1})_{
				0\le i\le p}.$$

		\item Let us define elements in $M( n'_{[i,j]},\mathbb{C})$
			by
			\[
				R'_{[i,j]}:=
				\begin{cases}
					R_{[i,j]}+
				(d_{i}(j,j_{i})+2)
				\xi_{\mathbf{i}}I_{n_{[i,j]}}&
				\text{if }i\neq 0,\\	
				R_{[0,j]}+
				d_{0}(j,j_{0})
				\xi_{\mathbf{i}}I_{n_{[0,j]}}&
				\text{if }i=0
			\end{cases}
			\]
			for all $i\in \{0,\ldots,p\}$ and 
			$j\in\{1,\ldots,m_{i}\}\backslash\{j_{i}\}$.
			Here we put $n'_{[i,j]}:=n_{[i,j]}$.
			
			Further define $R'_{[i,j_{i}]}
			\in M( n'_{[i,j_{i}]},\mathbb{C})$ 
			for $i=1,\ldots,p$ so that 
			equations hold,
			\begin{align*}
				&\mathrm{rank\,}
				( R'_{[i,j_{i}]}-\xi^{[i,j_{i}]}_{1})
				=	\mathrm{rank\,}
				( R_{[i,j_{i}]}-\xi^{[i,j_{i}]}_{1}),\\
				&\begin{aligned}
				\mathrm{rank\,}
				( R'_{[i,j_{i}]}-\xi^{[i,j_{i}]}_{1})
				&\prod_{k=2}^{l}
				( R'_{[i,j_{i}]}-\xi^{[i,j_{i}]}_{k}-
				\xi_{\mathbf{i}})\\
				&=
				\mathrm{rank\,}
				\prod_{k=1}^{l}
				(R_{[i,j_{i}]}-\xi_{k}^{[i,j_{i}]}),\quad
				l=2,\ldots,e_{[i,j_{i}]}.
			\end{aligned}
			\end{align*}
			Similarly define  
			$R'_{[0,j_{0}]}
			\in M( n'_{[0,j_{0}]},\mathbb{C})$ 
			so that 
			the following equations hold,
			\begin{align*}
				&\mathrm{rank\,}
				( R'_{[0,j_{0}]}-\xi^{[0,j_{0}]}_{1}+2
				\xi_{\mathbf{i}})=
				\mathrm{rank\,}
				(R_{[0,j_{0}]}-\xi_{1}^{[0,j_{0}]}),\\
				&\begin{aligned}
				\mathrm{rank\,}
				( R'_{[0,j_{0}]}-\xi^{[0,j_{0}]}_{1}+2
				\xi_{\mathbf{i}})
				&\prod_{k=2}^{l}
				( R'_{[0,j_{0}]}-\xi^{[0,j_{0}]}_{k}+
				\xi_{\mathbf{i}})\\
				&=
				\mathrm{rank\,}
				\prod_{k=1}^{l}
				(R_{[0,j_{0}]}-\xi_{k}^{[0,j_{0}]}),\quad
				l=2,\ldots,e_{[0,j_{0}]}.
			\end{aligned}
			\end{align*}
			Here we put 
			\[
				n'_{[i,j_{i}]}:=
				n_{[i,j_{i}]}+\mathrm{dim}_{\mathbb{C}}W
				-2n.
			\]
			Finally define 
			\begin{multline*}
				B'_{i}:=\\
				\mathrm{diag}\left(
					q_{[i,1]}(z^{-1})
					I_{n'_{[i,1]}}+
					R'_{[i,1]}z^{-1},\ldots,
					q_{[i,m_{i}]}(z^{-1})
					I_{n'_{[i,m_{i}]}}
					+R'_{[i,m_{i}]}
				z^{-1}
				\right)
			\end{multline*}
			for $i=0,\ldots,p$.
			
			Then $\mathrm{mc}_{\mathbf{i}}(\mathbf{A})
			\in \prod_{i=0}^{p}\mathcal{O}_{B'_{i}}$.
	\end{enumerate}
\end{prop}
\begin{proof}
	Corollary 4 in \cite{Y1} shows (1). 
	Lemma 11, Remark 14 and Proposition 16 in \cite{Y1} show 
	(2).
\end{proof}

By this proposition, we obtain the bijection 
\[
	\mathrm{mc}_{\mathbf{i}}\colon \mathfrak{M}(\mathbf{B})\longrightarrow
	\mathfrak{M}(\mathbf{B}')
\]
where $\mathbf{B}':=(B'_{i})_{0\le i\le p}$ defined in the proposition.

Proposition \ref{middleconv} shows the following.
\begin{prop}\label{middletoroot}
	Let $\xi$ and $\mathfrak{M}_{\lambda}(\mathsf{Q},\alpha)^{
		\text{dif}}$ be  same  as in 
	Theorem \ref{irreducibleandquasi}.
	Suppose that
	we can choose $\mathbf{i}=([i,j_{i}])_{0\le i\le p}
	\in \mathcal{J}$ so that 
	$$\lambda_{\mathbf{i}}
	:=\sum_{i\in I_{\text{irr}}}\lambda_{[i,j_{i}]}
	=-\xi_{\mathbf{i}}\neq 0.$$
	Define $\mathrm{mc}_{\mathbf{i}}(\alpha)
	:=(\alpha'_{a})_{a\in \mathsf{Q}_{0}}
	\in \mathbb{Z}^{\mathsf{Q}_{0}}$ and $
	\mathrm{mc}_{\mathbf{i}}(\lambda):=(\lambda'_{a})_{a\in \mathsf{Q}_{0}}
	\in \mathbb{C}^{\mathsf{Q}_{0}}$ by
	\begin{align*}
		\alpha'_{[i,j]}&:=
		\begin{cases}
			\alpha_{[i,j]}&\text{if }j\neq j_{i},\\
			\alpha_{[i,j_{i}]}+n_{\mathbf{i}}&
			\text{if }j=j_{i},
		\end{cases}\\
		\alpha'_{[i,j,k]}&:=\alpha_{[i,j,k]},\\
		\lambda'_{[i,j]}&:=
		\begin{cases}
			\lambda_{[i,j_{i}]}
			&\text{if }
			[i,j]=[i,j_{i}]\,\text{ and }\,i\neq 0,\\
			\lambda_{[0,j_{0}]}-2\lambda_{\mathbf{i}}&
			\text{if }
			[i,j]=[0,j_{0}],\\
			\lambda_{[i,j]}+(d_{i}(j,j_{i})+2)
			\lambda_{\mathbf{i}}&
			\text{if }i\neq 0\,\text{ and }\,j\neq m_{i},\\
			\lambda_{[0,j]}+d_{0}(j,j_{0})\lambda_{\mathbf{i}}
			&\text{if }i=0\,\text{ and }\,j\neq m_{0},
		\end{cases}\\
		\lambda'_{[i,j,k]}&:=
		\begin{cases}
			\lambda_{[i,j,k]}&
			\text{if }
			[i,j,k]\neq [i,j_{i},1],\\
			\lambda_{[i,j_{i},1]}+\lambda_{\mathbf{i}}.
		\end{cases}
	\end{align*}
	Here
	\begin{equation*}
		n_{\mathbf{i}}:=\sum_{i\in I_{\text{irr}}}
		\sum_{j=1}^{m^{(i)}}(d_{i}(j,j_{i})+1)
		\alpha_{[i,j]}+
		\sum_{i\in I_{\text{irr}}}((n-\alpha_{[i,j_{i}]})
		+\alpha_{[i,j_{i},1]})
		+\sum_{i\in I_{\text{reg}}}\alpha_{[i,1,1]}
		-2n.
	\end{equation*}
	Then there exists a bijection
	\[
		\mathrm{mc}_{\mathbf{i}}\colon 
		\mathfrak{M}_{\lambda}(\mathsf{Q},\alpha)^{\text{dif}}
		\longrightarrow
		\mathfrak{M}_{\mathrm{mc}_{\mathbf{i}}(\lambda)}
		(\mathsf{Q},\mathrm{mc}_{\mathbf{i}}(\alpha))^{\text{dif}}.
	\]  
\end{prop}
\begin{proof}
	We already have the bijection
	\[
		\mathrm{mc}_{\mathbf{i}}\colon \mathfrak{M}(\mathbf{B})
		\longrightarrow \mathfrak{M}(\mathbf{B}')
	\]
	for $\mathbf{B}'$ defined as in Proposition \ref{middleconv}.
	
	If we choose $\xi'=\left(\left\{
		(\xi')^{[i,j]}_{1},\ldots,(\xi')^{[i,j]}_{e_{[i,j]}}
	\right\}\right)_{\substack{0\le i\le p\\
		1\le j\le m^{(i)}}}$
	so that 
	\begin{align*}
		(\xi')^{[i,j]}_{k}=&
		\begin{cases}
			\xi^{[i,j]}_{k}+(d_{i}(j,j_{i})+2)\xi_{\mathbf{i}}&
			\text{if }i\neq 0\text{ and }j\neq j_{i},\\
			\xi^{[0,j]}_{k}+d_{0}(j,j_{0})\xi_{\mathbf{i}}&
			\text{if }i=0\text{ and }j\neq j_{0},\\
			\xi^{[i,j_{i}]}_{1}&\text{if }i\neq 0,\,j=j_{i}
			\text{ and }k=1,\\
			\xi^{[i,j_{i}]}_{k}+\xi_{\mathbf{i}}&
			\text{if }i\neq 0,\,j=j_{i}
			\text{ and }k\neq 1,\\
			\xi^{[0,j_{0}]}_{1}-2\xi_{\mathbf{i}}&\text{if }
			i=0,\,j=j_{0}\text{ and }k=1,\\
			\xi^{[0,j_{0}]}_{k}-\xi_{\mathbf{i}}&\text{if }
			i=0,\,j=j_{0}\text{ and }k\neq 1,
		\end{cases}
	\end{align*}
	then we have the bijection 
	\[
		\Phi_{\xi'}\colon \mathfrak{M}(\mathbf{B}')\longrightarrow
		\mathfrak{M}_{\mathrm{mc}_{\mathbf{i}}(\lambda)}
		(\mathsf{Q},\mathrm{mc}_{\mathbf{i}}(\alpha))^{\text{dif}}
	\]
	by Theorem \ref{irreducibleandquasi}.
	Thus we have the bijection 
	$\Phi_{\xi'}\circ\mathrm{mc}_{\mathbf{i}}\circ 
	\Phi^{-1}_{\xi}\colon \mathfrak{M}_{\lambda}(\mathsf{Q},\alpha)^{\text{dif}}
		\longrightarrow
		\mathfrak{M}_{\mathrm{mc}_{\mathbf{i}}(\lambda)}
		(\mathsf{Q},\mathrm{mc}_{\mathbf{i}}(\alpha))^{\text{dif}}.$

	The last equation is obtained by Lemma \ref{sizeW} as follows.
	Since $n_{\mathbf{i}}=n'_{[i,j_{i}]}-n_{[i,j_{i}]}$,
	we have
	\begin{align*}
		n_{\mathbf{i}}&=n'_{[i,j_{i}]}-n_{[i,j_{i}]}
		=\mathrm{dim}_{\mathbb{C}}W-2n\\
		&=\sum_{i\in I_{\text{irr}}}\sum_{\substack{1\le j\le m_{i}\\
		j\neq j_{i}}}d_{i}(j,j_{i})\alpha_{[i,j]}
		+\sum_{i\in I_{\text{irr}}}(2(n-\alpha_{[i,j_{i}]})+
		\alpha_{[i,j_{i},1]})\\
		&\quad+\sum_{i\in I_{\text{reg}}}\alpha_{[i,1,1]}-2n\\
		&=\sum_{i\in I_{\text{irr}}}
		\sum_{j=1}^{m_{i}}(d_{i}(j,j_{i})+1)
		\alpha_{[i,j]}+
		\sum_{i\in I_{\text{irr}}}((n-\alpha_{[i,j_{i}]})
		+\alpha_{[i,j_{i},1]})\\
		&\quad+\sum_{i\in I_{\text{reg}}}\alpha_{[i,1,1]}
		-2n.
	\end{align*}
\end{proof}
\subsection{More operations on 
$\mathfrak{M}_{\lambda}(\mathsf{Q},\alpha)^{\text{dif}}$}\label{operations}
Proposition \ref{middletoroot} 
enable us to define $\mathrm{mc}_{\mathbf{i}}$ as an operation
for $\mathfrak{M}_{\lambda}(\mathsf{Q},\alpha)^{\text{dif}}$.
We shall introduce some other operations. 

Let us note that the map $\Phi_{\xi}$ in Proposition \ref{irregularquiver} 
depends on the order of 
$\{\xi^{[i,j]}_{1},\ldots,\xi^{[i,j]}_{e_{[i,j]}}\}$ for 
$i=0,\ldots,p$ and $j=1,\ldots,m_{i}$.
Thus we shall see what happens when we change the order.
Let us define 
\[
	\sigma^{[i_{0},j_{0}]}_{s}(\xi):=
(\{\zeta^{[i,j]}_{1},\ldots,\zeta^{[i,j]}_{e_{[i,j]}}\})
_{\substack{0\le i\le p\\1\le j\le m_{i}}}
\]
by
\begin{align*}
	\zeta^{[i,j]}_{l}:=
	\begin{cases}
		\xi^{[i,j]}_{l}&\text{if }[i,j]\neq [i_{0},j_{0}],\\
		\xi^{[i_{0},j_{0}]}_{\sigma_{s}(l)}&\text{if }[i,j]=
		[i_{0},j_{0}]
	\end{cases}
\end{align*}
where $\sigma_{s}$ is the permutation $(s,s+1)$.
Then we can extend these permutations to operations on 
representations of $\overline{\mathsf{Q}}$,
\[
	\sigma^{[i_{0},j_{0}]}_{s}\colon 
	\mathfrak{M}_{\lambda}(\mathsf{Q},\alpha)^{\text{dif}}
	\longrightarrow
	\mathfrak{M}_{\lambda'''}(\mathsf{Q},\alpha''')^{\text{dif}}
\]
defined by $\Phi_{\sigma^{[i_{0},j_{0}]}_{s}(\xi)}\circ \Phi_{\xi}^{-1}$.
Here $\alpha'''$ and $\lambda'''$ can be computed as below.
\begin{prop}\label{changeorder}
The above $\alpha'''$ and $\lambda'''$ are defined as follows,
	\begin{align*}
		\alpha'''&=\begin{cases}
			\alpha&\text{if }\xi^{[i_{0},j_{0}]}_{s}=
			\xi^{[i_{0},j_{0}]}_{s+1},\\
			s_{[i_{0},j_{0},s]}(\alpha)
			&\text{otherwise},
		\end{cases}\\
		\lambda'''&=\begin{cases}
			\lambda&\text{if }\xi^{[i_{0},j_{0}]}_{s}=
			\xi^{[i_{0},j_{0}]}_{s+1},\\
			r_{[i_{0},j_{0},s]}(\lambda)&\text{otherwise}.
		\end{cases}
	\end{align*}
\end{prop}
\begin{proof}
	If $\xi^{[i_{0},j_{0}]}_{s}=
			\xi^{[i_{0},j_{0}]}_{s+1}$, then 
			$\Phi_{\sigma^{[i_{0},j_{0}]}_{s}(\xi)}
			\circ \Phi_{\xi}^{-1}=\Phi_{\xi}\circ \Phi_{\xi}^{-1}=
			\mathrm{id}$.
	Thus we obtain the result.		
	Next  we suppose 
	$\xi^{[i_{0},j_{0}]}_{s}\neq \xi^{[i_{0},j_{0}]}_{s+1}$.
	To see that $\alpha'=s_{[i_{0},j_{0},s]}(\alpha)$,
	it suffices to show that 
	\begin{multline*}
		\mathrm{rank\,}\prod_{l=1}^{s-1}(R_{[i_{0},j_{0}]}-
		\xi^{[i_{0},j_{0}]}_{l})-
		\mathrm{rank\,}\prod_{l=1}^{s}(R_{[i_{0},j_{0}]}-
		\xi^{[i_{0},j_{0}]}_{l})\\
		=\mathrm{rank\,}(R_{[i_{0},j_{0}]}
		-\xi^{[i_{0},j_{0}]}_{s+1})\prod_{l=1}^{s-1}
		(R_{[i_{0},j_{0}]}-
		\xi^{[i_{0},j_{0}]}_{l})-
		\mathrm{rank\,}\prod_{l=1}^{s+1}(R_{[i_{0},j_{0}]}-
		\xi^{[i_{0},j_{0}]}_{l}),
	\end{multline*}
	and 
	\begin{multline*}
		\mathrm{rank\,}\prod_{l=1}^{s}(R_{[i_{0},j_{0}]}-
		\xi^{[i_{0},j_{0}]}_{l})-
		\mathrm{rank\,}\prod_{l=1}^{s+1}(R_{[i_{0},j_{0}]}-
		\xi^{[i_{0},j_{0}]}_{l})\\
		=\mathrm{rank\,}\prod_{l=1}^{s-1}(R_{[i_{0},j_{0}]}-
		\xi^{[i_{0},j_{0}]}_{l})-
		\mathrm{rank\,}(R_{[i_{0},j_{0}]}-\xi^{[i_{0},j_{0}]}_{s+1})
		\prod_{l=1}^{s-1}(R_{[i_{0},j_{0}]}-
		\xi^{[i_{0},j_{0}]}_{l}).
	\end{multline*}
These equations follow from the following fact.
Let us suppose that $A,A'\in M(n,\mathbb{C})$ satisfy $AA'=A'A $ and 
$\mathrm{Ker\,}A\cap \mathrm{Ker\,}A'=\{0\}$. Let $V$ be an 
$A'$-invariant
subspace of $\mathbb{C}^{n}$. 
Then we have 
\begin{equation}\label{eq}
	\mathrm{dim}_{\mathbb{C}}V-\mathrm{dim}_{\mathbb{C}}A'V=
	\mathrm{dim}_{\mathbb{C}}AV-\mathrm{dim}_{\mathbb{C}}AA'V.
\end{equation}
Indeed, setting $\widetilde{V}=V\cap\mathrm{Ker\,}A$ and $\widetilde{W}
=A'V\cap \mathrm{Ker\,}A$, we have that $A'$ gives an injection 
from $\widetilde{V}\rightarrow \widetilde{W}$ since 
$\mathrm{Ker\,}A\cap \mathrm{Ker\,}A'=\{0\}$.
Since $\widetilde{W}\subset 
\widetilde{V}$, this implies that $\widetilde{W}=\widetilde{V}$,
which shows the equation $(\ref{eq})$.

Since $\xi^{[i_{0},j_{0}]}_{s}\neq \xi^{[i_{0},j_{0}]}_{s+1}$,
$(R_{[i_{0},j_{0}]}-\xi^{[i_{0},j_{0}]}_{s})$ and $(R_{[i_{0},j_{0]}}
-\xi^{[i_{0},j_{0}]}_{s+1})$ satisfy the above assumption.
Then equation $(\ref{eq})$ gives the required equations.

The remaining assertion
follows from a direct computation.
\end{proof}

Furthermore let us introduce an operation which is trivial on $\mathfrak{M}(
\mathbf{B})$. 
Let $t_{(i,0)}$ be the permutation $(i,0)$ for
$i\in I_{\text{irr}}\backslash\{0\}$. 
Define 
\[
	t_{(i,0)}(\mathbf{B}):=(B_{t_{(i,0)}(j)})_{j=0,\ldots,p}.
\]
Then we can define the bijection
\[
	\begin{array}{rccc}
		T_{(i,0)}\colon&\mathfrak{M}(\mathbf{B})&
		\longrightarrow&
		\mathfrak{M}(t_{(i,0)}(\mathbf{B}))\\
		&(A_{j}(z^{-1}))_{0\le j\le p}&
		\longmapsto
		&(A_{t_{(i,0)}(j)}(z^{-1}))_{0\le j\le p}
	\end{array}.
\]
Then $t_{(i,0)}(\xi)$, $t_{(i,0)}(\mathsf{Q})$, $t_{(i,0)}(\alpha)$ and 
$t_{(i,0)}(\lambda)$ can be defined by replacing $i$ with $0$ and vice versa.
Thus we can define the bijection
\[
	T_{(i,0)}\colon \mathfrak{M}_{\lambda}(\mathsf{Q},\alpha)^\text{dif}
	\longrightarrow
	\mathfrak{M}_{t_{(i,0)}(\lambda)}(t_{(i,0)}(\mathsf{Q}),t_{(i,0)}(
	\alpha))^{\text{dif}}
\]
by $\Phi_{t_{(i,0)}(\xi)}\circ T_{(i,0)}\circ \Phi^{-1}_{\xi}$.
\subsection{Middle convolution and reflection}
As we saw in the above proposition which shows that 
permutations on $\xi$ can be obtained by 
reflection,
middle convolutions $\mathrm{mc}_{\mathbf{i}}$  can also be obtained by 
reflections as follows.

For $\mathbf{i}=([i,j_{i}])_{0\le i\le p}
\in \mathcal{J}$, let us define $\epsilon_{\mathbf{i}}
\in \mathbb{Z}^{\mathsf{Q}_{0}}$ by
\[
	(\epsilon_{\mathbf{i}})_{a}:=
	\begin{cases}
		1&\text{if }a=[i,j_{i}],\,i\in I_{\text{irr}},\\
		0&\text{otherwise}.
	\end{cases}
\]
We note that $\epsilon_{\mathbf{i}}$ for $\mathbf{i}\in \mathcal{J}$
are positive real roots of $\mathsf{Q}$.
Let us define 
\[
	s_{\mathbf{i}}(\beta):=\beta-(\beta,\epsilon_{\mathbf{i}})
\epsilon_{\mathbf{i}}
\]
for $\mathbf{i}\in \mathcal{J}$ and $\beta\in \mathbb{Z}^{\mathsf{Q}_{0}}$.

Let us see this reflection $s_{\mathbf{i}}$ can be obtained by a 
product of simple reflections.
\begin{lem}\label{modifiedreflection}
	Let us take $\mathbf{i}=([i,j_{i}])_{0\le i\le p}\in \mathcal{J}$.
	Then we have 
	\[
		\left(\prod_{i\in I_{\text{irr}}\backslash\{0\}}
		s_{[i,j_{i}]}\right)\circ
		s_{[0,j_{0}]}\circ
		\left(\prod_{i\in I_{\text{irr}}\backslash\{0\}}
		s_{[i,j_{i}]}\right)(\beta)
		=s_{\mathbf{i}}(\beta)
	\]
	for any $\beta\in \mathbb{Z}^{\mathsf{Q}_{0}}$.
\end{lem}
\begin{proof}
Set $r:=\prod_{i\in I_{\text{irr}}\backslash\{0\}}
		s_{[i,j_{i}]}$ for short.
Note that $r$ is an involution and  
$\epsilon_{\mathbf{i}}=r(\epsilon_{[0,j_{0}]})$.		
Then 
\begin{align*}
	r\circ s_{[0,j_{0}]}\circ r(\beta)&=
	r(r(\beta)-(r(\beta),\epsilon_{[0,j_{0}]})\epsilon_{[0,j_{0}]})\\
	&=r^{2}(\beta)-(\beta,r^{-1}(\epsilon_{[0,j_{0}]}))
	r(\epsilon_{[0,j_{0}]})\\
	&=\beta-(\beta,r(\epsilon_{[0,j_{0}]}))r(\epsilon_{[0,j_{0}]})\\
	&=\beta-(\beta,\epsilon_{\mathbf{i}})\epsilon_{\mathbf{i}}\\
	&=s_{\mathbf{i}}(\beta).
\end{align*}
\end{proof}

This lemma tells us that  
$\mathrm{mc}_{\mathbf{i}}$ can be regarded as a reflection 
and a product of simple reflections as below.
\begin{prop}\label{middletoweyl}
	Retain the notation in Proposition \ref{middletoroot}.
	Then we have
	\begin{align*}
		\mathrm{mc}_{\mathbf{i}}(\alpha)
		&=s_{\mathbf{i}}(\alpha)\\
		&=\left(\prod_{i\in I_{\text{irr}}\backslash\{0\}}
		s_{[i,j_{i}]}\right)\circ
		s_{[0,j_{0}]}\circ
		\left(\prod_{i\in I_{\text{irr}}\backslash\{0\}}
		s_{[i,j_{i}]}\right)(\alpha).
	\end{align*}
\end{prop}
\begin{proof}
	From the definition of $\mathrm{mc}_{\mathbf{i}}(\alpha)$
	given in Proposition \ref{middletoroot}, it suffices to show 
	\[
		n_{\mathbf{i}}=-(\alpha,\epsilon_{\mathbf{i}}).
	\]
	Indeed 
	\begin{align*}
		(\alpha,\epsilon_{\mathbf{i}})&=\sum_{i\in I_{\text{irr}}}
		(\alpha,\epsilon_{[i,j_{i}]})\\
		&=\sum_{i\in I_{\text{irr}}\backslash\{0\}}
		\left(2\alpha_{[i,j_{i}]}-\sum_{\substack{1\le j\le m_{i}\\
			j\neq j_{i}}}d_{i}(j,j_{i})\alpha_{[i,j]}
			-\alpha_{[i,j_{i},1]}-\sum_{j=1}^{m_{0}}
			\alpha_{[0,j]}
		\right)\\
		&\quad +
		2\alpha_{[0,j_{0}]}-\sum_{\substack{1\le j\le m_{0}\\
		j\neq j_{0}}}d_{0}(j,j_{0})\alpha_{[0,j]}
		-\alpha_{[0,j_{0},1]}
		-\sum_{i\in I_{\text{irr}\backslash\{0\}}}\sum_{j=1}^{m_{i}}
		\alpha_{[i,j]}\\
		&\quad -\sum_{i\in I_{\text{reg}}\alpha_{[i,1,1]}}
		\alpha_{[i,j]}.
	\end{align*}
	Recalling that $d_{i}(j,j)=-1$ and $\sum_{j=1}^{m_{i}}\alpha_{[i,j]}=n$,
	we can continue the above computation,
	\begin{align*}
		(\alpha,\epsilon_{\mathbf{i}})&=
		-\sum_{i\in I_{\text{irr}}}\left(
		\sum_{j=1}^{m_{i}}(d_{i}(j,j_{i})\alpha_{[i,j]})
		+(n-\alpha_{[i,j_{i}]})+\alpha_{[i,j_{i},1]}\right)\\
		&\quad -\sum_{i\in I_{\text{reg}}}\alpha_{[i,1,1]}
		-(\#I_{\text{irr}}-2)n\\
		&=-\sum_{i\in I_{\text{irr}}}\left(
		\sum_{j=1}^{m_{i}}(d_{i}(j,j_{i}+1)\alpha_{[i,j]})
		+(n-\alpha_{[i,j_{i}]})+\alpha_{[i,j_{i},1]}\right)\\
		&\quad -\sum_{i\in I_{\text{reg}}}\alpha_{[i,1,1]}
		+2n\\
		&=-n_{\mathbf{i}}.
	\end{align*}
\end{proof}

These observations lead us to define transformations on 
$\mathbb{Z}^{\mathsf{Q}_{0}}\times \mathbb{C}^{\mathsf{Q}_{0}}$ as an 
analogy of middle 
convolutions and other operations.
For 
$(\beta,\nu)\in \mathbb{Z}^{\mathsf{Q}_{0}}\times \mathbb{C}^{\mathsf{Q}_{0}}$ 
with $\nu_{\mathbf{i}}=\sum_{i\in I_{\text{irr}}}\nu_{[i,j_{i}]}\neq 0$
define 
\[s_{\mathbf{i}}( (\beta,\nu)):=(\mathrm{mc}_{\mathbf{i}}
		(\beta'),\mathrm{mc}_{\mathbf{i}}(\nu')).
\]
Also for $(\beta,\nu)\in \mathbb{Z}^{\mathsf{Q}_{0}}\times \mathbb{C}^{\mathsf{Q}_{0}}$ with $\nu_{[i,j,k]}\neq 0$, define 
\[
	s_{[i,j,k]}( (\beta,\nu)):=(s_{[i,j,k]}(\beta),r_{[i,j,k]}(\nu)).
\]

Let us define
\[
	\mathcal{S}:=\left\{
		(\beta,\nu)\in \mathcal{L}\times \mathbb{C}^{\mathsf{Q}_{0}}
		\,\bigg|\,
		\beta\cdot \nu=\sum_{a\in \mathsf{Q}_{0}}\beta_a\nu_{a}=0
	\right\}.
\]
Then we can see that 
$s_{\mathbf{i}}$ and $s_{[i,j,k]}$ preserve $\mathcal{S}$,
see Propositions \ref{middletoroot} and \ref{changeorder}.

\subsection{Irreducibility and $\mathcal{L}$-irreducibility}
The $\mathcal{L}$-irreducibility is a weaker condition than the 
usual irreducibility. We shall show that if we shift the parameter $\lambda$
by using the operation $\mathrm{Add}$, then 
these two irreducibility can be identical.

Fix $i_{0}\in I_{\text{irr}}\backslash\{0\}$ and define an operation on 
$\mathcal{S}$ as an analogue of 
$\mathrm{Add}^{(i_{0})}_{z^{-1}}\circ\mathrm{Add}^{(0)}_{-z^{-1}}$
as follows.
Let us define $z^{(i_{0})}=(z_{a}^{(i_{0})})_{a\in \mathsf{Q}_{0}}
\in \mathbb{C}^{\mathsf{Q}_{0}}$ by 
\begin{align*}
	z^{(i_{0})}_{[i,j]}&=
	\begin{cases}
		1&\text{if }i=i_{0},\\
		-1&\text{if }i=0,\\
		0&\text{otherwise},
	\end{cases}\\
	z^{(i_{0})}_{[i,j,k]}&=0.
\end{align*}
Then let us define
\[
	\begin{array}{cccc}
		\mathrm{add}^{(i_{0})}_{\gamma}\colon
		&\mathcal{S}&\longrightarrow&\mathcal{S}\\
		&(\beta,\nu)&\longmapsto
		&(\beta,\nu+\gamma z^{(i_{0})})
	\end{array}
\]
for $i_{0}\in I_{\text{irr}}\backslash\{0\}$ and $\gamma\in \mathbb{C}$.

For $\nu\in \mathbb{C}^{\mathsf{Q}_{0}}$, 
let $R_{\nu}^{+}$ be the set of positive roots $\beta$ of $\mathsf{Q}$
satisfying $\beta\cdot \nu=0$.
Denote $\mathcal{L}\cap R^{+}_{\nu}$ by $\widetilde{R}^{+}_{\nu}$.
The subset  $\Sigma_{\nu}$ of $R_{\nu}^{+}$ consists 
of $\beta$ satisfying that  
$p(\beta)>\sum_{t} p(\beta_{t})$ for 
any decomposition $\beta=\beta_{1}+\cdots +\beta_{r}$ with 
$r\ge 2$ and $\beta_{t}\in R^{+}_{\nu}$. 
Similarly define $\widetilde{\Sigma}_{\nu}$ consisting of 
$\beta\in \widetilde{R}^{+}_{\nu}$ satisfying that 
$\beta\cdot \nu=0$ and $p(\beta)>\sum_{t} p(\beta_{t})$ for 
any decomposition $\beta=\beta_{1}+\cdots +\beta_{r}$
with $r\ge 2$ and $\beta_{t}\in \widetilde{R}^{+}_{\nu}$. 

If $\beta,\beta'\in (\mathbb{Z}_{\ge 0})^{\mathsf{Q}_{0}}$ satisfy that 
$\beta'_{a}\le \beta_{a}$ for all $a\in \mathsf{Q}_{0}$, 
then we write $\beta'\le \beta$.

\begin{lem}\label{slide}
	Fix $(\beta,\nu)\in \mathcal{S}$. There exist $\gamma_{i}\in 
	\mathbb{C}$ for  $i\in 
	I_{\text{irr}}\backslash\{0\}$ such that $\nu'=
	\nu+\sum_{i\in I_{\text{irr}}\backslash\{0\}}\gamma_{i}z^{(i)}$
	satisfies the following.
	If $\beta'\in (\mathbb{Z}_{\ge 0})^{\mathsf{Q}_{0}}$ satisfies
	that $\beta'\le \beta$ and $\beta'\cdot \nu'=0$,
	then $\beta'\in \mathcal{L}$.
\end{lem}
\begin{proof}
	Let $F_{\beta}$ be the set of all elements $\beta'$ in
	$(\mathbb{Z}_{\ge 0})^{\mathsf{Q}_{0}}$ satisfying 
	$\beta' \le \beta$ and $\beta'\notin \mathcal{L}$.
	Note that $F_{\beta}$ is a finite set.
	Define a closed subset of $\mathbb{C}^{\mathsf{Q}_{0}}$ by
	\[
		V_{\beta}:=\bigcup_{\beta'\in F_{\beta}}
		\{\eta\in \mathbb{C}^{\mathsf{Q}_{0}}\mid \beta'\cdot\eta=0
		\}.
	\]
	Namely, if $\nu\notin V_{\beta}$, then $\beta'\cdot \nu=0$  
	and  $\beta'\le \beta$ imply $\beta'\in \mathcal{L}$.
	Thus let us suppose $\nu \in V_{\beta}$.
	Consider the affine space  
	$W_{\nu}:=\{\nu+\sum_{i\in I_{\text{irr}}\backslash\{0\}}
	t_{i}z^{(i)}\mid t_{i}\in \mathbb{C}\}$. 
	Then 
	$W_{\nu}\cap 
	\{\eta\in \mathbb{C}^{\mathsf{Q}_{0}}\mid \beta'\cdot\eta=0\}$
	is a proper closed subset of $W_{\mu}$ for any $\beta'\in F_{\beta}$.
	Indeed, 
	since $\beta'\notin\mathcal{L}$ there exists $i_{0}\in I_{\text{irr}}
	\backslash\{0\}$ such that $\sum_{j=1}^{m^{(0)}}\beta'_{[0,j]}
	\neq \sum_{j=1}^{m^{(i_{0})}}\beta'_{[i_{0},j]}$.
	Then 
	the line $\{\nu+tz^{(i_{0})}\mid t\in \mathbb{C}\}\subset W_{\nu}$ 
	is not contained in the hyperplane
	$\{\eta\in \mathbb{C}^{\mathsf{Q}_{0}}\mid \beta'\cdot\eta=0\}$.
	Thus $\mathrm{dim\,}W_{\nu}>\mathrm{dim\,}\left(W_{\nu}\cap 
	\{\eta\in \mathbb{C}^{\mathsf{Q}_{0}}\mid \beta'\cdot\eta=0\}\right)$
	for any $\beta'\in F_{\beta}$ since $W_{\mu}$ is an irreducible 
	algebraic set. This shows the inequality,
	\begin{equation*}
		\mathrm{dim\,}W_{\nu}\cap V_{\beta}=
		\mathrm{max}_{\beta'\in F_{\beta}}
		\left\{\mathrm{dim\,}\left(W_{\nu}
			\cap 
	\{\eta\in \mathbb{C}^{\mathsf{Q}_{0}}\mid \beta'\cdot\eta=0\}\right)
	\right\}
	<\mathrm{dim\,}W_{\nu}.
	\end{equation*}
	Hence there exists $\nu'\in W_{\nu}$ 
	which is not contained in $V_{\beta}$ 
	as required.
\end{proof}

\begin{lem}\label{slide1}
	We have $\widetilde{\Sigma}_{\nu}=\widetilde{\Sigma}_{\nu+\gamma z^{(i)}}$
	for any $i\in I_{\text{irr}}\backslash\{0\}$ and 
	$\gamma \in \mathbb{C}$.
\end{lem}
\begin{proof}
	Recall that $(\beta,\nu)\in \mathcal{S}$ if and only if 
	$(\beta,\nu+\gamma z^{(i)})\in \mathcal{S}$.
\end{proof}

Obviously $\mathcal{L}\cap \Sigma_{\nu}\subset\widetilde{\Sigma}_{\nu}$.
The above lemmas show that for any $\beta\in \widetilde{\Sigma}_{\nu}$,
there exists $\nu'\in \mathbb{C}^{\mathsf{Q}_{0}}$ 
such that $\beta\in \Sigma_{\nu'}$.
\begin{prop}\label{slide2}
	For any $\beta \in \widetilde{\Sigma}_{\nu}$,
	there exist $\gamma_{i}\in \mathbb{C}$ for 
	$i\in I_{\text{irr}}\backslash\{0\}$ such that 
	$\beta\in \Sigma_{\nu+\sum_{i\in I_{\text{irr}}\backslash\{0\}}
	\gamma_{i} z^{(i)}}$.
\end{prop}
\begin{proof}
	For $\beta\in \widetilde{\Sigma}_{\nu}$, 
	let us choose $\gamma_{i}$ as in Lemma \ref{slide} and 
	set $\nu'=\nu+\sum_{i\in I_{\text{irr}}\backslash\{0\}}
	\gamma_{i} z^{(i)}$.
	Then Lemma \ref{slide} shows that $\beta\in \Sigma_{\nu'}$.
\end{proof}

\begin{lem}\label{birational}
	Suppose that $\mu^{-1}(\lambda)^{\text{dif}}\neq \emptyset$.
	Fix $i_{0}\in I_{\text{irr}}$ and $\gamma\in \mathbb{C}$.
	Then there exists a $\mathbf{G}(\alpha)$-equivariant 
	analytic bijection 
	\[
		\mathrm{add}^{(i_{0})}_{\gamma}\colon 
		\mu_{\alpha}^{-1}(\lambda)^{\text{dif}}\longrightarrow 
		\mu_{\alpha}^{-1}(\lambda+\gamma z^{(i_{0})})^{\text{dif}}.
	\]
\end{lem}
\begin{proof}
	The required map is obtained by 
	$\Phi_{\xi'}\circ \mathrm{Add}^{(i_{0})}_{-\gamma z^{-1}}
	\circ \mathrm{Add}^{(0)}_{\gamma z^{-1}}\circ \Phi_{\xi}^{-1}$
	with suitable $\xi$ and $\xi'$.
	Thus it follows that the map preserves the $\mathcal{L}$-irreducibility
	since $\mathrm{Add}$ preserves the irreducibility of differential
	equations.
	
	We can directly check that 
	for $x\in \mu_{\alpha}^{-1}(\lambda)^{\text{dif}}$,  
	its image $x':=\mathrm{add}^{(i_{0})}_{\gamma}(x)\in 
	\mu_{\alpha}^{-1}(\lambda+\gamma z^{(i_{0})})$  is written as follows.
	Set 
	\begin{align*}
		x_{\rho_{i_{0}}}&:=\left(x_{\rho^{[0,j]}_{[i_{0},j']}}
	\right)_{\substack{1\le j\le m_{0}\\
	1\le j'\le m_{i_{0}}}},& 
	x_{\rho^{*}_{i_{0}}}&:=\left(x_{(\rho^{[0,j]}_{[i_{0},j']})^{*}}
	\right)_{\substack{1\le j\le m_{0}\\
	1\le j'\le m_{i_{0}}}}.
\end{align*}
	Then 
	\[
		x'_{(\rho^{[0,j]}_{[i_{0},j']})^{*}}=
			\left(
				x_{\rho^{*}_{i_{0}}}
				+\gamma\cdot x_{\rho_{i_{0}}}^{-1}
			\right)_{[0,j],[i_{0},j']}
	\]
	for $1\le j\le m_{0}$ and $1\le j'\le m_{i_{0}}$ and 
	\[
		x'_{\rho}=x_{\rho}
	\]
	for the remaining $\rho\in \overline{\mathsf{Q}}_{1}$,
	which tells us that the map is analytic.
\end{proof}

\begin{thm}\label{ifnonempty}
	If $\mu_{\alpha}^{-1}(\lambda)^{\text{dif}}\neq \emptyset$, then
	$\alpha\in \widetilde{\Sigma}_{\lambda}$.
\end{thm}
\begin{proof}
	Let us suppose that there exists 
	an $\mathcal{L}$-irreducible representation $x\in 
	\mu_{\alpha}^{-1}(\lambda)^{\text{det}}$.
	Choose $\gamma_{i}\in\mathbb{C}$ for $i\in I_{\text{irr}}
	\backslash\{0\}$  as in Lemma \ref{slide} and put 
	$\lambda'=\lambda+\sum_{i\in I_{\text{irr}}\backslash\{0\}}
	\gamma_{i}z^{(i)}$.
	Then the operation $\prod_{i\in I_{\text{irr}}\backslash\{0\}}
	\mathrm{add}^{(i)}_{\gamma_{i}}
	$
	sends $x$ to the 
	$\mathcal{L}$-irreducible element 
	$x'\in \mu_{\alpha}^{-1}(\lambda')^{\text{det}}$.
	However Lemma \ref{slide} shows that 
	if an element in 
	$\mu_{\alpha}^{-1}(\lambda')$
	is $\mathcal{L}$-irreducible, then it is irreducible.
	Thus $x'$ is irreducible, which shows $\alpha\in 
	\mathcal{L}\cap \Sigma_{\lambda'}$ by Crawley-Boevey's
	result (see Theorem \ref{CB}).
	Hence $\alpha\in \mathcal{L}\cap \Sigma_{\lambda'}
	\subset \widetilde{\Sigma}_{\lambda'}=\widetilde{\Sigma}_{\lambda}$
	by Lemma \ref{slide1}.
\end{proof}

We close this section by seeing that 
$\mathfrak{M}_{\lambda}(\mathsf{Q},\alpha)^{\text{dif}}$ is a 
connected manifold if it is non-empty.
Thus the moduli space of connections $\mathfrak{M}(\mathbf{B})$ is 
also to be connected complex manifold.

\begin{thm}\label{moduliconnected}
	If $\mathfrak{M}_{\lambda}(\mathsf{Q},\alpha)^{\text{dif}}$ is 
	non-empty, then it is a connected complex manifold.
	Thus $\mathfrak{M}(\mathbf{B})$ is a connected complex manifold
	if it is non-empty, 
	under the identification by the isomorphism $\Phi_{\xi}$ given 
	in Theorem \ref{irreducibleandquasi}.
\end{thm}
\begin{proof}
	Let us take $\gamma_{i}\in \mathbb{C}$ for $i\in I_{\text{irr}}
	\backslash\{0\}$ as in Lemma \ref{slide} and set 
	$\lambda':=\lambda+\sum_{i\in I_{\text{irr}}\{0\}}\gamma_{i}z^{(i)}$.
	Then it suffices to consider $\mathfrak{M}_{\lambda'}(\mathsf{Q},
	\alpha)^{\text{dif}}$ by Lemma \ref{birational}.
	Note that 
	Lemma $\ref{slide}$ shows that the $\mathcal{L}$-irreducibility
	coincides with the  irreducibility in $\mu_{\alpha}^{-1}(\lambda')$.
	Thus $\mathfrak{M}_{\lambda'}(\mathsf{Q},\alpha)^{\text{dif}}$
	becomes an open subset of the complex manifold
	$\mathfrak{M}^{\text{reg}}_{\lambda'}(\mathsf{Q},\alpha)$.
	
	To show the connectedness, let us recall that 
	$\mu_{\alpha}^{-1}(\lambda')$ is 
	irreducible topological space from Theorem 1.2 in \cite{C1}
	and $\mu_{\alpha}^{-1}(\lambda')^{\text{irr}}$ is 
	open in $\mu_{\alpha}^{-1}(\lambda')$.
	Thus 
	\[
		\mu_{\alpha}^{-1}(\lambda')^{\text{dif}}=\left\{
			x\in \mu_{\alpha}{-1}(\lambda')^{\text{irr}}\,\middle|\,
	 \mathrm{det}\left(x_{\rho^{[0,j]}_{[i,j']}}\right)_{
 \substack{1\le j\le m_{0}\\1\le j'\le m_{i}}}\neq 0,\, i\in I_{\text{irr}}
 \backslash\{0\}\right\}
\]
is connected since it is open in the irreducible space 
$\mu_{\alpha}^{-1}(\lambda')$.
\end{proof}
\begin{thm}\label{embedding}
	If $\mathfrak{M}(\mathbf{B})\neq \emptyset$, there exist $\lambda'
	\in \mathbb{C}^{\mathsf{Q}_{0}}$ and 
	the injection 
	\[
		\Phi\colon \mathfrak{M}(\mathbf{B})\hookrightarrow 
		\mathfrak{M}_{\lambda'}^{\text{reg}}(\mathsf{Q},\alpha)
	\]
	whose image is $\left(
		\mu_{\alpha}^{-1}(\lambda')^{\text{det}}\cap
	\mu_{\alpha}^{-1}(\lambda')^{\text{irr}}\right)/\mathbf{G}(\alpha)$.
\end{thm}
\begin{proof}
	Let us choose $\lambda'$ as in the proof of 
	Theorem \ref{moduliconnected}. Then
	$\mathfrak{M}(\mathbf{B})\cong 
	\mathfrak{M}_{\lambda}(\mathsf{Q},\alpha)^{\text{dif}}
	\cong \mathfrak{M}_{\lambda'}(\mathsf{Q},\alpha)^{\text{dif}}
	=\left(
		\mu_{\alpha}^{-1}(\lambda')^{\text{det}}\cap
	\mu_{\alpha}^{-1}(\lambda')^{\text{irr}}\right)/\mathbf{G}(\alpha)
	\subset \mathfrak{M}_{\lambda'}^{\text{reg}}(\mathsf{Q},\alpha)$.
\end{proof}
\section{$\mathcal{L}$-fundamental set}\label{modifset}
This section and the next one are dedicated to 
show the converse of  Theorem \ref{ifnonempty}.
Namely we shall show that if $\alpha\in \widetilde{\Sigma}_{\lambda}$,
then $\mu_{\alpha}^{-1}(\lambda)^{\text{dif}}\neq 
\emptyset$.
For this purpose, first let us introduce an analogue of the fundamental set $F$.
\begin{df}[$\mathcal{L}$-fundamental set]\normalfont
	Let us define the subset of $\mathcal{L}$ by
	\[
	\tilde{F}:=\left\{\beta\in \mathcal{L}^{+}\backslash\{0\}
	\,\bigg|\, \begin{array}{c}
		(\beta,\epsilon_{a})\le 0\text{ for all }
		a\in \mathcal{J}\cup \mathsf{Q}_{0}^{\text{leg}},\\
		\text{support of $\beta$ is connected}
	\end{array}\right\}
	\]
	and call {\em $\mathcal{L}$-fundamental set}.
\end{df}
The aim of this section is to show that 
$\tilde{F}$ consists of positive imaginary roots of $\mathsf{Q}$ and 
$\mu_{\alpha}^{-1}(\lambda)^{\text{dif}}\neq \emptyset$ if $\alpha\in 
\widetilde{\Sigma}_{\lambda}
\cap \tilde{F}$.
\subsection{A symmetric Kac-Moody root lattice associated with 
	$\mathcal{L}$}
	As we saw in the previous section, the sublattice $\mathcal{L}\subset
	\mathbb{Z}^{\mathsf{Q}_{0}}$ has the action of 
	$\langle s_{a}\mid a\in \mathcal{J}\cup \mathsf{Q}_{0}^{\text{leg}}
	\rangle$
	which is a subgroup of the Weyl group of $\mathsf{Q}$.
	We would like to define an analogy of root system on $\mathcal{L}$
	and regard $\tilde{F}$ as the fundamental set of positive 
	imaginary roots 
	of $\mathcal{L}$.
	For this purpose, 
	we shall define a Kac-Moody root lattice $\mathcal{M}$ 
	and regard $\mathcal{L}$ as a quotient lattice of $\mathcal{M}$.
	And then $\tilde{F}$ will be the image of the fundamental set of 
	positive imaginary roots of $\mathcal{M}$.
\subsubsection{Lift of $\mathcal{L}$ to a Kac-Moody root lattice}
Let us note that $\mathcal{L}$ is  generated by
\[
	\left\{\epsilon_{a}\mid a\in \mathcal{J}\cup
	\mathsf{Q}_{0}^{\text{leg}}\right\},
\]
see Theorem 3.6 in \cite{H}.
It is easy to verify that
\begin{align}
	(\epsilon_{\mathbf{i}},\epsilon_{\mathbf{i}'})&=
	2-\sum_{\substack{0\le i\le p\\j_{i}\neq j'_{i}}}
	(d_{i}(j_{i},j'_{i})+2),\label{equ5}\\
	(\epsilon_{\mathbf{i}},\epsilon_{[i,j,k]})&=
	\begin{cases}
		-1&\text{if }j=j_{i}\text{ and }k=1,\label{equ6}\\
		0&\text{otherwise},
	\end{cases}\\
	(\epsilon_{[i,j,k]},\epsilon_{[i',j',k']})&=
	\begin{cases}
		2&\text{if }[i,j,k]=[i',j',k'],\\
		-1&\text{if }(i,j)=(i',j')\text{ and }
		|k-k'|=1,\\
		0&\text{otherwise},
	\end{cases}\label{equ7}
\end{align}
cf. section 3.2 in \cite{H}.
Here $\mathbf{i}=([i,j_{i}])_{0\le i\le p}$,
$\mathbf{i}'=([i,j'_{i}])_{0\le i\le p}\in \mathcal{J}$.
Thus we consider a new lattice $\mathcal{M}$
generated by the set of indeterminate
\[
	\mathcal{C}:=\left\{c_{a}\mid a\in \mathcal{J}\cup
	\mathsf{Q}_{0}^{\text{leg}}\right\},
\]
namely all $c_{a}\in \mathcal{C}$ have no relations, 
and define a symmetric bilinear form $(\,,\,)$ on $\mathcal{M}$
in accordance with equations $(\ref{equ5}),(\ref{equ6})$ and 
$(\ref{equ7})$.

We can attach $\mathcal{M}$ to a diagram, called {\em Dynkin diagram},
regarding elements in $\mathcal{C}$ as vertices and 
connecting $c,c'\in \mathcal{C}$ by $|(c,c')|$ edges if $c\neq c'$. 
We say $c,c'\in\mathcal{C}$ are {\em connected} if there exists 
a sequence $c_{0}=c,c_{1},\ldots,c_{r}=c'$ in $\mathcal{C}$ 
such that $(c_{i-1},c_{i})\neq 0$ for all $i=1,\ldots,r$.
Then we may define {\em Dynkin diagram} of $\gamma\in \mathcal{M}$
which is a subdiagram obtained by 
connecting the vertices in $\mathrm{supp}(\beta)$ in the same manner.

Also we can define reflections $s_{a}$ on $\mathcal{M}$ by 
\[
	s_{a}(\gamma):=\gamma-(\gamma,c_{a})c_{a}
\]
for $a\in 
\mathcal{J}\cup \mathsf{Q}_{0}^{\text{leg}}$ 
and $\gamma \in \mathcal{M}$.
Let us denote the set of all positive elements in $\mathcal{M}$
by $\mathcal{M}^{+}$.

Then the inclusion 
$\mathcal{L}\hookrightarrow \mathbb{Z}^{\mathsf{Q}_{0}}$
induces
\[
	\Xi\colon \mathcal{M}\longrightarrow
	\mathbb{Z}^{\mathsf{Q}_{0}}
\]	
where for $\gamma=\sum_{c\in\mathcal{C}}\gamma_{c}c\in \mathcal{M}$,
the image $\Xi(\gamma)=(\beta_{a})_{a\in \mathsf{Q}_{0}}$ is given by
\begin{align*}
	\beta_{[i,j]}&:=\sum_{\left\{\mathbf{i}=([i,j_{i}])\in \mathcal{J}\mid 
	j_{i}=j\right\}}\gamma_{c_{\mathbf{i}}},\\
	\beta_{[i,j,k]}&:=\gamma_{c_{[i,j,k]}}.
\end{align*}
\begin{prop}[Theorem 3.6 in \cite{H}]\label{quotientmap}
	We have the following.
	\begin{enumerate}
		\item We have $(\gamma,\gamma')=(\Xi(\gamma),\Xi(\gamma'))$
			for any $\gamma,\gamma'\in \mathcal{M}$.
		\item The image of $\Xi$ is $\mathcal{L}$.
		\item The map $\Xi$ is injective   
			if and only if 
			\[
				\#\{i\in \{0,\ldots,p\}\mid 
				m_{i}>1,\,i=0,\ldots,p\}\le 1.
			\]
		\item For $\gamma\in \mathcal{M}$ and $a\in \mathcal{J}
			\cup \mathsf{Q}_{0}^{\text{leg}}$, we have 
			\[
				\Xi(s_{a}(\gamma))=s_{a}(\Xi(\gamma)).
			\]
	\end{enumerate}
\end{prop}
From this proposition
$\mathcal{M}$ can be seen as a ``lift'' of 
$\mathcal{L}$
to a Kac-Moody root lattice in which $s_{\mathbf{i}}$ for  $\mathbf{i}\in 
\mathcal{J}$ are simple reflections.
\subsubsection{Some special examples}\label{specialexa}
The Dynkin diagram of $\mathbb{Z}^{\mathsf{Q}_{0}}$
can be defined as well as that of $\mathcal{M}$ and  
the Dynkin diagram of $\beta\in \mathbb{Z}^{\mathsf{Q}_{0}}$
is also defined.
Here we note that these diagrams coincide with diagrams obtained by 
forgetting the orientation of the quiver $\mathsf{Q}$ and the subquiver 
associated with $\mathrm{supp}(\beta)$ respectively.
Let us compare Dynkin diagrams of $\beta\in \mathcal{L}$ and the 
inverse image $\Xi^{-1}(\beta)\in \mathcal{M}$ in the following special cases for 
the latter use.

The kernel of $\Xi$ is a big space in general.
Thus if we consider the inverse image of an element $\beta\in \mathcal{L}$,
it is convenient to restrict $\Xi$ to some smaller space as follows.
Fix $\beta\in \mathcal{L}$.
Define $\mathcal{J}_{\beta}:=\{([i,j_{i}])_{i=0,\ldots,p}\in \mathcal{J}\mid 
\beta_{[i,j_{i}]}\neq 0\text{ for all }
i\in I_{\text{irr}}\}$,
$\mathsf{Q}_{0}^{\text{leg}}(\beta):=\mathrm{supp}(\beta)\cap 
\mathsf{Q}_{0}^{\text{leg}}$,
and 
a sublattice $\mathcal{M}_{\beta}:=
\sum_{\{a\in \mathcal{J}_{\beta}\cup 
\mathsf{Q}_{0}^{\text{leg}}(\beta)\}}\mathbb{Z}c_{a}$.
Denote the set of all positive elements in $\mathcal{M}_{\beta}$
by $\mathcal{M}^{+}_{\beta}$.
We write the restriction of $\Xi$ on $\mathcal{M}_{\beta}$
by $\Xi_{\beta}$.

First consider $\beta\in \mathcal{L}$ satisfying that
\[
	\{[i,j]\mid \beta_{[i,j]}\neq 0\}=\{[i,1]\}\quad 
	\text{ for all }i\in I_{\text{irr}}\backslash\{0\},
\]
\[
	\{[0,j]\mid \beta_{[0,j]}\neq 0\}=\{[0,1],\ldots, [0,m_{0}']\}
\]
for some $m_{0}'\le m_{0}$, and $\mathrm{supp}(\beta)$ is connected.
Then Proposition \ref{quotientmap} implies that $\Xi_{\beta}\colon 
\mathcal{M}_{\beta}
\rightarrow \mathcal{L}$ is injective and we can show that this bijection
preserves Dynkin diagrams of elements in $\mathcal{M}$ and $\mathcal{L}$
in the following way.
\begin{prop}\label{firstexa}
	For the above $\beta\in \mathcal{L}$, let us define 
	\[
		\beta^{\text{red}}:=
		\left(\prod_{i\in I_{\text{irr}}\backslash\{0\}}
		s_{[i,1,e_{[i,1]}-1]}\circ \cdots
			\circ s_{[i,1,2]}\circ s_{[i,1,1]}\circ 
			s_{[i,1]}
	\right)(\beta).
	\]
	Then Dynkin diagram of $\beta^{\text{red}}$
	coincides with that of $\Xi_{\beta}^{-1}(\beta)\in 
	\mathcal{M}_{\beta}$.
\end{prop}
\begin{proof}
	From Proposition \ref{quotientmap}, the map $\Xi_{\beta}$ is injective.
	And note the identification 
	\[
		\begin{array}{ccc}\mathcal{J}_{\beta}=
			\{1,\ldots,m'_{0}\}\times 
			\left(\prod_{i=1}^{p}\{1\}\right) 
			&
			\longrightarrow&
		\{1,\ldots,m'_{0}\}\\
		(j_{0},1_{1},\ldots,1_{p}) 
		&\longmapsto& 
			j_{0}
		\end{array}.
		\]
		Thus for $\beta\in \mathcal{L}$ of the form
		\begin{align*}
			\beta=&
			\sum_{j=1}^{m'_{0}}\left(
				\beta_{[0,j]}\epsilon_{[0,j]}
				+\sum_{k=1}^{e_{[0,j]}-1}\beta_{[0,j,k]}
				\epsilon_{[0,j,k]}
			\right)\\
			&+
			\sum_{i\in I_{\text{irr}}\backslash\{0\}}
			\left(\beta_{[i,1]}\epsilon_{[i,1]}
			+\sum_{k=1}^{e_{[i,1]}-1}\beta_{[i,1,k]}
			\epsilon_{[i,1,k]}\right)\\
			&+\sum_{i\in I_{\text{reg}}}\sum_{k=1}^{
		e_{[i,1]}-1}\beta_{[i,1,k]}\epsilon_{[i,1,k]},
		\end{align*}
		the inverse image $\tilde{\gamma}
		:=\Xi_{\beta}^{-1}(\beta)$ is written by
		\begin{align*}
			\tilde{\beta}=&\sum_{j\in \mathcal{J}_{\beta}=
		\{1,\ldots,m'_{0}\}}\left(\beta_{[0,j]}c_{j}+
				\sum_{k=1}^{e_{[0,j]}-1}\beta_{[0,j,k]}
				c_{[0,j,k]}
			\right)\\
			&+
			\sum_{i\in I_{\text{irr}}\backslash\{0\}}
			\sum_{k=1}^{e_{[i,1]}-1}\beta_{[i,1,k]}
			c_{[i,1,k]}\\
			&+\sum_{i\in I_{\text{reg}}}\sum_{k=1}^{
		e_{[i,1]}-1}\beta_{[i,1,k]}c_{[i,1,k]}.
		\end{align*}
		Noting that $\beta_{[i,1]}=\sum_{j=1}^{m_{0}}\beta_{[0,j]}$
		for $i=1,\ldots,p$ since $\beta\in \mathcal{L}$, 
		we can obtain that 
		\begin{align*}
			\beta^{\text{red}}=&
			\sum_{j=1}^{m'_{0}}\left(
				\beta_{[0,j]}\epsilon_{[0,j]}
				+\sum_{k=1}^{e_{[0,j]}-1}\beta_{[0,j,k]}
				\epsilon_{[0,j,k]}
			\right)\\
			&+
			\sum_{i\in I_{\text{irr}}\backslash\{0\}}
			\left(\beta_{[i,1,1]}\epsilon_{[i,1]}
			+\sum_{k=1}^{e_{[i,1]}-2}\beta_{[i,1,k+1]}
			\epsilon_{[i,1,k]}\right)\\
			&+\sum_{i\in I_{\text{reg}}}\sum_{k=1}^{
		e_{[i,1]}-1}\beta_{[i,1,k]}\epsilon_{[i,1,k]},
		\end{align*}
		from direct computation.
		Let us recall that $(\epsilon_{[i,j]},\epsilon_{[i,j,1]})=1$
		and 
\[
(\epsilon_{[i,j,k-1]},\epsilon_{[i,j,k]})=(c_{[i,j,k-1]},c_{[i,j,k]})=1
\]
for $i=0,\ldots,p$,\,$j=1,\ldots,m_{i}$ and
$k=2,\ldots, e_{[i,1]}-1$.
Also 
\begin{align*}
	(\epsilon_{[0,j]},\epsilon_{[0,j']})&=(c_{j},c_{j'})&
	 &\text{ for }j,j'=1,\ldots,m_{0},\\
	(\epsilon_{[0,j]},\epsilon_{[i,1]})&=(c_{j},c_{[i,1,1]})=1&
	&\text{ for }j=1,\ldots,m_{0},\,i\in I_{\text{irr}}\backslash\{0\},\\
	(\epsilon_{[0,j]},\epsilon_{[i,1,1]})&=(c_{j},c_{[i,1,1]})=1&
	&\text{ for }j=1,\ldots,m_{0},\,j\in I_{\text{reg}},\\
	(\epsilon_{[0,j]},\epsilon_{[0,j,1]})&=(c_{j},c_{[0,j,1]})=1&
	&\text{ for }j=1,\ldots,m_{0}.
\end{align*}
All the other pairs in $\mathrm{supp}(\beta)$ (resp. 
$\mathrm{supp}(\tilde{\beta})$) are zero.
Thus we are done.
\end{proof}

Next let us consider the following special $\beta\in \mathcal{L}$,
\[
	\{[0,j]\mid \beta_{[0,j]}\neq 0\}=\{[0,1],[0,2]\},\quad 
	\{[1,j]\mid \beta_{[1,j]}\neq 0\}=\{[1,1],[1,2]\},
\]
\[
	\{[i,j]\mid \beta_{[i,j]}\neq 0\}=\{[i,1]\}\text{ for all }i\in 
	I_{\text{irr}}\backslash\{0,1\},
\]
\[
	d_{0}(1,2)=d_{1}(1,2)=1,
\]
and 
\[
	\mathsf{Q}_{0}^{\text{leg}}\cap \mathrm{supp}(\beta)=\emptyset.
\]
Then
the Dynkin diagram of $\mathcal{M}_{\beta}$ is 
	\begin{xy}
		\ar@{=}(0,0)*+!R+!U(0.5){c_{\mathbf{i}_{1}}}*\cir<4pt>{}
		;(7,7) *+!L+!U(0.5){c_{\mathbf{i}_{4}}}*\cir<4pt>{},
		\ar@{=}(0,7)*+!R+!U(0.5){c_{\mathbf{i}_{2}}}*\cir<4pt>{}
		;(7,0) *+!L+!U(0.5){c_{\mathbf{i}_3}}*\cir<4pt>{},
        \ar@{} (8,-5) *{}
    \end{xy}
    where $\mathbf{i}_{1},\ldots,\mathbf{i}_{4}\in 
    \mathcal{J}=\{1_{0},2_{0}\}\times \{1_{1},2_{1}\}\times \{1_{2}\}
    \times \cdots \times \{1_{p}\}$ are defied by
    \begin{align*}
	    &\mathbf{i}_{1}:=(1_0,1_{1},1_{2},\ldots,1_{p}),&
	    &\mathbf{i}_{2}:=(2_{0},1_{1},1_{2},\ldots,1_{p}),\\
	    &\mathbf{i}_{3}:=(1_{0},2_{1},1_{2},\ldots,1_{p}),&
	    &\mathbf{i}_{4}:=(2_{0},2_{1},1_{2},\ldots,1_{p}).
    \end{align*}
    In this case corresponding subquiver $\mathsf{Q}_{\beta}$ 
    generated by $\mathrm{supp}(\beta)$
    is 
    
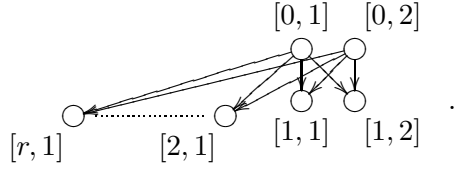
\[	\begin{xy}
		(-10,-2)*+!R+!U{[2,1]}*\cir<4pt>{}="E",
		(-30,-2)*+!R+!U{[r,1]}*\cir<4pt>{}="F",
		(-13,-2)="G",
		(-27,-2)="H",
		\ar@{..}"G";"H",
		\ar@{<-}(0,0)*++!U{[1,1]}*\cir<4pt>{}="A"
		;(7,7) *+!L+!D{[0,2]}*\cir<4pt>{}="B",
		\ar@{<-}(7,0)*+!L+!U{[1,2]}*\cir<4pt>{}="C"
		;(0,7) *++!D{[0,1]}*\cir<4pt>{}="D",
		\ar@{->}"D";"A",
		\ar@{->}"B";"C",
	        \ar@{} (8,-5) *{}
		\ar@{->}"D";"E",
		\ar@{->}"B";"E",
		\ar@{->}"D";"F",
		\ar@{->}"B";"F",
    \end{xy}.
\]
Then the image of  $\tilde{\beta}=\beta_{1}c_{\mathbf{i}_{1}}+\cdots +
\beta_{4}c_{\mathbf{i}_{4}}\in \mathcal{M}_{\beta}$
by $\Xi_{\beta}\colon\mathcal{M}_{\beta}\rightarrow \mathcal{L}$ is 
\begin{align*}
	\beta:=\Xi(\tilde{\beta})=
	&(\beta_{1}+\beta_{3})\epsilon_{[0,1]}+
	(\beta_{2}+\beta_{4})\epsilon_{[0,2]}
	+(\beta_{1}+\beta_{2})\epsilon_{[1,1]}+
	(\beta_{3}+\beta_{4})\epsilon_{[1,2]}\\
	&+(\beta_{1}+\cdots +\beta_4)\epsilon_{[2,1]}+
	\cdots
	+(\beta_{1}+\cdots +\beta_{4})\epsilon_{[r,1]}.
\end{align*}
Let us define 
\[
	\beta^{\text{red}}:=\left(\prod_{i=2}^{r}s_{[i,1]}\right)(\beta).
\]
Then the Dynkin diagram of $\beta^{\text{red}}$ is of type $A_{3}^{(1)}$,
   
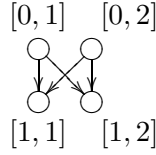
\[	\begin{xy}
		\ar@{<-}(0,0)*++!U{[1,1]}*\cir<4pt>{}="A"
		;(7,7) *+!L+!D{[0,2]}*\cir<4pt>{}="B",
		\ar@{<-}(7,0)*+!L+!U{[1,2]}*\cir<4pt>{}="C"
		;(0,7) *++!D{[0,1]}*\cir<4pt>{}="D",
		\ar@{->}"D";"A",
		\ar@{->}"B";"C",
	            \end{xy}
\]
if $\beta_{s}+
\beta_{t}\neq 0$ for all $(s,t)=(1,2),(1,3),(2,4),(3,4)$.
\begin{rem}\label{reflfunc}\normalfont
	Recalling the orientation of $\mathsf{Q}$, we can check that 
	reflections 
	defining $\beta^{\text{reg}}$, 
	$\prod_{i\in I_{\text{irr}}\backslash\{0\}}
	s_{[i,1,e_{[i,1]}-1]}\circ \cdots
	\circ s_{[i,1,2]}\circ s_{[i,1,1]}\circ 
	s_{[i,1]}$ in Proposition \ref{firstexa}
	and $\prod_{i=2}^{r}s_{[i,1]}$ 
	in the latter example,
	are  products of {\em admissible reflections},
	i.e., reflection 
	functors of Bernstein-Gel'fand-Ponomarev
	on representations of $\mathsf{Q}$
	associated to these reflections are well-defined.
	The detail of the reflection functor can be found 
	in \cite{Kac} for example.
\end{rem}
\subsubsection{Inverse image of a positive imaginary root is a positive 
imaginary root}	
The following lemma shows that if $\beta\in \mathcal{L}^{+}$,
then $\Xi^{-1}_{\beta}(\beta)\cap \mathcal{M}_{\beta}^{+}\neq \emptyset.$
Namely there exist at least one positive element in the inverse image of a
positive element in $\mathcal{L}$.
\begin{lem}[Lemma 3 in \cite{HO}]\label{inverseroot}
	Take $\beta\in \mathcal{L}^{+}\backslash\{0\}$ and  set 
	\begin{align*}
	\overline{m}_{i}&:=\mathrm{max}\{j\in \{1,\ldots,m_{i}\}
		\mid \beta_{[i,j]}\neq 0\},\\
	\underline{m}_{i}&:=\mathrm{min}\{j\in \{1,\ldots,m_{i}\}
		\mid \beta_{[i,j]}\neq 0\},
	\end{align*}
		for $i\in I_{\text{irr}}$.
		Further set $\overline{\mathbf{i}}:=([i,\overline{m}_{i}])
	_{0\le i\le p}$ and 
	$\underline{\mathbf{i}}:=([i,\underline{m}_{i}])
	_{0\le i\le p}$
	where we put $\overline{m}_{i}=\underline{m}_{i}:=1$ for 
	$i\in I_{\text{reg}}$.
	
	Then there exists 
	$\tilde{\beta}\in \mathcal{M}^{+}_{\beta}$ such that 
	$\Xi(\tilde{\beta})=\beta$ and 
	$\tilde{\beta}_{c_{\overline{\mathbf{i}}}}\cdot
	\tilde{\beta}_{c_{\underline{\mathbf{i}}}}\neq 0$.

\end{lem}

From this lemma and the above proposition, we can see the following.
\begin{cor}\label{negative}
	For $\beta\in \tilde{F}$ we have $q(\beta)\le 0$.
\end{cor}
\begin{proof}
	Let us take $\tilde{\beta}\in \mathcal{M}^{+}$ as 
	Lemma \ref{inverseroot}.
	Then Proposition \ref{quotientmap} says that 
	$(\tilde{\beta},c_{a})\le 0$ for all $a\in \mathcal{J}
	\cup \mathsf{Q}_{0}^{\text{leg}}$.
	Thus $q(\beta)=q(\tilde{\beta})\le 0$, see the proof of 
	Proposition 5.2 in \cite{Kacbook} for example.
\end{proof}

For $\beta\in \mathcal{L}$, let us define
\[
	\mathrm{Irr}(\beta)=\left\{i\in I_{\text{irr}}\,\bigg|\,
	\begin{array}{l}
	\text{there exist at least two distinct  $j,j'$ such that }\\
	\beta_{[i,j]}\neq 0\text{ and }\beta_{[i,j']}\neq 0
	\end{array}
	\right\}.
\]
Then we can identify 
\[
	\mathcal{J}_{\beta}\cong \prod_{i\in \mathrm{Irr}(\beta)}
	\{[i,j]\mid \beta_{[i,j]}\neq 0\}.
\]
Proposition \ref{quotientmap} shows that 
if $\#\mathrm{Irr}(\beta)\le 1$ for $\beta\in \mathcal{L}^{+}$,
then $\Xi_{\beta}$ is injective, which 
implies that $\Xi_{\beta}^{-1}(\beta)=\{\tilde{\beta}\}$.
Here we take $\tilde{\beta}$ as in Lemma \ref{inverseroot}.

Next we shall see that for $\beta\in \tilde{F}$, the support of $\tilde{\beta}
\in \mathcal{M}_{\beta}^{+}$ in Lemma \ref{inverseroot} 
is connected with one exception.

\begin{lem}\label{glue}
	For $\mathbf{i}=([i,j_{i}])_{0\le i\le p}$, 
	$\mathbf{i}'=([i,j'_{i}])_{0\le i\le p}\in \mathcal{J}$ and 
	$\mathbf{i}\neq \mathbf{i}'$, 
	we have $(c_{\mathbf{i}},c_{\mathbf{i}'})=0$ 
	if and only if the following are satisfied;
	\begin{enumerate}
		\item $\#\{i\mid j_{i}\neq j'_{i}\}=1$,
		\item setting $\{i_{0}\}:=
	\{i\mid j_{i}\neq j'_{i}\}$, we have 
		 $d_{i_{0}}(j_{i_{0}},j'_{i_{0}})=0$.
 \end{enumerate}
\end{lem}
\begin{proof}
	Since $d_{i}(j,j')\ge 0$ for $j\neq j'$, 
	$(c_{\mathbf{i}},c_{\mathbf{i}'})=2-\sum_{\substack{0\le i\le p\\j_{i}\neq j'_{i}}}
	(d_{i}(j_{i},j'_{i})+2)<0$ 
	if $\#\{i\mid j_{i}\neq j'_{i}\}\ge 2$.
	Also if $\{i_{0}\}=\{i\mid j_{i}\neq j'_{i}\}$, then 
	$(c_{\mathbf{i}},c_{\mathbf{i}'})=2-(d_{i}(j_{i},j'_{i})+2)$.
	Thus we obtain the result.
\end{proof}
\begin{lem}\label{connected}
	For $\beta\in \tilde{F}$, take $\tilde{\beta}
	\in {\mathcal{M}}^{+}$ as in Lemma \ref{inverseroot}.
	\begin{enumerate}
		\item If $c_{[i,j,k]}\in \mathrm{supp}(\tilde{\beta})$,
			then 
			$c_{[i,j,k']}\in \mathrm{supp}(\tilde{\beta})$
	for  $k'\le k$ and there exists 
	$\mathbf{i}=([i',j_{i'}])_{i'=0,\ldots,p}$ 
	satisfying  $j_{i}=j$ such that $c_{\mathbf{i}}
	\in\mathrm{supp}(\tilde{\beta}) $. 
\item If one of the following is satisfied,
	\begin{enumerate}
		\item $\#\mathrm{Irr}(\beta)\ge 3$,
		\item there exist $[i,j],[i,j']\in \mathrm{supp}(\beta)$
			such that $d_{i}(j,j')\ge 1$,
	\end{enumerate}
	then any distinct 
	$c_{\mathbf{i}},c_{\mathbf{i}'}\in \mathrm{supp}(\tilde{\beta})$
	for $\mathbf{i},\mathbf{i}'\in \mathcal{J}_{\beta}$ are connected.
	\end{enumerate}
\end{lem}
\begin{proof}
	Let us note that 
	$\sum_{\left\{\mathbf{i}\in \mathcal{J}
	\mid j_{i}=j\right\}}\tilde{\beta}_{c_{\mathbf{i}}}\ge 
	\tilde{\beta}_{c_{[i,j,1]}}\ge 
	\tilde{\beta}_{c_{[i,j,2]}}\ge 
	\tilde{\beta}_{c_{[i,j,e_{[i,j]}-1]}}$,
	since $\beta\in \widetilde{F}$.
	This shows $(1)$.

	Let us suppose that $\#\mathrm{Irr}(\beta)\ge 3$.	
	If $(c_{\mathbf{i}},c_{\mathbf{i}'})\neq 0$ for a pair
	of distinct 
	$c_{\mathbf{i}},
	c_{\mathbf{i}'}\in \mathrm{supp}(\tilde{\beta})$, 
	then they are obviously connected.
	Thus we consider the case 
	$(c_{\mathbf{i}},c_{\mathbf{i}'})=0$.
	Then 
	$\#\{i\mid j_{i}\neq j'_{i}\}=1$ by Lemma \ref{glue}.
	The hypothesis $\#\mathrm{Irr}(\beta)\ge 3$ assures that 
	the existence of $c_{\mathbf{i}''}\in \mathrm{supp}(\tilde{\beta})$
	such that $\#\{i\mid j_{i}\neq j''_{i}\}\neq 1$ and 
	$\#\{i\mid j'_{i}\neq j''_{i}\}\neq 1$ are satisfied.
	Namely $(c_{\mathbf{i}},c_{\mathbf{i}''})\le -1$ and 
	$(c_{\mathbf{i}'},c_{\mathbf{i}''})\le -1$.
	Thus $c_{\mathbf{i}}$ and $c_{\mathbf{i}'}$ are connected.
	
	Next suppose that  
there exist $[\underline{i},\underline{j}],[\underline{i},
\underline{j}']\in \mathrm{supp}(\beta)$
such that $d_{\underline{i}}(\underline{j},\underline{j}')\ge 1$.
	
	We note that for any $i\in\{0,\ldots,p\}$ and $j,j',j''
	\in \{1,\ldots,m_{i}\}$, the triangle inequality 
	\[
		d_{i}(j,j')\le \mathrm{max}
		\left\{d_{i}(j,j''),d_{i}(j'',j')
		\right\}
	\]
	holds (see equations (26) in \cite{HO}).
	Thus $d_{\underline{i}}(j,\underline{j})\ge 1$ 
	or $d_{\underline{i}}(j,\underline{j}')\ge 1$
	holds for $j\in \{1,\ldots,m_{\underline{i}}\}
	\backslash\{\underline{j},\underline{j}'\}$
	since the condition 
	$d_{\underline{i}}(j,\underline{j})=
	d_{\underline{i}}(j,\underline{j}')=0$ breaks the triangle 
	inequality and $d_{\underline{i}}(\underline{j},\underline{j}')\ge 1$. 
	
	Let us fix $\mathbf{i}=([i,j_{i}])_{0\le i\le p}\in 
	\mathcal{J}$ such that $c_{\mathbf{i}}\in \mathrm{supp}(\tilde{\beta})$
	and define $\mathcal{J}_{\mathbf{i},\underline{i}}:=\{
	([i,j'_{i}])_{0\le i\le p}\in \mathcal{J}\mid 
	j'_{i}=j_{i} \text{ for all } i\neq \underline{i}
\}$. Then for all $\mathbf{i}',\mathbf{i}''\in 
\mathcal{J}_{\mathbf{i},\underline{i}}$, $c_{\mathbf{i}'}$ and 
$c_{\mathbf{i}''}$ are connected. 
	Indeed define $\mathbf{i}_{\underline{j}}=
	([i,j^{(1)}_{i}])_{0\le i\le p} \in 
	\mathcal{J}_{\mathbf{i},\underline{i}}$
	and $\mathbf{i}_{\underline{j}'}=([i,j^{(2)}_{i}])_{0\le i\le p}\in 
	\mathcal{J}_{\mathbf{i},\underline{i}}$
	by 
	\[
		j^{(1)}_{i}=\begin{cases}
			\underline{j}&\text{ if }i=\underline{i},\\
			j_{i}&\text{ otherwise}
		\end{cases}
	\]
	and 
	\[
		j^{(2)}_{i}=\begin{cases}
			\underline{j}'&\text{ if }i=\underline{i},\\
			j_{i}&\text{ otherwise}
		\end{cases}
	\]
	respectively. Then
		$c_{\mathbf{i}'}$ and $c_{\mathbf{i}''}$
	are connected with $c_{\mathbf{i}_{\underline{j}}}$ or 
	$c_{\mathbf{i}_{\underline{j}'}}$ since  
	$d_{\underline{i}}(j,\underline{j})\ge 1$ 
	or $d_{\underline{i}}(j,\underline{j}')\ge 1$ holds for 
	$j\in \{1,\ldots,m_{\underline{i}}\}
	\backslash\{\underline{j},\underline{j}'\}$ as we saw above.

	Moreover $c_{\mathbf{i}_{\underline{j}}}$ and $c_{
	\mathbf{i}_{\underline{j}'}}$ are connected. Thus 
	$c_{\mathbf{i}'}$ and $c_{\mathbf{i}''}$ are connected.
	Finally consider $c_{\mathbf{i}'''}\in \mathrm{supp}(\tilde{\beta})$
	with 
	$\mathbf{i}'''=([i,j'''_{i}])_{0\le i\le p}
	\notin \mathcal{J}_{\mathbf{i},\underline{i}}$,
	namely there exists $0\le i_{0}(\neq \underline{i})\le p$
	such that $j'''_{i_{0}}\neq j_{i}$.
	Then we can choose $\mathbf{i}''''=(i,j''''_{i})_{0\le i\le p}\in 
	\mathcal{J}_{\mathbf{i},\underline{i}}$ so that
	$j'''_{\underline{i}}\neq j''''_{\underline{i}}$ 
	since $\#\{1,\ldots,m_{\underline{i}}\}\ge 2$.
	Thus 
	\[
		(c_{\mathbf{i}'''},c_{\mathbf{i}''''})\le 
	2-(d_{i_{0}}(j'''_{i_{0}},j_{i_{0}})+2)+
		(d_{\underline{i}}(j'''_{\underline{i}},j_{\underline{i}})+2)
		<0.
	\]
	
	Summing up these results, we can conclude that 
	any distinct $c_{\mathbf{i}^{(1)}},c_{\mathbf{i}^{(2)}}
	\in \mathrm{supp}(\tilde{\beta})$ is connected with 
	some $c_{\mathbf{i}^{(3)}},c_{\mathbf{i}^{(4)}}
	\in \mathrm{supp}(\tilde{\beta})$, 
	$\mathbf{i}^{(3)},\mathbf{i}^{(4)}
	\in \mathcal{J}_{\mathbf{i},\underline{i}}$
	respectively. Since $c_{\mathbf{i}^{(3)}}$ and 
	$c_{\mathbf{i}^{(4)}}$ are connected, 
	$c_{\mathbf{i}^{(1)}}$ and  $c_{\mathbf{i}^{(2)}}$
	are connected as well.
\end{proof}
 
\begin{prop}\label{connectedsupport}
	For $\beta\in \tilde{F}$ take $\tilde{\beta}
	\in {\mathcal{M}}^{+}$ as in Lemma \ref{inverseroot}.
	Then the support of $\tilde{\beta}$ is connected or the 
	following holds.
	There exist $i_{1},i_{2}\in I_{\text{irr}}$ and $j_{s,t}
	\in \{1,\ldots,m_{i_{s}}\}$ for $s,t=1,2$ such that 
	$\mathrm{Irr}(\beta)=\{i_{1},i_{2}\}$ 
	and $\mathrm{supp\,}(\tilde{\beta})=\{c_{\mathbf{i}_{u,v}}\mid
		u,v=1,2
	\}$ where
	\[
		\mathbf{i}_{u,v}:=
		([i_{1},j_{1,u}],[i_{2},j_{2,v}])\in \mathcal{J}_{\beta}\cong
		\prod_{s=1}^{2}\{[i_{s},j_{s,t}]\mid t=1,2 \}
	\]
	and
	\begin{align*}
		(c_{\mathbf{i}_{u,v}},c_{\mathbf{i}_{u'v'}})=
		\begin{cases}
			0&\text{if }u=u'\text{ or }v=v',\\
			-2&\text{otherwise},
		\end{cases}
	\end{align*}
	namely the Dynkin diagram of $\mathrm{supp}(\tilde{\beta})$ is 
	
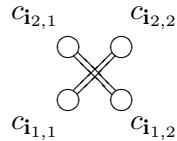
\begin{xy}
		\ar@{=}(0,0)*+!R+!U{c_{\mathbf{i}_{1,1}}}*\cir<4pt>{}
		;(7,7) *+!L+!D{c_{\mathbf{i}_{2,2}}}*\cir<4pt>{},
		\ar@{=}(0,7)*+!R+!D{c_{\mathbf{i}_{2,1}}}*\cir<4pt>{}
		;(7,0) *+!L+!U{c_{\mathbf{i}_{1,2}}}*\cir<4pt>{},
        \ar@{} (8,-5) *{}
    \end{xy}.
	In the latter case we have $q(\beta)=0$.
\end{prop}
\begin{proof}
	From (1) and (2)-(a) in Lemma \ref{connected}, 
	it suffices to consider the case $\#\mathrm{Irr}(\beta)\le 2$.
	First suppose that $\#\mathrm{Irr}(\beta)\le 1$.
	Then there exists $i_{0}\in \{0,\ldots,p\}$ and we may identify
	$\mathcal{J}_{\beta}\cong\{[i_{0},1],\ldots,[i_{0},m'_{i_{0}}]\}$
	with some $m_{i_{0}}'\le m_{i_{0}}$.
	Suppose that $\supp(\tilde{\beta})$ is not connected.
	Then $c_{[i,j,k]}\notin \mathrm{supp}(\tilde{\beta})$ for any 
	$i\neq i_{0}$, $j=1,\ldots,m_{i}$ and $k=1,\ldots,e_{[i,j]}-1$.
	Indeed, all $c_{[i,j,k]}\in \mathrm{supp}(\tilde{\beta})$ 
	($i\neq i_{0}$) are connected with 
	all $c_{\mathbf{i}}$ $(\mathbf{i}\in \mathcal{J}_{\beta})$
	from (1) in Lemma \ref{connected}. And any $c_{[i_{0},j,k]}\in
	\mathrm{supp}(\tilde{\beta})$ is connected with one of $c_{\mathbf{i}}
	$ for $\mathbf{i}\in \mathcal{J}_{\beta}$. Thus $\mathrm{supp}(
	\tilde{\beta})$ is connected.
	Moreover $(c_{\mathbf{i}},c_{\mathbf{i}'})=0$ for all $\mathbf{i}\neq 
	\mathbf{i}'\in \mathcal{J}_{\beta}$ from (2)-(b) in Lemma 
	\ref{connected}.
	Thus connected components of 
	the Dynkin diagram of $\mathrm{supp}(\tilde{\beta})$ are 
	of type $A_{n}$,
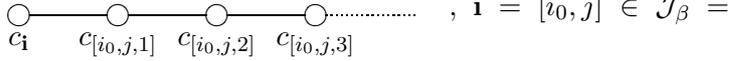
	\begin{xy}
		*++!U{c_{\mathbf{i}}}*\cir<4pt>{}="A";
		\ar@{-} "A"; (13,0) *++!U{c_{[i_0,j,1]}}*{}*\cir<4pt>{}="B",
		\ar@{-} "B";(26,0) *++!U{c_{[i_0,j,2]}}*\cir<4pt>{}="C",
		\ar@{-} "C";(39,0) *++!U{c_{[i_0,j,3]}}*\cir<4pt>{}="D",
		\ar@{.}"D";(52,0) *+!L+!D{}*{}
    	\end{xy},
	$\mathbf{i}=[i_{0},j] \in \mathcal{J}_{\beta}=\{[i_{0},1],
	\ldots,[i_{0},m'_{i_{0}}]\}$.
	However the root system of type $A_{n}$ has no imaginary root
	which contradict to $\beta\in \widetilde{F}$.

	The remaining case is $\#\mathrm{Irr}(\beta)=2$. 
	We may suppose that $\mathrm{Irr}(\tilde{\beta})=\{0,1\}$
	without loss of generality.
	
	Recall that 
	for 
	$\overline{\mathbf{i}}=([0,\overline{m_{0}}],[1,\overline{m_{1}}])$
	and 
	$\underline{\mathbf{i}}=([0,\underline{m_{0}}],[1,
	\underline{m_{1}}])$,
	$c_{\overline{\mathbf{i}}},c_{\underline{\mathbf{i}}}\in 
	\mathrm{supp}(\tilde{\beta})$ and $(c_{\overline{\mathbf{i}}},
	c_{\underline{\mathbf{i}}})
	\le 2-(d_{0}(\overline{m_{0}},\underline{m_{0}})+2)-
	(d_{1}(\overline{m_{1}},\underline{m_{1}})+2)<0$.
	Let us suppose that $\mathrm{supp}(\tilde{\beta})$ is not connected.
	Then $c_{[i,j,k]}\notin\mathrm{supp}(\tilde{\beta})$ for $i\neq 0,1$
	as above.
	And there exists 
	$c_{\mathbf{i}}\in \mathrm{supp\,}(\tilde{\beta})$ such that 
	neither $c_{\underline{\mathbf{i}}}$ nor $c_{\overline{\mathbf{i}}}$ 
	is connected with $c_{\mathbf{i}}$.
	Here $\mathbf{i}$ must be 
	$\mathbf{i}_{1}=([0,\overline{m}_{0}],[1,\underline{m}_{1}])$ or 
	$\mathbf{i}_{2}=([0,\underline{m}_{0}],[1,\overline{m}_{1}])$
	because of the condition 1 in Lemma \ref{glue}.
	And $d_{i}(\overline{m_{i}},\underline{m_{i}})=0$ for $i=0,1$.
	We may suppose 
	$\mathbf{i}=\mathbf{i}_{1}$.
	Then $\mathrm{supp\,}(\tilde{\beta})\cap
	\{c_{\mathbf{j}}\mid \mathbf{j}\in \mathcal{J}\}
	\subset \{c_{\underline{\mathbf{i}}},
	c_{\overline{\mathbf{i}}},c_{\mathbf{i}_{1}},c_{\mathbf{i}_{2}}\}$.
	Indeed, if there exists  $\mathbf{i}'\in \mathcal{J}$ such that 
	$c_{\mathbf{i}'}\in \mathrm{supp\,}(\tilde{\beta})\backslash
	\{c_{\underline{\mathbf{i}}},
	c_{\overline{\mathbf{i}}},c_{\mathbf{i}_{1}},c_{\mathbf{i}_{2}}\}$,
	then the condition 1 in Lemma \ref{glue} shows that 
	$c_{\mathbf{i}'}$ must be connected with all 
	$c_{\underline{\mathbf{i}}},
	c_{\overline{\mathbf{i}}},c_{\mathbf{i}_{1}},c_{\mathbf{i}_{2}}$.
	For example,
	\begin{xy}
		(0,14)*+!R+!D{c_{\mathbf{i}'}}*\cir<4pt>{}="C",
		\ar@{=}(0,0)*+!R+!U{c_{\underline{\mathbf{i}}}}*\cir<4pt>{}
		;(7,7) *+!L+!D{c_{\overline{\mathbf{i}}}}*\cir<4pt>{}="A",
		\ar@{=}(0,7)*+!R+!D{c_{\mathbf{i}_{1}}}*\cir<4pt>{}
		;(7,0) *+!L+!U{c_{\mathbf{i}_{2}}}*\cir<4pt>{}="B",
		\ar@{=}"A";"C",
		\ar@{=}"B";"C",
    \end{xy}.
	Then ${c}_{\mathbf{i}}$ is connected with 
	$c_{\underline{\mathbf{i}}}$
	and $c_{\overline{\mathbf{i}}}$, which is a contradiction.

	Let us note that $c_{\mathbf{i}_{1}}$ can be connected 
	only with $c_{\mathbf{i}_{2}}$ in $\mathrm{supp\,}(\tilde{\beta})$
	because if there exists $c_{[i,j,1]}$ which is connected with
	$c_{\mathbf{i}_{1}}$, then $c_{[i,j,1]}$ must be connected with
	$c_{\underline{\mathbf{i}}}$ or $c_{\overline{\mathbf{i}}}$ which 
	implies that $c_{\mathbf{i}_{1}}$ is connected with them.
	The same argument shows that  $c_{\mathbf{i}_{2}}$ can be connected 
	only with $c_{\mathbf{i}_{1}}$ in $\mathrm{supp\,}(\tilde{\beta})$.
	Thus if $c_{\mathbf{i}_{2}}\notin \mathrm{supp\,}(\tilde{\beta})$,
	then $c_{\mathbf{i}_{1}}$ is isolated in 
	$\mathrm{supp\,}(\tilde{\beta})$, which contradicts to that 
	$\beta\in \widetilde{F}$.
	
	Hence the remaining possibility is $\mathrm{supp\,}(\tilde{\beta})=
	\{c_{\underline{\mathbf{i}}},
	c_{\overline{\mathbf{i}}},c_{\mathbf{i}_{1}},c_{\mathbf{i}_{2}}\}
	$ and 
	\begin{align*}
		(c_{\underline{\mathbf{i}}},c_{\overline{\mathbf{i}}})&=-2,&
		(c_{\underline{\mathbf{i}}},c_{\mathbf{i}_{i}})&=0,\ \ i=1,2,\\
		(c_{\mathbf{i}_{1}},c_{\mathbf{i}_{2}})&=-2,&
		(c_{\overline{\mathbf{i}}},c_{\mathbf{i}_{i}})&=0,\ \ i=1,2.
	\end{align*}
\end{proof}
\subsection{Wild case}
Let us separate our argument into the two cases, namely  the wild case, $q(\beta)<0$,
and the tame case, $q(\beta)=0$, for $\beta\in \tilde{F}$.
We consider the wild case first.
\begin{prop}\label{nontame}
	Let $\beta=\gamma_{1}+\cdots +\gamma_{r}\in \tilde{F}$ with 
	$q(\beta)<0$, 
	$r\ge 2$ and $\gamma_{1},\ldots,\gamma_{r}\in \mathcal{L}^{+}
	\backslash\{0\}$ then
	$q(\beta)< q(\gamma_{1})+\cdots +q(\gamma_{r})$.
\end{prop}
\begin{proof}
	For the above $\gamma_{1},\ldots,\gamma_{r}$, take 
	$\tilde{\gamma_{1}},\ldots,\tilde{\gamma_{r}}\in 
	\mathcal{M}^{+}\backslash\{0\}$ as in Lemma \ref{inverseroot}
	and define $\tilde{\beta}=
	\tilde{\gamma_{1}}+\cdots+\tilde{\gamma_{r}}$.
	Then $\tilde{\beta}$ satisfies conditions in Lemma \ref{inverseroot}.
	Thus the support of $\tilde{\beta}$ is connected from 
	Proposition \ref{connectedsupport}.
	Recall that $(\tilde{\beta},c_{a})\le 0$ for 
	all $a\in \mathcal{J}\cup \mathsf{Q}_{0}^{\text{leg}}$ and the assumption $q(\beta)<0$.
	Then the standard argument (see Lemma 2 in \cite{KR} for example) shows
	that $q(\beta)=(\tilde{\beta},\tilde{\gamma})< \sum_{i=1}^{r}
	(\tilde{\gamma}_{i},\tilde{\gamma}_{i})=
	\sum_{i=1}^{r}q(\gamma_{i})$. 
\end{proof}

Let us fix $\beta \in \widetilde{F}$ with $q(\beta)<0$.
Define a non-empty open subset of $\mathrm{Rep}(\mathsf{Q},\beta)$ by
\begin{multline*}
	\mathrm{Rep}(\mathsf{Q},\beta)^{\text{det}}:=\\
	\left\{
		x\in 
            \mathrm{Rep}(\mathsf{Q},\beta)\,\middle|\,
            \mathrm{det}\left(
            x_{\rho^{[0,j]}_{[i,j']}}
            \right)_{\substack{1\le j\le m_{0}\\
            1\le j'\le m_{i}}}\neq 0
            ,\,i\in I_{\text{irr}}\backslash\{0\}
    \right\}.
\end{multline*}

\begin{lem}\label{homogeneousdecomposition}
	If $x\in \mathrm{Rep}(\mathsf{Q},\beta)^{\text{det}}$ is decomposed as 
	$x=x_{1}\oplus \cdots\oplus x_{r}$
	in $\mathrm{Rep}(\mathsf{Q},\beta)$, then 
	${\mathrm{\mathbf{dim}\,}}x_{i}\in \mathcal{L}^{+}$ for all 
	$i=1,\ldots,r$.
\end{lem}
\begin{proof}
	Since $x\in  \mathrm{Rep}(\mathsf{Q},\beta)^{\text{det}}$,
	subrepresentations $x_{t}$ with 
	$\gamma_{t}=
	{\mathrm{\mathbf{dim}}\,}x_{t}$ for $t=1,\ldots,r$ satisfy
	$\sum_{j=1}^{m_{0}}(\gamma_{t})_{[0,j]}\le 
	\sum_{j=1}^{m_{i}}(\gamma_{t})_{[i,j]}$
	for all $i\in I_{\text{irr}}\backslash\{0\}$.
	If there exists ``$<$'' among these inequalities, then
	$\beta=\gamma_{1}+
	\cdots +\gamma_{r}\notin \mathcal{L}^{+}$ which contradicts
	to the assumption $\beta \in \widetilde{F}\subset\mathcal{L}^{+}$.
	Thus $\gamma_{t}\in \mathcal{L}^{+}$ for all $t=1,\ldots,r$.
\end{proof}

Let us recall the notion of {\em generic decomposition}.
A decomposition $\beta=\gamma_{1}+\cdots+\gamma_{r}$, 
$\gamma_{t}\in (\mathbb{Z}_{\ge 0})^{\mathsf{Q}_{0}}\backslash\{0\}$, is called 
the generic decomposition if 
\begin{multline*}
	\mathrm{Ind}(\mathsf{Q};\gamma_{1},\ldots,\gamma_{r})=\\
	\left\{
		x_{1}\oplus \cdots\oplus x_{r}\in \mathrm{Rep}(\mathsf{Q},\beta)\,
		\bigg|\,
		\begin{array}{c}
			\mathrm{\mathbf{dim}\,}x_{t}=\gamma_{t} \text{ and }
			x_{t}\text{ are }\\\text{ indecomposable for }t=1,\ldots,r
		\end{array}
	\right\}
\end{multline*}
contains a non-empty open dense subset of $\mathrm{Rep}(\mathsf{Q},\beta)$.
It is known that the generic decomposition  uniquely exists
for any $\beta'\in (\mathbb{Z}_{\ge 0})^{\mathsf{Q}_{0}}\backslash\{0\}$, 
see Proposition 2.7 in 
\cite{KR} for example.

\begin{prop}\label{canonicaldecomposition}
	Let us take $\beta\in \widetilde{F}$ with $q(\beta)<0$.
	If $\beta=\gamma_{1}+\cdots+\gamma_{r}$ is the generic decomposition
	of $\beta\in \widetilde{F}$, then $r=1$.
\end{prop}
\begin{proof}
	If $\beta=\gamma_{1}+\cdots+\gamma_{r}$ is the generic decomposition,
	then 
	\[
		\mathrm{Ind}(\mathsf{Q};\gamma_{1},\ldots,\gamma_{r})\cap
		\mathrm{Rep}(\mathsf{Q},\beta)^{\text{det}}\neq \emptyset.
	\]
	Thus $\gamma_{i}\in \mathcal{L}^{+}$ for $i=1,\ldots,r$ by Lemma 
	\ref{homogeneousdecomposition}.
	Then Proposition \ref{nontame} shows that 
	$q(\beta)<q(\gamma_{1})+\cdots +q(\gamma_{r})$ if $r\ge 2$.
	This contradicts to that $\beta=\gamma_{1}+\cdots+\gamma_{r}$
	is the generic decomposition by the standard argument,
	see Theorem 3.3 in \cite{KR} for example.
	Thus $r=1$.
\end{proof}

\begin{cor}\label{wildroot}
	If $\beta\in \widetilde{F}$ and $q(\beta)<0$, then 
	$\beta$ is a positive root of $\mathsf{Q}$.
\end{cor}
\begin{proof}
	Proposition \ref{canonicaldecomposition} shows that 
	$\mathrm{Rep}(\mathsf{Q},\beta)$ contains a indecomposable 
	representation. Then Kac's theorem (Theorem 1.10 in 
	\cite{Kac}) tells us that $\beta$ is 
	a positive root of $\mathsf{Q}$.
\end{proof}

\begin{cor}\label{existencewild}
	Let us take  $\beta\in \widetilde{\Sigma}_{\nu}$ 
	and suppose that $\beta\in \tilde{F}$ and $q(\beta)<0$.
	Then $\mu^{-1}_{\beta}(\nu)^{\text{dif}}\neq 
	\emptyset$.
\end{cor}
\begin{proof}
	Let us take $\nu'$ as in Lemma \ref{slide}
	and show that 
	$\mu^{-1}_{\beta}(\nu')^{\text{dif}}\neq 
	\emptyset$.
	Let us note that $\sum_{j=1}^{m_{i}}\beta_{[i,j]}
	\ge 1$ for all $i\in I_{\text{irr}}$ since $\mathrm{supp}(\beta)$
	is connected.
		Then Proposition 
	\ref{canonicaldecomposition} shows that the subset $Z$ of 
	$\mathrm{Rep}(\mathsf{Q},\beta)$ consisting of 
	all indecomposable representations is a dense subset.
	Thus $Z\cap \mathrm{Rep}(\mathsf{Q},\beta)^{\text{det}}\neq 
	\emptyset$. Then Theorem 3.3 in \cite{C1} shows that 
	$\mu^{-1}_{\beta}(\nu')^{\text{det}}\neq \emptyset$.
	Moreover Theorem 1.2 in \cite{C1} says that the set of all
	irreducible representations in 
	the irreducible topological set $\mu^{-1}_{\beta}(\nu')$ is a dense 
	subset. Here we note that $\beta\in \Sigma_{\nu'}$.
	Thus the non-empty open subset 
	$\nu^{-1}_{\beta}(\nu')^{\text{det}}$
	contains a irreducible representation $x$.
	Thus $\mu^{-1}_{\beta}(\nu')^{\text{dif}}\neq \emptyset$
	which shows that $\mu^{-1}_{\beta}(\nu)^{\text{dif}}\neq \emptyset$ 
	by Lemma \ref{birational}.
	\end{proof}
\subsection{Tame case}
Let us consider the remaining case $q(\beta)=0$ for  
$\beta\in \widetilde{F}$.
The proof depends on a classification of Dynkin diagrams of 
$\mathrm{supp}(\beta)$. 
To recall the classification, first let us introduce the shape
of $\beta\in \mathcal{L}$.
\begin{df}[shape]\normalfont 
	Fix a Kac-Moody root lattice $L=\bigoplus_{i\in I}\mathbb{Z}\alpha_{i}$ 
	and $\alpha=\sum_{i\in I}m_{i}\alpha_{i}\in L$.
For the Dynkin diagram of the support of $\alpha$, we attach each coefficient $m_{i}$ of $\alpha$ to the vertex corresponding 
to $\alpha_{i}$, then  we obtain the diagram with the coefficients, which
we call the {\em shape} of $\alpha$.
\end{df}
For example, if $\alpha=m_{1}\alpha_{i_1}+m_{2}\alpha_{i_2}
+m_{3}\alpha_{i_3}\in L$ with the diagram of the support
$
\begin{xy}
\ar@{-} *++!U{\alpha_{i_1}}*\cir<4pt>{};(7,0) *++!U{\alpha_{i_2}}*\cir<4pt>{}="A",
\ar@{-} "A"; (14,0) *++!U{\alpha_{i_3}}*\cir<4pt>{}
\end{xy}$,
the diagram with coefficients is 
$
\begin{xy}
    \ar@{-} *++!D{m_{1}}*++!U{\alpha_{i_1}}*\cir<4pt>{}
      ;(7,0) *++!D{m_{2}}*++!U{\alpha_{i_2}}*\cir<4pt>{}="A",
    \ar@{-} "A"; (14,0) *++!D{m_{3}}*++!U{\alpha_{i_3}}*\cir<4pt>{}
\end{xy}
$.

By using this we define  shapes of 
elements in $\mathcal{L}$ as follows.
\begin{df}\normalfont
	For $\beta\in \mathcal{L}$,
	the {\em shape} of $\beta$ is 
	the set of shapes of 
	elements in $\Xi_{\beta}^{-1}(\beta)\subset \mathcal{M}_{\beta}$.
\end{df}
\begin{exa}\label{exmp}\normalfont
	For example, suppose $p=1$, $m_{0}=m_{1}=2$, $e_{[i,j]}=1\ (i=0,1$ and $j=1,2)$, $d_0(1,2)=d_1(1,2)=0$.
	Consider $\beta=\epsilon_{[0,1]}+\epsilon_{[0,2]}+\epsilon_{[1,1]}+
	\epsilon_{[1,2]}$.
Then the shape of $\beta$ is 
$
\begin{xy}
    \ar@{=}(0,-4)*+!R+!D{1-a}*\cir<4pt>{};(7,3) *++!D{1-a}*\cir<4pt>{},
    \ar@{=}(0,3)*+!R+!D{a}*\cir<4pt>{};(7,-4) *+!L+!D{a}*\cir<4pt>{}
\end{xy}
\quad(a\in \mathbb{Z}),
$
where we simply denote $\{x_{a}\mid a\in \mathbb{Z}\}$ by $x_{a}\ (a\in \mathbb{Z})$. 

Suppose $p=0$, $m_{0}=4$, $d_0(i,j)=1$ for $1\le i<j\le 4$ and $e_{0,\nu}=1$
for $1\le\nu\le 4$. Then
the shape of $\beta=\sum_{\nu=1}^{4}\epsilon_{[0,\nu]}$ is  
$
\begin{xy}
    \ar@{-}(0,-4)*++!R{1}*\cir<4pt>{}="A";(7,3) *++!L{1}*\cir<4pt>{}="B",
    \ar@{-}(0,3)*++!R{1}*\cir<4pt>{}="C";(7,-4) *++!L{1}*\cir<4pt>{}="D",
    \ar@{-} "A";"C",
    \ar@{-} "A";"D",
    \ar@{-} "B";"C",
    \ar@{-} "B";"D"
\end{xy}.
$
\end{exa}

Then a classification of shapes of $\alpha\in \tilde{F}$ is know as follows.
\begin{thm}[
	Theorem 9 in \cite{HO}]\label{classifyHO}
	For $\beta\in \tilde{F}$ with $q(\beta)=0$, there exists a positive integer 
	$m$ and the shape of $\beta$ is  one of the following.
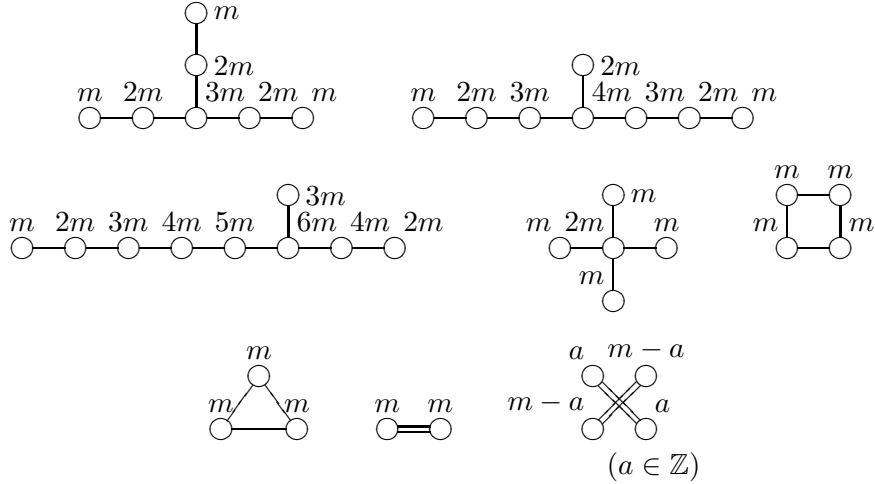
\begin{gather*}
    \begin{xy}
        *++!D{m}*\cir<4pt>{}="A";
        \ar@{-} "A"; (7,0) *++!D{2m}*{}*\cir<4pt>{}="B",
        \ar@{-} "B";(14,0) *+!L+!D{3m}*\cir<4pt>{}="C",
        \ar@{-} "C";(21,0) *+!L+!D{2m}*\cir<4pt>{}="D",
        \ar@{-} "D";(28,0) *+!L+!D{m}*\cir<4pt>{}
        \ar@{-} "C";(14,7) *++!L{2m}*\cir<4pt>{}="E",
        \ar@{-} "E";(14,14) *++!L{m}*\cir<4pt>{}
    \end{xy}
    \quad
    \begin{xy}
        *++!D{m}*\cir<4pt>{}="A";
        \ar@{-} "A"; (7,0) *++!D{2m}*\cir<4pt>{}="B",
        \ar@{-} "B";(14,0) *++!D{3m}*\cir<4pt>{}="C",
        \ar@{-} "C";(21,0) *+!L+!D{4m}*\cir<4pt>{}="D",
        \ar@{-} "D";(28,0) *+!L+!D{3m}*\cir<4pt>{}="E",
        \ar@{-} "D";(21,7) *++!L{2m}*\cir<4pt>{},
        \ar@{-} "E";(35,0) *+!L+!D{2m}*\cir<4pt>{}="F",
        \ar@{-} "F";(42,0) *+!L+!D{m}*\cir<4pt>{}
    \end{xy}\quad\\
    \begin{xy}
        *++!D{2m}*\cir<4pt>{}="A";
        \ar@{-} "A"; (7,0) *++!D{3m}*\cir<4pt>{}="B",
        \ar@{-} "B";(14,0) *++!D{4m}*\cir<4pt>{}="C",
        \ar@{-} "C";(21,0) *++!D{5m}*\cir<4pt>{}="D",
        \ar@{-} "D";(28,0) *+!L+!D{6m}*\cir<4pt>{}="E",
        \ar@{-} "E";(35,0) *+!L+!D{4m}*\cir<4pt>{}="F",
        \ar@{-} "F";(42,0) *+!L+!D{2m}*\cir<4pt>{},
        \ar@{-} "A";(-7,0) *++!D{m}*\cir<4pt>{}="F",
        \ar@{-} "E";(28,7) *++!L{3m}*\cir<4pt>{},
    \end{xy}\quad
    \begin{xy}
        \ar@{-} *+!R+!D{2m}*\cir<4pt>{}="A";(7,0) *++!D{m}*\cir<4pt>{},
        \ar@{-} "A"; (0,7) *++!L{m}*\cir<4pt>{},
        \ar@{-} "A"; (-7,0) *+!R+!D{m}*\cir<4pt>{},
        \ar@{-} "A"; (0,-7) *+!R+!D{m}*\cir<4pt>{}
    \end{xy}\quad
    \begin{xy}
        \ar@{-} *+!R+!D{m}*\cir<4pt>{}="A"; (7,0) *+!L+!D{m}*\cir<4pt>{}="B",
        \ar@{-} (0,7) *++!D{m}*\cir<4pt>{}="C";(7,7) *++!D{m}*\cir<4pt>{}="D",
        \ar@{-} "A";"C",
        \ar@{-} "B";"D"
    \end{xy}\\
    \begin{xy}
        \ar@{-}*++!D{m}*\cir<4pt>{}="A";(10,0)*++!D{m}*\cir<4pt>{}="B",
        \ar@{-}"A";(5,7)*++!D{m}*\cir<4pt>{}="C",
        \ar@{-}"B";"C"
    \end{xy}\quad
    \begin{xy}
        \ar@{=}*++!D{m}*\cir<4pt>{};(7,0)*++!D{m}*\cir<4pt>{}
    \end{xy}\ 
    \begin{xy}
        \ar@{=}(0,0)*+!R+!D{m-a}*\cir<4pt>{};(7,7) *++!D{m-a}*\cir<4pt>{},
        \ar@{=}(0,7)*+!R+!D{a}*\cir<4pt>{};(7,0) *+!L+!D{a}*\cir<4pt>{},
        \ar@{} (8,-5) *{(a\in \mathbb{Z})}
    \end{xy}
    \end{gather*}
    Here the last shape corresponds to the latter case in Proposition 
    \ref{connectedsupport}. 
    For the other shapes, the condition $\#\mathrm{Irr}(\beta)\le 1$ holds,
    namely $\Xi^{-1}(\beta)=\{\tilde{\beta}\}$.
\end{thm}
By using this classification, we can show that $\mu^{-1}_{\beta}(\nu)\neq 
\emptyset$ for $\beta\in \widetilde{\Sigma}_{\nu}\cap \tilde{F}$ and 
$q(\beta)=0$.
\begin{thm}\label{existencetame}
	For $\beta\in \widetilde{\Sigma}_{\nu}\cap \tilde{F}$ and 
	$q(\beta)=0$, we have $\mu^{-1}_{\beta}(\nu)^{\text{dif}}\neq \emptyset$.
\end{thm}
\begin{proof}
	Applying $T_{(i,0)}\colon 
	\mathfrak{M}_{\nu}(\mathsf{Q},\beta)^{\text{dif}}
	\rightarrow 
	\mathfrak{M}_{t_{(i,0)(\nu)}}(t_{(i,0)}(\mathsf{Q}),
	t_{(i,0)}(\beta))^{\text{dif}}$ defined in Section \ref{operations} if necessary, 
	we can assume $0\in \mathrm{Irr}(\beta)$ if $\mathrm{Irr}(\beta)\neq 
	\emptyset$.
	Since $\beta\in \widetilde{\Sigma}_{\nu}\cap \tilde{F}$ and 
	$q(\beta)=0$, the coefficients of $\beta$ is indivisible, i.e.,
	have no common divisor 
	other than $1$, namely the integer $m=1$ in Theorem \ref{classifyHO}
	for $\Xi^{-1}(\beta)$.
	Then 
	we obtain that $\beta^{\text{reg}}$ defined in Section \ref{specialexa}
	is the indivisible null root of a Euclidean root system.
	Indeed, if the shape of $\beta$ is one of 
	the first 7 cases in Theorem \ref{classifyHO} which 
	are indivisible null roots of Euclidean root systems. 
	Proposition \ref{firstexa} shows that $\beta^{\text{reg}}$ 
	is an indivisible null
	root.
	The last case in Theorem \ref{classifyHO} corresponds to 
	the second example in Section \ref{specialexa}. Thus 
	$\beta^{\text{reg}}$ is the indivisible null root of 
	the root system of type $A^{(1)}_{3}$.

	Since it is known that indivisible null roots of 
	Euclidean root systems are Schur roots of quivers,
	thus we obtain that $\beta^{\text{reg}}$ is a Schur root.
	As we noted in Remark \ref{reflfunc}, $\beta^{\text{reg}}$ is 
	obtained by the product of admissible reflections from $\beta$.
	This implies that $\beta$ is a Schur root as well.
	Therefore general elements of $\mathrm{Rep}(\mathsf{Q},\beta)$ 
	are indecomposable.
	Then we have the result by the same argument in Corollary 
	\ref{existencewild}.
\end{proof}

Combining Corollary \ref{wildroot} and the proof in Theorem 
\ref{existencetame},
we have the following theorem.
\begin{thm}\label{quasiisroots}
	The $\mathcal{L}$-fundamental set is a subset of the set 
	of positive imaginary roots of $\mathsf{Q}$.
\end{thm}
Also Corollary \ref{existencewild} and Theorem \ref{existencetame} show the 
following theorem.
\begin{thm}\label{nonempty}
	For $\beta\in \widetilde{\Sigma}_{\nu}\cap \tilde{F}$,  
	we have $\mu_{\beta}^{-1}(\nu)^{\text{dif}}\neq 
	\emptyset$.
\end{thm}
\section{Existence of $\mathcal{L}$-irreducible representations}\label{proofofmain}
In this section we shall give a proof of our main theorem.
First we recall some basic properties of real roots of 
$\mathsf{Q}$ in $\mathcal{L}^{+}$.
The following lemma shows that 
$\epsilon_{\mathbf{i}}$
behaves as a simple root in $\widetilde{R}^{+}:=
\Delta^{+}\cap \mathcal{L}^{+}$
for each $\mathbf{i}\in\mathcal{J}$.
\begin{lem}\label{negaposi}
	If we have $s_{\mathbf{i}}(\alpha)\notin 
	\widetilde{R}^{+}$ for some $\mathbf{i}\in \mathcal{J}$ and $\alpha
	\in \widetilde{R}^{+}$, then $\alpha =\epsilon_{\mathbf{i}}$.
\end{lem}
\begin{proof}
	Let us take $\alpha\in \widetilde{R}^{+}$ and $\mathbf{i}=([i,j_{i}])
	\in 
	\mathcal{J}$ as above.
	By Proposition \ref{modifiedreflection}, 
	$s_{\mathbf{i}}(\alpha)=\left(\prod_{i\in I_{\text{irr}}\backslash\{0\}}
	s_{[i,j_{i}]}\right)
	\circ
	s_{[0,j_{0}]}\circ
	\left(\prod_{i\in I_{\text{irr}}\backslash\{0\}}
	s_{[i,j_{i}]}\right)(\alpha)
	$.
	
	Suppose that 
	$\alpha_{1}=\left(\prod_{i\in I_{\text{irr}}\backslash\{0\}}
	s_{[i,j_{i}]}\right)(\alpha)\notin \Delta^{+}$.
	Then there exists $i_{0}\in I_{\text{irr}}\backslash\{0\}$ 
	such that $\alpha_{1}=-\epsilon_{[i_{0},j_{i_{0}}]}$
	since $s_{[i,j_{i}]}$ for $i\in I_{\text{irr}}\backslash\{0\}$
	are commutative.
	This implies that $\alpha=\epsilon_{[i_{0},j_{i_{0}}]}\notin
	\mathcal{L}$ which contradicts to $\alpha\in \widetilde{R}^{+}$.

	Next suppose that 
	$\alpha_{2}=s_{[0,j_{0}]}\circ
	\left(\prod_{i\in I_{\text{irr}}\backslash\{0\}}
	s_{[i,j_{i}]}\right)(\alpha)\notin \Delta^{+}$.
	Then $\alpha_{2}=-\epsilon_{[0,j_{0}]}$ which
	shows that $\alpha=\epsilon_{\mathbf{i}}$ and 
	$s_{\mathbf{i}}(\alpha)=-\epsilon_{\mathbf{i}}$.

	Finally suppose that $\alpha_{2}\in \Delta^{+}$.
	Then $\alpha_{3}=s_{\mathbf{i}}(\alpha) 
	=-\epsilon_{[i_{0},j_{i_{0}}]}$ for some $i_{0}\in I_{\text{irr}}
	\backslash\{0\}$, which shows that 
	$\alpha\notin \widetilde{R}^{+}$ as above.
	This is a contradiction.
\end{proof}

\begin{lem}\label{Katzalgorithm}
	Take $\beta\in \widetilde{R}^{+}$ satisfying
	$\sum_{j=1}^{m_{i}}\beta_{[i,j]}>0
	$ for all $i\in I_{\text{irr}}$.
	Let us suppose that 
	there exists $a\in \mathcal{J}\cup \mathsf{Q}_{0}^{\text{leg}}$ 
	such that 
	$\beta'=s_{a}(\beta)$ satisfies that $\beta'\in \widetilde{R}^{+}$ and 
	$\beta'_{[i,j]}=0$ for all $i\in I_{\text{irr}}$ and 
	$j=1,\ldots,m_{i}$.
	Then there exist $\mathbf{i}=([i,j_{i}])_{0\le i\le p}
	\in \mathcal{J}$, $i_{0}\in \{0,\ldots,p\}$ 
	and $l\in \{1,\ldots,e_{[i_{0},j_{i_{0}}]}-1\}$ such that 
	$a=\mathbf{i}$ and 
	$\beta=\epsilon_{\mathbf{i}}+\epsilon_{[i_{0},j_{i_{0}},1]}+
	\cdots +\epsilon_{[i_{0},j_{i_{0}},l]}$.
	In this case $\beta$ is a real root.
\end{lem}
\begin{proof}
	From the assumption, $\mathrm{supp}(\beta')\subset 
	\mathsf{Q}_{0}^{\text{leg}}$.
	Since $\beta'\in \widetilde{R}^{+}$, there exist $i_{0}\in 
	\{0,\ldots,p\}$, $j^{i_{0}}\in 
	\{1,\ldots,m_{i_{0}}\}$ and $1\le k<l \le m_{i_{0}}$ such that 
	$\beta'=\epsilon_{[i_{0},j^{i_{0}},k]}+\cdots 
	+\epsilon_{[i_{0},j^{i_{0}},l]}$.
	From the assumption, there exists $a\in \mathcal{J}\cup
	\mathsf{Q}_{0}^{\text{leg}}$ such that
	$s_{a}(\beta')=\beta$ and $\sum_{j=1}^{m^{(i)}}\beta_{[i,j]}>0$.
	Thus there exists $\mathbf{i}=([i,j_{i}])\in \mathcal{J}$
	such that $j_{i_{0}}=j^{i_{0}}$ and $a=\mathbf{i}$.
	Also it follows that $k=1$.
\end{proof}
\begin{lem}\label{existsimple}
	Let us take $\beta=\epsilon_{\mathbf{i}}$ or 
	$\beta=\epsilon_{\mathbf{i}}+\epsilon_{i_{0},j_{i_{0}},1}+
	\cdots +\epsilon_{i_{0},j_{i_{0}},l}$
	for some $\mathbf{i}=([i,j_{i}])_{0\le i\le p}
	\in \mathcal{J}$, $i_{0}\in \{0,\ldots,p\}$ 
	and $l\in \{1,\ldots,e_{[i_{0},j_{i_{0}}]}-1\}$.
	Then we have $\mu^{-1}_{\beta}(\nu)^{\text{dif}}\neq \emptyset$ for any 
	$\nu\in \mathbb{C}^{\mathsf{Q}_{0}}$ satisfying 
	$\beta\in\widetilde{\Sigma}_{\nu}$.
\end{lem}
\begin{proof}
	Applying the  operator $T_{(i_{0},0)}$ defined in Section 
	\ref{operations} if necessary, we may assume $i_{0}=0$.
	Let us choose $\gamma_{i}\in \mathbb{C}$ for all $i\in I_{\text{irr}}
	\backslash\{0\}$ such that $\beta\in \Sigma_{\nu'}$ where 
	$\nu':=\nu+\sum_{i\in I_{\text{irr}}\backslash\{0\}}\gamma_{i}z^{(i)}$
	as in Proposition \ref{slide2}.
	Note that that  $\nu'_{[i,j_{i}]}\neq 0$ for all $i\in 
I_{\text{irr}}\backslash\{0\}$. 
Indeed if there exists $i_{1}\in I_{\text{irr}}\backslash\{0\}$
such that $\nu'_{[i_{1},j_{i_{1}}]}=0$, then we have the decomposition
$\beta=\epsilon_{[i_{1},j_{i_{1}}]}+(\beta-\epsilon_{[i_{1},j_{i_{1}}]})$
such that 
$\nu'\cdot \epsilon_{[i_{1},j_{i_{1}}]}=0$, 
$\nu'\cdot (\beta-\epsilon_{[i_{1},j_{i_{1}}]})=0$, 
$\epsilon_{[i_{1},j_{i_{1}}]},(\beta-\epsilon_{[i_{1},j_{i_{1}}]})
\in \Delta^{+}$, and $p(\epsilon_{[i_{1},j_{i_{1}}]})
=p(\beta-\epsilon_{[i_{1},j_{i_{1}}]})=0$.
 This contradicts to the assumption $\beta\in \Sigma_{\nu'}$.
Thus we obtain that any $x\in \mu^{-1}_{\beta}(\nu)$ satisfies 
that 
$\mathrm{det}(x_{\rho^{[0,j]}_{[i,j']}})_{
\substack{1\le j\le m_{0}\\1\le j'\le m_{i}}}=
x_{\rho^{[0,j_{0}]}_{[i_{j_{i}}]}}\neq 0$ 
for all $i \in I_{\text{irr}}\backslash\{0\}$ because 
$x_{\rho^{[0,j_{0}]}_{[i,j_{i}]}}x_{(\rho^{[0,j_{0}]}_{[i,j_{i}]})^{*}}=
\nu_{[i,j_{i}]}\neq 0$.
Therefore $\mu^{-1}_{\beta}(\nu')^{\text{dif}}=\mu^{-1}_{\beta}(\nu')^{\text{irr}}
\neq \emptyset$ since $\beta\in \Sigma_{\nu'}$. 
Applying $\left(\prod_{i\in I_{\text{irr}}\backslash\{0\}}\mathrm{add}^{(i)}_{
\gamma_{i}z}\right)^{-1}$, we obtain $\mu_{\beta}^{-1}(\nu)^{\text{dif}}\neq 
\emptyset$.
\end{proof}

Now let us study the structure of $\widetilde{\Sigma}_{\mu}$.
Let $\sim$ be the smallest equivalence relation on $\mathcal{S}$
with $(\alpha,\lambda)\sim s_{a}(\alpha,\lambda)$ for 
$a\in \mathcal{J}\cup \mathsf{Q}_{0}^{\text{leg}}$ 
whenever $s_{a}$ is 
defined.
We write $\mathbb{N}\widetilde{R}^{+}_{\lambda}$ for the set 
of sums of elements of $\widetilde{R}^{+}_{\lambda}$ including 0.
Then the following statements are obtained by the same arguments in 
\cite{C1}.

\begin{lem}[cf. Lemma 5.1 in \cite{C1}]
	Given any pair $(\alpha,\lambda)$ with $\alpha\in \mathbb{N}
	\widetilde{R}^{+}_{\lambda}$, if $a\in \mathcal{J}\cup 
	\mathsf{Q}_{0}^{\text{leg}}$ with $\lambda_{a}=0$ and 
	$(\alpha,\epsilon_a)>0$, then $\alpha-\epsilon_a\in 
	\mathbb{N}\widetilde{R}^{+}_{\lambda}$.
\end{lem}
\begin{proof}
	The proof is just an analogy of the proof of Lemma 5.1 in \cite{C1}.
	We can write $\alpha=\sum_{t=1}^{r}\gamma^{(t)}$ as
	a sum of positive roots. If any $\gamma^{(t)}$ is equal to
	$\epsilon_{a}$, then we are done.
	Otherwise all $s_{a}(\gamma^{(t)})$ are positive roots 
	by Lemma \ref{negaposi}, so in $\widetilde{R}^{+}_{\lambda}$.
	Thus $s_{a}(\alpha)=\alpha-(\alpha,\epsilon_{a})\epsilon_{a}
	\in \mathbb{N}\widetilde{R}^{+}_{\lambda}$. Then
	adding on a suitable number of copies of $\epsilon_{a}
	\in \widetilde{R}^{+}_{\lambda}$, it follows that $\alpha-
	\epsilon_{a}\in \mathbb{N}\widetilde{R}^{+}_{\lambda}$.
\end{proof}
\begin{lem}[cf. Lemma 5.2 in \cite{C1}]\label{lem0}
	If $(\alpha,\lambda)\sim (\alpha',\lambda')$ then
	\begin{enumerate}
		\item $\alpha\in \widetilde{R}^{+}_{\lambda}$ if and only if 
			$\alpha'\in \widetilde{R}^{+}_{\lambda'}$,
		\item $\alpha\in \mathbb{N}\widetilde{R}^{+}_{\lambda}$
			if and only if 
			$\alpha'\in \mathbb{N}\widetilde{R}^{+}_{\lambda'}$,
		\item $\alpha\in \widetilde{\Sigma}_{\lambda}$
			if and only if 
			$\alpha'\in \widetilde{\Sigma}_{\lambda'}$.
	\end{enumerate}
\end{lem}
\begin{proof}
	Lemma \ref{negaposi} enable us to apply the same argument as 
	in Lemma 5.2 in \cite{C1} to this lemma.
\end{proof}

\begin{lem}[cf. Lemma 5.3 in \cite{C1}]
	Given any pair $(\alpha,\lambda)$ with $\alpha \in\mathbb{N}
	\widetilde{R}^{+}_{\lambda}$, there is an equivalent pair
	$(\alpha',\lambda')$ with the property that 
	$(\alpha',\epsilon_a)\le 0$ whenever $\lambda'_{a}\neq 0$
	for $a\in \mathcal{J}\cup \mathsf{Q}_{0}^{\text{leg}}$.
\end{lem}
\begin{proof}
	This follows form Lemma \ref{lem0} as well as the proof of Lemma 5.3 in 
	\cite{C1}.
\end{proof}
The following Lemmas \ref{lem11}, \ref{lem2} and Theorem \ref{thm1} 
can be shown by the arguments in the proofs of the 
corresponding 
statements in \cite{C1} without any change.

\begin{lem}[cf. Lemma 5.4 in \cite{C1}]\label{lem11}
	Suppose that $0\neq \alpha \in \mathbb{N}\widetilde{R}^{+}_{\lambda}$
	and $(\alpha,\epsilon_{a})\le 0$ for all $a\in \mathcal{J}\cup 
	\mathsf{Q}_{0}^{\text{leg}}$ with $\lambda_{a}\neq 0$.
	If $(\beta,\alpha-\beta)\le -2$ whenever $\beta,\alpha-\beta$ are 
	nonzero and in $\mathbb{N}\widetilde{R}^{+}_{\lambda}$,
	then $\alpha$ is either $\epsilon_{a}$ where $a\in \mathcal{J}
	\cup \mathsf{Q}_{0}^{\text{leg}}$ or in the 
	$\mathcal{L}$-fundamental set.
\end{lem}

\begin{lem}[cf. Lemma 5.5 in \cite{C1}]\label{lem2}
	If $0\neq \alpha \in \mathbb{N}\widetilde{R}^{+}_{\lambda}$
	and $(\beta,\alpha-\beta)\le -2$ whenever $\beta,\alpha-\beta$ are 
	nonzero and in $\widetilde{R}^{+}_{\lambda}$, then 
	$\alpha\in \widetilde{R}^{+}_{\lambda}$.
\end{lem}

\begin{thm}[cf. Theorem 5.6 in \cite{C1}]\label{thm1}
	If $\alpha \in \mathcal{L}^{+}$ then $\alpha\in 
	\widetilde{\Sigma}_{\lambda}$ if and only if 
	$0\neq \alpha \in \mathbb{N}\widetilde{R}^{+}_{\lambda}$
	and $(\beta,\alpha-\beta)\le -2$ whenever $\beta,\alpha-\beta$ 
	are nonzero and in $\mathbb{N}\widetilde{R}^{+}_{\lambda}$.
\end{thm}

Combining these results, we can obtain the following theorem as in \cite{C1}.
\begin{thm}[cf. Theorem 5.8 in \cite{C1}]\label{reduction}
	If $\alpha \in \widetilde{\Sigma}_{\lambda}$ then there is an 
	equivalent pair $(\alpha',\lambda')$ with $\alpha'$ which is 
	either 
	$\epsilon_{a}$ where $a\in \mathcal{J}
	\cup \mathsf{Q}_{0}^{\text{leg}}$
	or in the $\mathcal{L}$-fundamental set.
\end{thm}

\begin{thm}\label{nonemptyness}
	For $\beta\in \widetilde{\Sigma}_{\nu}$,
we have $\mu^{-1}_{\beta}(\nu)^{\text{dif}}\neq 
	\emptyset$.
\end{thm}
\begin{proof}
	By Theorems \ref{nonempty}, \ref{reduction} and 
	Lemma \ref{Katzalgorithm}, we may assume  
	$(\beta,\nu)\sim(\epsilon_{a},\nu')$, $a\in \mathcal{J}
	\cup \mathsf{Q}_{0}^{\text{leg}}$.
	If $a\in \mathcal{J}$, then Lemma \ref{existsimple}
	shows that $\mu^{-1}_{\beta}(\nu)\neq \emptyset$.
	Suppose $a\in \mathsf{Q}_{0}^{\text{leg}}$.
	If $\sum_{j=1}^{m_{i}}\beta_{[i,j]}>0$ for all $i\in I_{\text{irr}}$,
	then Lemma \ref{Katzalgorithm} and Theorem \ref{reduction} imply
	that there exists a sequence $a_{1},\ldots,a_{r}\in 
	\mathcal{J}\cup \mathsf{Q}_{0}^{\text{leg}}$ such
	that $(\beta^{(k)},\nu^{(k)})=s_{a_{k}}s_{a_{k-1}}\cdots
	s_{a_{1}}(\beta,\nu)$ are well defined and 
	$\sum_{j=1}^{m_{i}}\beta^{(k)}_{[i,j]}>0\ (i\in I_{\text{irr}})$
	for all $k=1,\ldots,r$ and moreover 
	$\beta^{(r)}=\epsilon_{\mathbf{i}}+\epsilon_{[i_{0},j_{0},1]}+
	\cdots$ as in Lemma \ref{Katzalgorithm}.
	Then we have $\mu_{\beta}^{-1}(\nu)^{\text{dif}}\neq 
	\emptyset$ by Lemma \ref{existsimple}.
	If $\sum_{j=1}^{m_{i}}\beta_{[i,j]}=0$ for all 
	$i\in I_{\text{irr}}$, then 
	$\mu_{\beta}^{-1}(\nu)^{\text{dif}}=
	\mu_{\beta}^{-1}(\nu)^{\text{irr}}
	\neq 
	\emptyset$.
\end{proof}

Then Theorems \ref{ifnonempty} and \ref{nonemptyness} show our main theorem.
\begin{thm}\label{mainthm}
	We use the same notation as in Section 
	\ref{middleconvolutionandreflection}.
	Let us consider 
	the additive Deligne-Simpson problem
	for $k_{0},\ldots,k_{p}$ and 
	the HTL normal forms $B_{i}\in \mathfrak{g}^{*}_{k_{i}}$
	for $i=0,\ldots,p$. Then the problem has a solution 
	if and only if $\alpha\in \widetilde{\Sigma}_{\lambda}$.
\end{thm}
\begin{cor}\label{maincor}
	Let 
	$\mathbf{B}=(B_{i})_{0\le i\le p}
	\in \prod_{i=0}^{p}\mathfrak{g}_{k_{i}}$
	be the collection of HTL normal forms as above theorem. Then
	$\mathfrak{M}(\mathbf{B})\neq \emptyset$ if and only if 
	$\alpha \in \widetilde{\Sigma}_{\lambda}$.

\end{cor}

\end{document}